\documentclass{amsart}
\usepackage[utf8]{inputenc}
\usepackage{amsmath, amsthm, amssymb, color, mathtools, amsfonts, mathrsfs,hyperref,enumitem, footmisc}

\usepackage{appendix}
\usepackage{wrapfig}

\usepackage[backend=bibtex ,style=alphabetic, doi=false,isbn=false,url=false, maxcitenames = 99, maxbibnames = 99]{biblatex}
\usepackage[noabbrev, capitalise]{cleveref}
\Crefname{appsec}{Appendix}{Appendices}
\crefformat{footnote}{#2\footnotemark[#1]#3}
\numberwithin{equation}{section}

\usepackage{todonotes}
\setuptodonotes{color=green!20, figcolor=white}

\newtheorem{thm}{Theorem}[section]
\newtheorem{lem}[thm]{Lemma}
\newtheorem{prop}[thm]{Proposition}
\newtheorem{cor}[thm]{Corollary}
\theoremstyle{definition}
\newtheorem{ex}[thm]{Example}
\theoremstyle{definition}
\newtheorem{defn}[thm]{Definition}
\theoremstyle{remark}
\newtheorem{rmk}[thm]{Remark}

\newcommand{\SkAlg}[2]{\operatorname{Sk}_{#1}\left(#2\right)}
\newcommand{\SkAlgSt}[1]{\operatorname{Sk}^{\operatorname{st}}_q\left(#1\right)}
\newcommand{\AW}[1]{\operatorname{AW}(#1)}
\newcommand{\V}{\mathcal{V}}
\newcommand{\slt}{\mathfrak{sl}_2}
\newcommand{\Uq}{\mathcal{U}_q(\slt)}
\newcommand{\qtr}{\operatorname{tr}_q(\mathcal{K})}
\newcommand{\rchoose}[2]{\left(\hspace{-0.4em}\left( \begin{matrix} #1 \\ #2 \end{matrix} \right)\hspace{-0.4em}\right)}

\newcommand{\catname}[1]{\normalfont{\textbf{#1}}}
\newcommand{\SkCat}[2]{\catname{SkCat}_{#1}(#2)}
\newcommand{\faktor}[2]{{#1}/{#2}} 

\DeclareMathOperator{\op}{op}
\DeclareMathOperator{\id}{id}
\DeclareMathOperator{\lf}{lf}

\DeclareMathOperator{\ch}{ch}
\DeclareMathOperator{\SL}{SL}

\input{svg-macros}

\allowdisplaybreaks
\title{Higher Rank Askey--Wilson Algebras as Skein Algebras}
\author[J. Cooke]{Juliet Cooke}
\author[A. Lacabanne]{Abel Lacabanne}
\address{J.C.: School of Mathematical Sciences, University Park, Nottingham, NG7 2RD, United Kingdom, \href{https://julietcooke.net/}{www.julietcooke.net}}
\email{Juliet.Cooke@nottingham.ac.uk}
\address{A.L.: Laboratoire de Math{\'e}matiques Blaise Pascal (UMR 6620), Universit{\'e} Clermont Auvergne, Campus Universitaire des C{\'e}zeaux, 3 place Vasarely, 63178 Aubi{\`e}re Cedex, France, \href{http://www.normalesup.org/~lacabanne}{www.normalesup.org/$\sim$lacabanne}}
\email{abel.lacabanne@uca.fr}
\usepackage[en-GB]{datetime2}
\date{\today}

\begin{document}
\maketitle

\begin{abstract}

    In this paper we give a topological interpretation and diagrammatic calculus for the rank $(n-2)$ Askey--Wilson algebra by proving there is an explicit isomorphism with the Kauffman bracket skein algebra of the $(n+1)$-punctured sphere. To do this we consider the Askey-Wilson algebra in the braided tensor product of $n$ copies of either the quantum group $\mathcal{U}_q{(\mathfrak{sl}_2)}$ or the reflection equation algebra. We then use the isomorpism of the Kauffman bracket skein algebra of the $(n+1)$-punctured sphere with the $\mathcal{U}_q{(\mathfrak{sl}_2})$ invariants of the Aleeksev moduli algebra to complete the correspondence. We also find the graded vector space dimension of the $\mathcal{U}_q{(\mathfrak{sl}_2})$ invariants of the Aleeksev moduli algebra and apply this to finding a presentation of the skein algebra of the five-punctured sphere and hence also find a presentation for the rank $2$ Askey--Wilson algebra.
\end{abstract}

\setcounter{tocdepth}{1}
\tableofcontents

\section{Introduction}

The main goal of this paper is to give a topological interpretation of the higher rank Askey--Wilson algebra. More precisely, we will prove that there is an explicit isomorphism
\[\AW{n} \xrightarrow{\sim }\SkAlg{q}{\Sigma_{0,n+1}}\] 
between the rank $(n-2)$ Askey-Wilson algebra $\AW{n}$ and the Kauffman bracket skein algebra $\SkAlg{q}{\Sigma_{0,n+1}}$ of the $(n+1)$-punctured sphere. The Kauffman bracket skein algebra is an invariant of oriented surfaces given by considering framed links in the thickened surface and imposing the following skein relations
\begin{align*}
	\skeindiagram{leftnoarrow} &= q^{\frac{1}{2}}\; \skeindiagram{horizontal} + q^{-\frac{1}{2}} \; \skeindiagram{vertical}, \\ 
	\skeindiagram{circle} &= -q - q^{-1}
  \end{align*}
  which allows one to resolve all crossing and remove trivial links at the cost of a constant. These relations and the diagrammatic calculus based on Jones--Wenzl idempotents \cite{MV94,Lickorish93,KL94book}, make skein algebras an ideal setting for carrying out concrete computations. The Kauffman bracket skein relation can be renormalised to give the famous Jones polynomial and the skein algebra $\SkAlg{q}{\Sigma}$ itself is a quantisation of the $\SL_2$ character variety of the surface \cite{Bullock97,PrzytyckiSikora00,BullockFrohmanKania99}. 

On the other hand, the original Askey--Wilson algebra $\operatorname{Zh}_q(a,b,c,d)$ was first introduced by Zhedanov \cite{Zhedanov91} in 1991 as the algebra of the bispectral operators of the Askey--Wilson polynomials.
Askey--Wilson polynomials are hypergeometric orthogonal polynomials which can be considered as Macdonald polynomials for the affine root system $(C^{\vee}_1, C_1)$ \cite{NS04}. 
Askey--Wilson algebras and polynomials have applications in physics such as to the one-dimensional Asymmetric Simple Exclusion Process (ASEP) statistical mechanics model~\cite{USW04} as well as applications to more algebraic areas such as the theory of Leonard pairs \cite{TV04}.
Askey--Wilson polynomials can also be truncated to give $q$-Racah polynomials which encode the $6j$-symbols or Racah coefficients which occur in angular momentum recoupling when there are three sources of angular momentum.
The Askey--Wilson algebra\footnote{This is the original Askey--Wilson algebra of Zhedanov. In this paper we use the special Askey--Wilson algebra which if we specialise its parametes is isomorphic to the truncated $q$-Onsager algebra quotiented by the Skylanin determinant.} can in turn be considered as a truncated version of the $q$-Onsager algebra; the $q$-Onsager algebra is the reflection equation algebra of the quantum group $\mathcal{U}_q\big(\widehat{\mathfrak{sl}_2}\big)$ of the affine lie group $\widehat{\mathfrak{sl}_2}$ \cite{Terwilliger01,Baseilhac05} and is used in integrable systems such as in the analysis of the XXZ spin chains with non-diagonal boundary conditions \cite{BK05,BB13}.

There are multiple alternative versions of Askey--Wilson algebras. As well as of Zhedanov's original algebra $\operatorname{Zh}_q(a,b,c,d)$ there is $\operatorname{aw(3)}$ in which the parameters $a,b,c,d$ have been replaced by central elements and the universal Askey--Wilson algebra $\Delta_q$ of Terwilliger \cite{Terwilliger11} with a different choice of central elements. We shall use $\AW{3}$ which is a quotient of $\operatorname{aw(3)}$ by a relation involving the quantum Casimir. For a more complete account of the different versions of the Askey--Wilson algebra and their applications see \cite{CFG21} and the references therein. 

  There are multiple possible approaches to generalising the definition of the Askey--Wilson algebra to higher ranks and we shall follow the approach based on relating $\AW{3}$ to quantum groups. Huang  showed that there is an embedding
 \[\AW{3} \xhookrightarrow{} \Uq \otimes \Uq \otimes \Uq\] 
with the generators $\{\,\Lambda_{A} \;\vert\; A \subseteq \{1, 2, 3\} \,\}$ of $\AW{3}$ under the embedding being constructed out the quantum Casimir of $\Uq$ using coproducts and a map $\tau$ \cite{Huang17,CrampeGaboriaudVinetZaimi20}. This definition was then generalised by Post and Walter to $\AW{4}$ \cite{PostWalter17} and by de Clerq et al.\ to $\AW{n}$ which is defined as a subalgebra $\AW{n} \subset \Uq^{\otimes n}$ generated by explicit generators $\{\,\Lambda_A\;\vert\; A \subseteq \{1,\ldots,n\}\,\}$ \cite{DeClercq19,DeBieDeClercqVanDeVijver20}. De Clerq et al.\ also showed that an algebra isomorphic to $\AW{n}$, the rank $(n-2)$ Bannai--Ito algebra, was the symmetric algebra of the $\mathbb{Z}^n_2$ $q$-Dirac--Dunkl model \cite{DeBieDeClercqVanDeVijver20}.

\subsection*{Isomorphism between Askey--Wilson and Skein Algebras}
  In this paper we shall prove
  \begin{thm}
    \label{thm:AWskeiniso}
    There is an isomorphism 
    \[\SkAlg{q}{\Sigma_{0,n+1}} \xrightarrow{\sim} \AW{n}: s_A \mapsto -\Lambda_A\] 
    between the Kauffman bracket skein algebra $\SkAlg{q}{\Sigma_{0,n+1}}$ of the $(n+1)$-punctured sphere and the rank $(n-2)$ Askey-Wilson algebra $\AW{n}$ which sends the\footnote{There are two choices: either the curves always go below or always go above the points they do not include. In this paper we shall choose below.} simple closed curve $s_A$ around the punctures $A$ to the Askey-Wilson generator $\Lambda_A$ with a negative coefficient.
  \end{thm}
        
  \noindent To show the power of this theorem  in \cref{sec:commutator}, we shall use some skein algebraic calculations to obtain an elegant new proof of the theorem of de Clerq \cite[Theorem~3.2]{DeClercq19} which states that $\AW{n}$ satisfies a generalisation of the commutator relations which are used to define the classical Askey--Wilson algebras.
  We also show in \cref{sec:braid_grp} that this isomorphism is compatible with the action of the braid group.

  \cref{thm:AWskeiniso} is a generalisation of the classical result that $\AW{3}$ is isomorphic to the Kauffman bracket skein algebra of the four-punctured sphere. This was proven by showing that $\AW{3}$ is isomorphic to the $(C^{\vee}_1, C_1)$ spherical double affine Hecke algebra (DAHA) \cite{Koornwinder07,Terwilliger13} and by comparing the presentation of the $(C^{\vee}_1, C_1)$ spherical DAHA to the presentation of the Kauffman bracket skein algebra of the four-punctured sphere \cite{,BS18, Cooke18,Hikami19}.
  This approach is not readily generalisable so we will instead prove \cref{thm:AWskeiniso} by chaining together the following three maps 
  \[
    \SkAlg{q}{\Sigma_{0, n+1}} \subseteq \SkAlgSt{\Sigma_{0, n+1}} \xrightarrow{\sim} \mathcal{O}_q(\mathfrak{sl}_2)^{\tilde{\otimes} n} \xrightarrow{\sim} \left( \Uq^{\lf} \right)^{\tilde{\otimes} n} \xhookrightarrow{} \Uq^{\otimes n},
  \]
  which we shall now discuss.
  \subsection*{Braiding the Tensor Product and the Generators of the Askey--Wilson Algebra}
  The last of these maps is a injective homomorphism that unbraids the braided tensor product $\tilde{\otimes}$ of Majid \cite{Majid91,Majid95book} to give the ordinary tensor product. In \cref{sec:braided_tensor_product} we will show that if we consider the Askey--Wilson algebra $\AW{n}$ as a subalgebra of the braided rather than unbraided tensor product of $n$ copies of $\Uq$ we obtain a simpler description of the generators with no $\tau$ map:
\begin{thm}
  \label{thm:braidedgenerators}
  The generator $\Lambda_A$ for $A = \{i_1< \dots< i_k\} \subseteq \{1, \dots, n\}$ as an element of $\left( \Uq^{\lf} \right)^{\tilde{\otimes} n}$ is given by $\underline{\Delta}^k(\Lambda)_{\underline{i}}$ where $\Lambda$ is the quantum Casimir, $\underline{\Delta}$ is the `braided' coproduct, and the subscript $\underline{i} = (i_1, \dots, i_k)$ denotes that $j^{th}$ tensor factor of the coproduct is placed in position $i_j$ of the tensor product and $1$ is placed in all the empty positions. For example,
  \[\Lambda_{134} = \Lambda_{(1)} \otimes 1 \otimes \Lambda_{(2)} \otimes \Lambda_{(3)} \in \left( \Uq^{\lf} \right)^{\tilde{\otimes} 4}\]
  where we are using Sweedler notation for the coproduct. 
\end{thm}

\subsection*{Reflection Equation Algebras and Presentations} The middle map is a Hopf algebra isomorphism and is given componentwise by the Rosso isomorphism from the reflection equation algebra $\mathcal{O}_q(\mathfrak{sl}_2)$ to $\Uq^{\lf}$, the locally finite subalgebra of $\Uq$ \cite{KS97book}. 
The defining relation for the reflection equation algebra is the reflection equation which first arose in integrable systems related to factorisable scattering on a half-line with a reflecting wall and is based on the standard $R$-matrix of $\Uq$ \cite{Kulish96}. The algebra $\mathcal{O}_q(\mathfrak{sl}_2)^{\tilde{\otimes} n}$ is a special case for the $(n+1)$-punctured sphere of the Aleeksev moduli algebra $\mathcal{L}_{\Sigma}$ which is defined combinatorically for more general surfaces with different tensor products depending on how the handles of the handlebody decomposition of the surface interact \cite{Alekseev94,AGS96}.

The subalgebra $\mathcal{L}_{\Sigma}^{\Uq}$ of the Aleeksev moduli algebra which is invariant under the action of $\Uq$ is naturally filtered by degree. 
In \cref{sec:dimensions} we will use the explicit algebraic description for $\mathcal{L}_{\Sigma}$ to compute the Hilbert series of $\mathcal{L}_{\Sigma}^{\Uq}$ which enumerates the vector space dimension of each graded part of the associated graded algebra. 
For a general punctured surface $\Sigma_{g,r}$ of genus $g$ with $r>1$, the graded algebras associated with $\mathcal{L}_{\Sigma_{g,r}}^{\Uq}$ and the skein algebra $\SkAlg{q}{\Sigma_{g,r}}$ punctures are graded isomorphic. 
Also, for the punctured sphere $\Sigma_{0,n+1}$, the graded algebras associated with $\mathcal{L}_{\Sigma_{0, n+1}}^{\Uq}$ and the Askey--Wilson algebra $\AW{n}$ are graded isomorphic. 
Thus, we also obtain Hilbert series for $\SkAlg{q}{\Sigma_{g,r}}$ and $\AW{n}$:

\begin{thm}
  \label{thm:hilbert}
  The Hilbert series of $\mathcal{L}_{\Sigma_{g,r}}^{\Uq}$, $\AW{n}$ and $\SkAlg{q}{\Sigma_{g,r}}$ is
  \[
  h(t) = \frac{(1+t)^{n-2}}{(1-t)^{n}(1-t^2)^{2n-3}}\left(\sum_{k=0}^{n-2}{\binom{n-2}{k}}^2t^{2k} - \sum_{k=0}^{n-3}\binom{n-2}{k}\binom{n-2}{k+1}t^{2k+1}\right)
  \]
  where $n = 2g + r - 1$.
\end{thm}

Despite skein algebras dating back to the 80s, presentations of Kauffman bracket skein algebras are only known for a handful of punctured surfaces: spheres with up to four punctures and tori with up to two punctures \cite{BullockPrzytycki00}. Whilst it is not difficult to find relations between elements of a skein algebra, it is difficult to conclude you have enough relations. In \cref{sec:presentation} we will use this Hilbert series together with a Poincare--Birkhoff--Witt basis to solve this problem and obtain a presentation for the skein algebra of the five-punctured sphere. This also give us a presentation for the rank 2 Askey--Wilson algebra $\AW{4}$ which is the first of the generalised Askey--Wilson algebras and also the case originally considered by Post and Walter. 
\begin{thm}
  \label{thm:presshort}
  A presentation for $\SkAlg{q}{\Sigma_{0,5}} \cong \AW{4}$ is given by the simple loops $s_A$ for $A \subseteq \{1,2,3,4\}$ subject to the generalised Askey--Wilson commutator relations, non-intersecting loops commuting and relations of types (see \cref{thm:presentation_fivepunctures} and \cref{app:1} for full details):
  \begin{gather*}
    \diagramhh{present}{4cubic}{0pt}{0pt}{0.25} \; 
  \diagramhh{present}{cubictriple}{0pt}{0pt}{0.25} \;
  \diagramhh{present}{quadratic}{0pt}{0pt}{0.25} \;
  \diagramhh{present}{2tripleloop}{0pt}{0pt}{0.25} \\
  \diagramhh{present}{cross}{0pt}{0pt}{0.25} \;
  \diagramhh{present}{2triplelink}{0pt}{0pt}{0.25} \; 
  \diagramhh{present}{doubletriplecross}{0pt}{0pt}{0.25}
  \end{gather*}
\end{thm}
In particular, this means that $\AW{n}$ contains many relations which are not derived from the original Askey--Wilson operator relations.

\subsection*{Skein Algebras and Generalisations}
Finally, the leftmost map is a Hopf algebra isomorphism from the stated skein algebra $\SkAlgSt{\Sigma_{0, n+1}}$ to the Aleeksev moduli algebra. Stated skein algebras are an extension of Kauffman bracket skein algebras which allows tangles with end points rather than only closed loops \cite{Le18,CL20}. This extension means that, unlike ordinary skein algebras, stated skein algebras behave well under gluing: they satisfy an excision property. This is crucial in constructing the isomorphism to $\mathcal{L}_{\Sigma_{0,n+1}} = \mathcal{O}_q(\mathfrak{sl}_2)^{\tilde{\otimes} n}$ given in \cite{CL20}. 

The stated skein algebra is a special case for $\Uq$ of a more general construction based on skein categories $\SkCat{\mathcal{V}}{\Sigma}$
\cite{WalkerTQFT,JohnsonFreyd15,Cooke19} and internal skein algebras $\SkAlg{\mathcal{V}}{\Sigma}^{\operatorname{int}}$ 
\cite{GunninghamJordanSafronov19} which generalises to other quantum groups $\mathcal{U}_q(\mathfrak{g})$ or indeed any ribbon category $\mathcal{V}$ 
(for the explicit relation between stated and internal skein algebras see \cite{haioun21}). 
Skein categories are categories whose hom-spaces are vector spaces of ribbon tangles (often with coupons) in the thickened surface with skein relations imposed: the skein relations are determined by the choice of quantum group or ribbon category. 
The skein algebra $\SkAlg{\mathcal{V}}{\Sigma}$ is then simply $\operatorname{Hom}_{\SkCat{\mathcal{V}}{\Sigma}}(\varnothing, \varnothing)$ and the internal skein algebra is $\SkAlg{\mathcal{V}}{\Sigma}^{\operatorname{int}} = \operatorname{Hom}_{\SkCat{\mathcal{V}}{\Sigma}}(\_, \varnothing): \mathcal{V} \to \operatorname{Vect}$. 
These skein categories satisfy excision and can thus be considered as factorisation homology theories \cite{CookeThesis,Cooke19} (and independently when $\mathcal{V}$ is modular by \cite{KT21}). 

For any punctured surface $\Sigma$ and any quantum group $\mathcal{U}_q(\mathfrak{g})$ when $q$ is not a root of unity the internal skein algebra is isomorphic to the Aleeksev moduli algebra and the skein algebra is the $\mathcal{U}_q(\mathfrak{g})$-invariant subalgebra \cite{GunninghamJordanSafronov19} (and \cite{Faitg2020} for the $G = \SL_2$ case without using factorisation homology).
This gives us a generalisation of the left-hand map to an isomorphism
\[\SkAlg{\mathcal{U}_q(\mathfrak{g})}{\Sigma} \to \mathcal{L}_{G, \Sigma}^{\mathcal{U}_q(\mathfrak{g})}\]
for any punctured surface $\Sigma$ and for any quantum group $\mathcal{U}_q(\mathfrak{g})$ assuming $q$ is generic. For a higher genus surface $\mathcal{L}_{G, \Sigma_{g,r}} = \mathcal{O}_q(\mathfrak{g})^{\hat{\otimes} 2g} \tilde{\otimes} \mathcal{O}_q(\mathfrak{g})^{\tilde{\otimes} r-1}$ where $\hat{\otimes}$ is a different tensor product from the Majid braided tensor product. Whilst it is beyond the scope of this paper, the consideration of other gauge groups in particular $\mathfrak{g} = \mathfrak{sl}_n$ using this connection to skein theory may prove fruitful and would be interesting to compare to other generalisations such as the one based of the affine $\hat{\mathfrak{sl}_n}$ $q$-Osanger algebra in \cite{BCP19}.

\subsection*{Summary of Sections}
\begin{description}
  \item[\cref{sec:AWAlgebras}] In this section we define $\AW{n}$ and its set of generators $\Lambda_A$.
  \item[\cref{sec:braided_tensor_product}] In this section we define the braided tensor product, show that the unbraiding map is an injective morphism of algebras and prove \cref{thm:braidedgenerators}.
  \item[\cref{sec:moduli}] In this section we define the reflection equation algebra $\mathcal{O}_q(\SL_2)$, Alekseev moduli algebra $\mathcal{L}_{\Sigma}$ and the map between the reflection equation algebra and the quantum group $\Uq$.
  \item[\cref{sec:skein}]  In this section we define the stated skein algebra and prove \cref{thm:AWskeiniso}.
  \item[\cref{sec:commutator}] In this section we use skein algebras to give a much shorter proof of Theorem 3.2 of \cite{DeClercq19}.
  \item[\cref{sec:braid_grp}] In this section we show that the isomorphism given in \cref{thm:AWskeiniso} is compatible with the action of the braid group.
  \item[\cref{sec:dimensions}] In this section we find the graded vector space dimension of the graded algebra associated to $\mathcal{L}_{\Sigma}^{\Uq}$ (\cref{thm:hilbert}) and consequently of the associated graded algebras of the Askey--Wilson algebra $\AW{n}$ and the skein balgebra $\SkAlg{q}{\Sigma_{g,r}}$ for $r>1$.
  \item[\cref{sec:presentation}] In this section we prove \cref{thm:presshort} by constructing a confluent terminating term rewriting system based on the relations to obtain a linear basis with the same Hilbert series as $\SkAlg{q}{\Sigma_{0,5}}$.
\end{description}

\subsection*{Notation}

For an algebra $A$ over a field, two integers $k<l$ and $\underline{i} = (i_1,\ldots,i_k)$ with $1\leq i_1\leq \cdots \leq i_k \leq l$, we will use an embedding $A^{\otimes k} \rightarrow A^{\otimes l}$. It is defined by $x_1\otimes \cdots x_k \mapsto 1 \otimes \cdots \otimes 1 \otimes x_{1}\otimes 1 \otimes \cdots \otimes 1 \otimes x_{k}  \otimes 1 \otimes \cdots \otimes 1$, where the tensorand $x_{j}$ is at the $i_j$-th position. The image of $x\in A^{\otimes k}$ will be then denoted by $x_{\underline{i}}$.

We will also use the Sweedler notation for coproducts and coaction: if $(C,\;\Delta)$ is a coalgebra and $(M,\;\Delta_M)$ is a right-comodule over $C$, the coproduct of $c\in C$ will be denoted by $\Delta(c) = \sum c_{(1)}\otimes c_{(2)}$ and the coaction of $C$ on $m\in M$ by $\Delta_M(m) = \sum m_{(1)}\otimes m_{(0)}$.

\subsection*{Acknowledgements}
We would like to thank Peter Samuelson for first pointing us towards the work of De Clerq et. al..
We would also like to thank Martina Balagovic, Hendrik De Bie, Hadewijch De Clercq, Matthieu Faitg, Julien Gaboriaud, David Jordan, Pedro Vaz and Thomas Wright for many valuable conversations related to this work.
This work was supported by the F.R.S.-FNRS., the Max Planck Institute for Mathematics Bonn and by a PEPC JCJC grant from INSMI (CNRS).


\section{Askey--Wilson Algebras}
\label{sec:AWAlgebras}

In this section we shall define the Askey--Wilson algebra $\AW{3}$, explain how it can be embedded into three tensor copies of the quantum group $\Uq$ and thus generalised to the higher rank Askey--Wilson algebra $\AW{n}$.

\subsection{Classical Askey--Wilson algebras}
The Askey--Wilson algebra was originally defined by Zhedanov \cite{Zhedanov91} to study Askey--Wilson orthogonal polynomials as the representations of Askey--Wilson algebras can be used to better understand the associated  polynomials. 

\begin{defn}
  The \emph{Zhedanov Askey--Wilson algebra $\operatorname{Zh}_q(a_1,a_2,a_3,a_{123})$} is the algebra over $\mathbb{C}(q)$ with generators $A$, $B$ and $C$ such that 
  \begin{align*}
    A + \left(q^2 - q^{-2}\right)^{-1} [B, C]_q &= \left(q + q^{-1}\right)^{-1} \left( C_1 C_2 + C_3 C_{123} \right) \\
    B + \left(q^2 - q^{-2}\right)^{-1} [C, A]_q &= \left(q + q^{-1}\right)^{-1} \left( C_2 C_3 + C_1 C_{123} \right) \\
    C + \left(q^2 - q^{-2}\right)^{-1} [A, B]_q &= \left(q + q^{-1}\right)^{-1} \left( C_3 C_1 + C_2 C_{123} \right)  
  \end{align*}
  where $[X, Y]_q := qXY - q^{-1}YX$ is the \emph{quantum Lie bracket} and 
  $C_i := q^{a_i} + q^{-a_i}$. 
\end{defn}

If we weaken the relations, so that instead of having equalities we simply require the expressions on the left-hand-side are central in the algebra, we obtain the \emph{universal} Askey--Wilson algebras which was defined by Terwilliger \cite{Terwilliger11}:

\begin{defn}
  The \emph{universal Askey--Wilson algebra} $\Delta_q$ is the algebra over $\mathbb{C}(q)$ with generators $A$, $B$, $C$ such that 
  \[A + \left(q^2 - q^{-2}\right)^{-1} [B, C]_q \quad B + \left(q^2 - q^{-2}\right)^{-1} [C, A]_q \quad C + \left(q^2 - q^{-2}\right)^{-1} [A, B]_q\]
  are central. 
\end{defn}

If we set 
\begin{align*}
  \alpha &= \left(q+q^{-1}\right)\left(A + \left(q^2 - q^{-2}\right)^{-1}[B, C]_q\right), \\
  \beta &= \left(q+q^{-1}\right)\left( B + \left(q^2 - q^{-2}\right)^{-1} [C, A]_q \right) , \\ 
  \gamma &= \left(q+q^{-1}\right)\left(C + \left(q^2 - q^{-2}\right)^{-1} [A, B]_q \right)
\end{align*}
and also define the `Casimir' element
\[\Omega = qABC + q^2 A^2 + q^{-2} B^2 + q^2 C^2 - q A \alpha - q^{-1} B \beta - q C \gamma\]
we also have the following alternative presentation:
\begin{prop}[{\cite[Proposition~2.8]{Terwilliger13}}]
	The \emph{universal Askey--Wilson algebra} $\Delta_q$ is the algebra over $\mathbb{C}(q)$ with generators $A$, $B$, $C$, $\alpha$, $\beta$, $\gamma$ and $\Omega$ such that $\alpha$, $\beta$, $\gamma$ and $\Omega$ are central and
	\begin{align*}
		[A, B]_q &= -\left( q^2 - q^{-2} \right) C + \left( q - q^{-1} \right) \gamma \\
		[B, C]_q &= -\left( q^2 - q^{-2} \right) A + \left( q - q^{-1} \right) \alpha \\
		[C, A]_q &= -\left( q^2 - q^{-2} \right) B + \left( q - q^{-1} \right) \beta \\
		\Omega   &= q ABC + q^2 A^2 + q^{-2} B^2 + q^2 C^2 - q A \alpha - q^{-1} B \beta - q C \gamma
	\end{align*}
\end{prop}

\subsection{Huang's embedding into $\Uq^{ \otimes 3 }$}

Huang \cite{Huang17} showed that the universal Askey--Wilson algebra can be embedded into $\Uq^{ \otimes 3 }$. We first recall the classical definition of $\Uq$ and choose one of its many Hopf algebra structure. We stress that we follow the conventions of \cite{CL20} which are different from \cite{Huang17}.

\begin{defn}
The quantum algebra $\Uq$ is the $\mathbb{C}(q)$-algebra generated by $K$, $K^{-1}$, $E$ and $F$, subject to the following relations:
  \[
    K^{\pm 1}K^{\mp 1} = 1, \quad KE=q^2EK, \quad KF=q^{-2}FK \quad \text{and} \quad [E,F]=\frac{K-K^{-1}}{q-q^{-1}}.
  \]
  
We endow it with a Hopf algebra structure with the following comultiplication $\Delta$, counit $\varepsilon$ and antipode $S$:
  \begin{align*}
    \Delta(K)&= K\otimes K, & \varepsilon(K) &= 1,& S(K) &= K^{-1},\\
    \Delta(E)&= E\otimes K + 1\otimes E, & \varepsilon(E) &= 0,& S(E) &= -EK^{-1},\\
    \Delta(F)&= F\otimes 1 + K^{-1} \otimes F, & \varepsilon(F) &= 0,& S(F) &= -KF.
  \end{align*}
The \emph{quantum Casimir} of $\Uq$ is
\[
	\Lambda := \left(q - q^{-1}\right)^2 FE + q K + q^{-1} K^{-1} = \left(q - q^{-1}\right)^2 EF + q^{-1} K + q K^{-1};
\]
this is a central element.
\end{defn}

The embedding is given by:

\begin{thm}[{\cite[Theorem 4.1 and Theorem 4.8]{Huang17}}]
  Let $\flat: \Delta_{q^{-1}} \rightarrow \Uq \otimes \Uq \otimes \Uq$ be the following map
\begin{align*}
	&A \mapsto \Delta(\Lambda) \otimes 1, \\
	&B \mapsto 1 \otimes \Delta(\Lambda), \\
        &C \mapsto \frac{q^{-1}(\Delta(\Lambda)\otimes 1)(1\otimes \Delta(\Lambda))-q(1\otimes \Delta(\Lambda))(\Delta(\Lambda)\otimes 1)}{q^2-q^{-2}}+\frac{\Lambda \otimes 1 \otimes \Lambda + (1 \otimes \Lambda \otimes 1) \Delta^{(2)} (\Lambda)}{q+q^{-1}},\\
	&\alpha \mapsto \Lambda \otimes \Lambda \otimes 1 + (1 \otimes 1 \otimes \Lambda) \Delta^{(2)} (\Lambda), \\
	&\beta \mapsto 1 \otimes \Lambda \otimes \Lambda + (\Lambda \otimes 1 \otimes 1) \Delta^{(2)} (\Lambda), \\
	&\gamma \mapsto \Lambda \otimes 1 \otimes \Lambda + (1 \otimes \Lambda \otimes 1) \Delta^{(2)} (\Lambda),
\end{align*}
where $\Delta^{(2)} = (\Delta\otimes \id)\circ \Delta$. Then $\flat$ is an injective morphism of algebras.
\end{thm}


\begin{rmk}
  In Huang's original result, the isomorphism $\Delta_{q} \rightarrow \Uq \otimes \Uq \otimes \Uq$ uses the opposite coproduct to $\Delta$. It is related to $\flat$ by the algebra isomorphism $\Delta_q \rightarrow\Delta_{q^{-1}}$ which sends $(A,B,C,\alpha,\beta,\gamma,\delta)$ to $(B,A,C,\beta,\alpha,\gamma)$ and with the algebra automorphism of $\Uq\otimes\Uq\otimes\Uq$ sending $x\otimes y \otimes z$ to $z \otimes y \otimes x$.
\end{rmk}

For the generator $C$, which is not required to generate the algebra, but makes the presentation more symmetric, \textcite{CrampeGaboriaudVinetZaimi20} showed that its image has the simpler expression $C^{\flat} = (1 \otimes \tau_L) \Delta(\Lambda)$ where $\tau_L$ is defined below.
\begin{defn}
  \label{def:coaction_sl2}
  We denote by $\mathcal{I}_L$ the subalgebra of $\Uq$ generated by $E$, $FK$, $K$ and $\Lambda$. It is a left $\Uq$-comodule with action 
  \[
    \tau_L \colon \left\{
      \begin{array}{lcl}
        \mathcal{I}_L & \longrightarrow & \Uq \otimes \mathcal{I}_L \\
        E & \longmapsto & K \otimes E\\
        FK & \longmapsto & K^{-1} \otimes FK -q F \otimes \Lambda + q\left(q+q^{-1}\right)F\otimes K - q^{-1}(q-q^{-1})^2F^2K\otimes E\\
        K & \longmapsto & 1 \otimes K - q^{-1} ( q - q^{-1} )^2 FK \otimes E\\
        \Lambda & \longmapsto &  1 \otimes \Lambda
      \end{array}
    \right.
  \]
\end{defn}


Clearly the image of the map $\flat$ is contained in the subalgebra of $\Uq^{ \otimes 3 }$ generated by the elements
\[
	\Lambda_1 := \Lambda \otimes 1 \otimes 1, \quad \Lambda_2 := 1 \otimes \Lambda \otimes 1, \quad \Lambda_2 := 1 \otimes 1 \otimes \Lambda, \quad
	\Lambda_{123} := \Delta^2(\Lambda),
\]
\[
	\Lambda_{12} := \Delta(\Lambda) \otimes 1, \quad \Lambda_{13} := (1 \otimes \tau_L)\Delta(\Lambda), \quad \Lambda_{23} := 1 \otimes \Delta^2(\Lambda).
\]
It was shown by Huang \cite[Corollary 4.6]{Huang17} that this subalgebra is contained in the centraliser of $\Uq$ in $\Uq^{ \otimes 3 }$
\[
	\mathfrak{C}(\Uq) = \Big\{\, X \in \Uq^{ \otimes 3 } \Big| \big[(\Delta \otimes \id)\Delta(x),\; X\big] = 0, \; \forall x \in \Uq \,\Big\}
\]

	Note that whilst there is an injective $\mathbb{C}$-algebra homomorphism $\Delta_{q^{-1}} \to \mathfrak{C}(\Uq)$ it is not surjective as we do not have $\Lambda_i$ and $\Lambda_{123}$ in the image. We instead need the following centrally extended universal Askey--Wilson algebra:  

  \begin{defn}
    The \emph{Askey--Wilson algebra $\operatorname{AW}(3)$}\footnote{Our Askey--Wilson algebra $\AW{3}$ is the Special Askey--Wilson algebra ${\bf{saw}}(3)$ of \cite{CrampeGaboriaudVinetZaimi20}.} is the algebra over $\mathbb{C}(q)$ with generators $A$, $B$, $C$ and central generators $C_1$, $C_2$, $C_3$, $C_{123}$ such that 
    \begin{gather*}
      A + \left(q^2 - q^{-2}\right)^{-1} [B, C]_q = \left(q + q^{-1}\right)^{-1} \left( C_1 C_2 + C_3 C_{123} \right) \\
      B + \left(q^2 - q^{-2}\right)^{-1} [C, A]_q = \left(q + q^{-1}\right)^{-1} \left( C_2 C_3 + C_1 C_{123} \right) \\
      C + \left(q^2 - q^{-2}\right)^{-1} [A, B]_q = \left(q + q^{-1}\right)^{-1} \left( C_3 C_1 + C_2 C_{123} \right) \\
    \Omega = \left(q + q^{-1} \right)^2 - C_1^2 - C_2^2 - C_3^2 - C_{123}^2 - C_1 C_2 C_3 C_{123}
  \end{gather*}
    where 
    \[\Omega := qABC + q^2 A^2 + q^{-2} B^2 + q^2 C^2 - q A \left( C_1 C_2 + C_3 C_{123} \right) - q^{-1} B \left( C_2 C_3 + C_1 C_{123} \right) - q C \left( C_3 C_1 + C_2 C_{123} \right).\]
  \end{defn}

  
\subsection{Higher rank Askey--Wilson algebras}

The embedding of the universal Askey--Wilson algebra into the centraliser of three copies of the quantum group $\Uq$ makes it possible define a higher rank Askey--Wilson algebra by constructing Casimirs $\Lambda_A$ for $A \subseteq \{1, \dots, n\}$ in the centraliser of $\Uq^{ \otimes n }$  for $n \geq 1$. 
Post and Walter \cite{PostWalter17} generalised to $n = 4$, and  De Clercq, De Bie and Van de Vijer \cite{DeClercq19,DeBieDeClercqVanDeVijver20} gave a definition for general $n$. The following definition is a slight rewriting of the definition of \cite[Definition~2.3]{DeClercq19} with our conventions of coproduct.


\begin{defn}
  \label{def:AW_generators}
  Let $A=\{i_1 < \cdots <i_k\}$ be a non-empty subset of $\left\{1, \dots, n\right\}$ and let $\Lambda$ be the quantum Casimir of $U_q(\slt)$. The \emph{Casimir $\Lambda_A$} is defined by 
  \[
    \Lambda_A = \left(\left(\id^{\otimes (i_k-2)}\otimes \alpha_{i_k-1}\right)\circ\cdots\circ (\id\otimes \alpha_2) \circ \alpha_1 (\Lambda)\right)\otimes 1^{n-i_k},
  \]
  with $\alpha_i=\Delta$ if $i\in A$ and $\alpha_i=\tau_L$ otherwise. We also set $\Lambda_{\emptyset}=q+q^{-1}$ by convention.
\end{defn}

\begin{defn}
  \label{def:AW}
  The \emph{Askey--Wilson algebra of rank $(n-2)$} denoted $\AW{n}$ is the subalgebra of $\Uq^{\otimes n}$ generated by $\Lambda_A$ for all non-empty subsets $A \subseteq \{1, \dots, n\}$. 
\end{defn}

\begin{rmk}
  This algebra is defined over $\mathbb{C}(q)$, but will need to consider filed extension of this algebra, notably to $\mathbb{C}(q^{1/4})$ where $q^{1/4}$ is a fixed fourth root of $q$.
\end{rmk}

De Clercq \cite[Theorem 3.1 and 3.2]{DeClercq19} show that the higher rank Askey--Wilson algebras satisfy a generalised version of the commutator relations which define $\AW{3}$. 
We shall give an alternative and much shorter proof of this result in \cref{sec:commutator}.

\subsection{Changing the coproduct or how to obtain isomorphisms}
\label{sec:change_conventions}

We show now explain how the change of conventions for the coproduct from those used in \cite{Huang17} to those used in this paper affect the definition of the Askey--Wilson algebra $\AW{n}$. Let us denote by $\widetilde{\mathrm{AW}}(n)$ the Askey--Wilson algebra of \cite{DeClercq19,DeBieDeClercqVanDeVijver20} obtained from the coproduct $\Delta^{\op}$. Explicitly, $\widetilde{\mathrm{AW}}(n)$ is the subalgebra of $\Uq^{\otimes n}$ generated by $\tilde{\Lambda}_A$ for $A$ a non-empty subset of $\{1,\ldots,n\}$ given by
\[
  \tilde{\Lambda}_A = 1^{\otimes i_1-1}\otimes\left(\left(\alpha_{i_1+1}\otimes\id^{\otimes (n-i_1-1)}\right)\circ\cdots\circ (\alpha_{n-1}\otimes \id) \circ \alpha_n (\Lambda)\right)
\]
for $A=\{i_1< \ldots <i_k\}$, with $\alpha_i = \Delta^{\op}$ if $i\in A$ and $\alpha_i = \tau_L^{\op}$ otherwise. 

The following isomorphism is the higher rank version of the algebra isomorphism $\Delta_q \rightarrow \Delta_{q^{-1}}$ of \cite[Lemma 2.11]{Terwilliger13} which sends $(A,B,C,\alpha,\beta,\gamma)$ to $(B,A,C,\beta,\alpha,\gamma)$.

\begin{prop}
  The algebra isomorphism $\Uq^{\otimes n} \rightarrow \Uq^{\otimes n}$ given by $x_1\otimes \cdots \otimes x_n \mapsto x_n \otimes \cdots \otimes x_1$ restricts in an algebra isomorphism $\AW{n} \simeq \widetilde{\mathrm{AW}}(n)$. This isomorphism moreover sends $\Lambda_A$ to $\tilde{\Lambda}_{\tilde{A}}$ where $\tilde{A}=\left\{\,n+1-i \ \middle\vert \ i\in A\,\right\}$. 
\end{prop}

\begin{proof}
  This is immediate from the definitions of $\Lambda_A$ and $\tilde{\Lambda}_A$.
\end{proof}

The following anti-isomorphism is the higher rank version of the algebra anti-isomorphism $\Delta_q \rightarrow \Delta_{q^{-1}}$ which sends $(A,B,C,\alpha,\beta,\gamma)$ to $(A,B,C,\alpha,\beta,\gamma)$.

\begin{prop}
  The algebra anti-isomorphism $S^{\otimes n} \colon \Uq^{\otimes n} \rightarrow \Uq^{\otimes n}$ restricts in an algebra anti-isomorphism $\AW{n} \simeq \widetilde{\mathrm{AW}}(n)$. This isomorphism moreover sends $\Lambda_A$ to $\tilde{\Lambda}_{A}$. 
\end{prop}

\begin{proof}
  The isomorphism follows from the fact that $S\otimes S \circ \Delta = \Delta^{\op}$, that $S\otimes S \circ \tau_L = \tau_R\circ S$ and from \cite[Proposition~2.3]{DeClercq19}. Here $\tau_R$ is given in \cite[Definition 2.1]{DeClercq19}.
\end{proof}

Composing the two previous isomorphisms, we obtain the higher rank version of the algebra anti-automorphism of $\AW{n}$ of \cite[Lemma 2.9]{Terwilliger13}.

\begin{prop}
  There exists an algebra anti-automorphism of $\AW{n}$ sending $\Lambda_A$ to $\Lambda_{\tilde{A}}$. 
\end{prop}


\section{Braided Tensor Product of copies of the locally finite part of $\Uq$}
\label{sec:braided_tensor_product}

As explained in the introduction, the isomorphism between the skein algebra of the $(n+1)$-punctured sphere and the Askey--Wilson algebra $\AW{n}$ consists of a sequence of steps. In this section, we describe the Askey--Wilson algebra as the image of an injective morphism called the unbraiding map and we give an explicit form of the preimages of the generators $\Lambda_A$. One of the key ideas is the interpretation of the coaction $\tau$ as the conjugation by the $R$-matrix \cite{CrampeGaboriaudVinetZaimi20} and the fact that the left coideal $\mathcal{I}_L$ is the locally finite part of $\Uq$ for the left adjoint action.

\subsection{Grading and the left adjoint action}
\label{sec:left_adjoint_action}

There exists a $\mathbb{Z}$-grading on $\Uq$ given on the generators by $\lvert E \rvert =1$, $\lvert F \rvert = -1$ and $\lvert K \rvert = 0$. Note that $Kx = q^{2\lvert x \rvert} xK$ for any homogeneous element $x$.

As it is a Hopf algebra, $\Uq$ acts on itself via the left adjoint action given by
\[
  h\rhd x = \sum h_{(1)}xS\left(h_{(2)}\right),
\]
for any $h,x\in \Uq$. Note that $\lvert h\rhd x \rvert = \lvert h\rvert + \lvert x \rvert$.

\begin{defn}
The locally finite elements of $\Uq$ for the left adjoint action are:
\[
  \Uq^{\lf} = \Big\{ \,x \in \Uq  \mathrel{\Big\vert}  \Uq\rhd x \text{ is finite dimensional} \,\Big\}.
\]
\end{defn}

By definition, $\Uq$ then acts on $\Uq^{\lf}$ via the left adjoint action. It is known that $\Uq^{\lf}$ is a subalgebra of $\Uq$ and also a left coideal of $\Uq$, that is $\Delta\left(\Uq^{\lf}\right) \subset \Uq\otimes \Uq^{\lf}$, see \cite[Lemma 3.112]{voigt-yuncken} for example.

The following is a theorem of Joseph--Letzter \cite{JosephLetzter92} in the specific case of $\Uq$. Note that our convention for the coproduct is different from \cite{JosephLetzter92}.

\begin{prop}
  The subalgebra and left coideal $\Uq^{\lf}$ coincides with $\mathcal{I}_L$.
\end{prop}

\subsection{The quasi-R-matrix}
\label{sec:notations}

In this subsection, we introduce the $R$-matrix for $\Uq$ following \cite[Chapter 4]{Lusztig93}. The Hopf algebra $\Uq$ fails to be a quasi-triangular Hopf algebra. This problem is usually overcomed using the notion of quasi-$R$-matrix.

First, we define $\Psi\colon \Uq \otimes \Uq \rightarrow \Uq\otimes \Uq$ as the algebra isomorphism defined on homogeneous elements by $\Psi(x\otimes y) = x K^{-\lvert y\rvert} \otimes K^{-\lvert x \rvert} y =  K^{-\lvert y\rvert}x \otimes yK^{-\lvert x \rvert}$.
\begin{thm}[{\cite[Chapter 4]{Lusztig93}}]
  \label{thm:quasi-R-matrix}
There exists an element $\Theta = \sum_{i\geq 0}\Theta_i$ with $\Theta_i = u_i\otimes v_i$ satisfying the following properties:
\begin{itemize}
\item \label{itm:grad_theta} $u_i \otimes v_i = q^{i(i-1)/2}\displaystyle\frac{ \left(q-q^{-1}\right)^i}{[i]!} E^i \otimes F^i$;
\item for all $x\in \Uq$, we have $\Psi(\Delta^{\op}(x))\Theta = \Theta\Delta(x)$;
\item \label{itm:inv_theta} $\Theta$ is invertible with inverse $\Gamma = \sum_{i\geq 0} \Gamma_i$ with $\Gamma_i = S(u_i)K^i \otimes v_i = u_i \otimes S^{-1}(v_i)K^{-i}$;
\item $\Delta\otimes \id(\Theta) = \Psi_{23}(\Theta_{13})\Theta_{23}$ and $\id\otimes\Delta(\Theta) = \Psi_{12}(\Theta_{13})\Theta_{12}$.
\end{itemize}
\end{thm}

In Sweedler notation, one has that for any homogeneous $x\in \Uq$
\begin{equation}
  \label{eq:coproduct-op-sweedler}
  \sum_{i\geq 0} x_{(2)}K^{-\lvert x_{(1)} \rvert}u_i \otimes x_{(1)}K^{\lvert x_{(2)} \rvert}v_i
  =
  \sum_{i\geq 0} u_i x_{(1)} \otimes v_i x_{(2)}.
\end{equation}

Thanks to the quasi-$R$-matrix, one can endow the category of type $1$ integrable weight modules with a braiding as follows. Let $M$ and $N$ be two integrable weight modules of type $1$. For $m\in M$ of weight $\mu$ and $n\in N$ of weight $\nu$, we define
\[
  c_{M,N}(m \otimes n) = \sum_{i \geq 0}q^{(\mu + 2i)(\nu - 2i)/2}(v_i n)\otimes (u_i m).
\]
Note that in general the coefficients appearing in the formula above are in $\mathbb{Q}(q^{1/2})$. This then defines a map $c_{M,N}\colon M\otimes N \rightarrow N\otimes M$ which is a braiding for the category of type 1 integrable weight modules (see \cite[Section 3]{Jantzen96} for more details).

\subsection{Braided tensor product}
\label{sec:braided_tensor}

For two algebras in a braided monoidal category, Majid had defined their braided tensor product \cite[Lemma 9.2.12]{Majid95book}. In the spectific case of a category of modules of a quasi-triangular Hopf algebra, the multiplication in the braided tensor product is explicitely given using the $R$-matrix \cite[Corollary 9.2.13]{Majid95book}. In our situation, since we are working in the category of integrable weight modules over $\Uq$, we can not define the braided tensor product of $\Uq$ with itself since it is not an integrable module for the left adjoint action. One way to remedy to this problem is to work with the locally finite part $\Uq^{\lf}$ so that the definition of the multiplication is still valid.

\begin{prop}
  For any homogeneous elements $g,h,x$ and $y$ of $\Uq^{\lf}$, we define the braided product $\cdot$ by
  \begin{equation}
    (g\otimes h)\cdot (x\otimes y) = \sum_{i\geq 0}q^{2(\lvert h \rvert + i)(\lvert x \rvert - i)}g(v_i\rhd x) \otimes (u_i \rhd h)y.
  \end{equation}

  This defines an associative multiplication on the vector space $\Uq^{\lf}\otimes \,\Uq^{\lf}$. The resulting algebra is the \emph{braided tensor product} and will be denoted by $\Uq^{\lf}\tilde{\otimes}\,\Uq^{\lf}$.
\end{prop}

We can also inductively define a braided product on $\Uq^{\otimes n}$ and we denote the resulting algebra by $\left(\Uq^{\lf}\right)^{\tilde{\otimes} n}$. One can check that this multiplication is compatible with the adjoint action: for any $h\in \Uq$ and $x,y \in \left(\Uq^{\lf}\right)^{\tilde{\otimes} n}$, one has
\[
  h\rhd (x\cdot y) = \sum (h_{(1)}\rhd x)\cdot (h_{(2)}\rhd y).
\]

This multiplication endows $\left(\Uq^{\lf}\right)^{\tilde{\otimes} n}$ with an algebra structure in the category of integrable $\Uq$-modules. We can also endow this algebra with a bialgebra structure in the category of $\Uq$-modules.
Define $\underline{\Delta}$ following \cite[Section 3]{Majid91}:
\[
  \underline{\Delta}(x) = \sum_{i\geq 0} x_{(1)}S\left(K^{i+\lvert x_{(2)} \rvert}v_i\right)\otimes u_i\rhd x_{(2)}
\]
One may check that this endows $\Uq^{\lf}$ with a bialgebra structure in the category of left $\Uq$-modules, that is the multiplication $\cdot$ and the comultiplication $\underline{\Delta}$ satisfy the usual axioms of a bialgebra and that they are all compatible with the adjoint action. For example,
\[
  \underline{\Delta}(h\rhd x) = h\rhd \underline{\Delta}(x).
\]

Furthermore, we inductively define the iterated comultiplication $\Delta^{(n)}$ by
\[
  \Delta^{(0)} = \id,\quad\text{and}\quad \Delta^{(n+1)} = \left(\id\otimes\Delta^{(n)}\right)\circ \Delta,
\]
and define the iterated comultiplication $\underline{\Delta}^{(n)}$ similarly. Note that $\Delta^{(1)}=\Delta$ and $\underline{\Delta}^{(1)}=\underline{\Delta}$. One may inductively show that
\begin{align}
  \underline{\Delta}^{(n-1)}(x)
  &= \sum_{i\geq 0} x_{(1)}S\left(K^{i+\lvert x_{(2)} \rvert}v_i\right)\otimes  \underline{\Delta}^{(k-2)}\left(u_i\rhd x_{(2)}\right)\nonumber \\
  &= \sum_{i\geq 0} x_{(1)}S\left(K^{i+\lvert x_{(2)} \rvert}v_i\right)\otimes  u_i\rhd \underline{\Delta}^{(n-2)}\left(x_{(2)}\right) \label{eq:iterated_underlined_comultiplication_qgr}.
\end{align}

We have the following explicit formulas for $\underline{\Delta}(x)$ when $x$ is a generator of $\left(\Uq^{\lf}\right)^{\tilde{\otimes} n}$, each of which follows from direct computation.

\begin{lem}\label{lem:formula_bar_coproduct}
  We have
  \begin{align*}
    \underline{\Delta}(E) & = E\otimes K + q^{-1}\left(\Lambda-q^{-1}K\right)\otimes E,\\
    \underline{\Delta}(KF) & = K\otimes KF + q^{-1}KF\otimes \left(C-q^{-1}K\right),\\
    \underline{\Delta}(K) & = K\otimes K + q^{-1}\left(q-q^{-1}\right)^2KF\otimes E,\\
    \underline{\Delta}(\Lambda) & = q^{-1}C\otimes C - q^{-2}C\otimes K - q^{-2} K \otimes C + q^{-2}\left(q+q^{-1}\right)K\otimes K \\
    &\quad+ \left(q-q^{-1}\right)^2\left(E\otimes KF + q^{-2}KF\otimes E\right).
  \end{align*}
\end{lem}

\subsection{Unbraiding the braided tensor product}
\label{sec:unbraiding}

Following \cite{FioreSteinackerWess03}, we unbraid the tensor product into the usual tensor product. Since we work with the quantum group $\Uq$, we need to use the quasi-$R$-matrix $\Theta$ and the locally finite part.

\begin{prop}
  \label{prop:one-step_unbraiding_qgr}
  For any $k\geq 1$, the map $\varphi_n\colon \left(\Uq^{\lf}\right)^{\tilde{\otimes} n} \rightarrow \Uq\otimes \left(\Uq^{\lf}\right)^{\tilde{\otimes} (n-1)}$ defined by
  \[
    \varphi_n(a\otimes b) = \sum_{i\geq 0} aK^{i+\lvert b \rvert}v_i\otimes u_i\rhd b,
  \]
  for any homogeneous elements $a\in \Uq^{\lf}$ and $b\in \left(\Uq^{\lf}\right)^{\tilde{\otimes} (n-1)}$, is an injective morphism of algebras.
\end{prop}

\begin{proof}
  First, it is easy to check that $a\otimes b\mapsto \left( \sum_{i\geq 0} \left(aS(v_i)K^{-(i+\lvert b \rvert)} \right) \otimes u_i\rhd b \right)$ is a left inverse of $\varphi_n$.

  We now check that $\varphi_n$ is a morphism of algebra. Let $a,c\in \Uq^{\lf}$ and $b,d\in \left(\Uq^{\lf}\right)^{\tilde{\otimes} (n-1)}$ and we compute:
  \begin{align*}
    \varphi_n\big((a\otimes b)\cdot (c\otimes d)\big)
    &=\sum_{i\geq 0}q^{2(\lvert b \rvert + i)(\lvert c\rvert - i)} \varphi_n\big(a(v_i\rhd c)\otimes ((u_i\rhd b)\cdot d)\big)\\
    &=\sum_{i,j\geq 0}q^{2(\lvert b \rvert + i)(\lvert c\rvert - i)} a(v_i \rhd c)K^{i+j+\lvert b \rvert + \lvert d \rvert}v_j \otimes u_j \rhd ((u_i\rhd b)\cdot d).
  \end{align*}
  Now, we use the definition of the left adjoint action, its compatibility with the braided product, together with $\Delta\otimes\id(\Theta) = \Psi_{23}(\Theta_{13})\Theta_{23}$ and $\id\otimes \Delta(\Theta) = \Psi_{12}(\Theta_{13})\Theta_{12}$ to find that
  \begin{multline*}
    \varphi_n\big((a\otimes b)\cdot (c\otimes d)\big) =\\
    \sum_{i,j,k,l \geq 0} q^{2(\lvert b \rvert + i + k)(\lvert c\rvert - i - k)} aK^{-i}v_k c S(v_i)K^{i+j+k+l+\lvert b\rvert +\lvert d \rvert} v_j v_l \otimes (u_j u_i u_k \rhd b)\cdot \left(K^j u_l \rhd d\right).
  \end{multline*}
  Since for any homogeneous $x\in \Uq$ we have $Kx = q^{2\lvert x\rvert} x K$, we find that
  \[
    \varphi_n\big((a\otimes b)\cdot (c\otimes d)\big) =
    \sum_{i,j,k,l \geq 0}  aK^{k+\lvert b \rvert}v_k c S(v_i)K^{j} v_j K^{l+\lvert d \rvert}v_l \otimes (u_j u_i u_k \rhd b)\cdot (u_l \rhd d).
  \]
  Finally, since $\sum_{i,j\geq 0}u_j u_i \otimes S(v_i)K^{j} v_j = 1\otimes 1$, we have
  \[
     \varphi_n\big((a\otimes b)\cdot (c\otimes d)\big) =
                                                 \sum_{k,l\geq 0}  \left(aK^{k+\lvert b \rvert}v_k\right)\left(c K^{l+\lvert d \rvert}v_l\right) \otimes (u_k \rhd b)\cdot (u_l \rhd d)\\
    = \varphi_n(a\otimes b)\cdot \varphi_n(c\otimes d),
  \]
  and $\varphi_n$ is a morphism of algebra.
\end{proof}

By composing the one-step unbraiding maps $\varphi_n$ from \cref{prop:one-step_unbraiding_qgr}, one can inductively define an unbraiding map $\gamma_n\colon \left(\Uq^{\lf}\right)^{\tilde{\otimes} n}\rightarrow \Uq^{\otimes n}$. We define $\gamma_1$ to be the injection of the locally finite part $\Uq^{\lf}$ inside the whole quantum group $\Uq$ and for $n\geq 1$ we define $\gamma_{n+1}=(\id\otimes \gamma_n)\circ \varphi_{n+1}$.

\begin{cor}
  The map $\gamma_n\colon \left(\Uq^{\lf}\right)^{\tilde{\otimes} n}\rightarrow \Uq^{\otimes n}$ is an injective morphism of algebras.
\end{cor}

We end this subsection by explaining how the adjoint action $\rhd$ of $\Uq$ on $\left(\Uq^{\lf}\right)^{\tilde{\otimes}n}$ behaves under the unbraiding map $\gamma_n$.

\begin{prop}
  \label{prop:unbraiding_action}
  For any $x \in \left(\Uq^{\lf}\right)^{\tilde{\otimes} n}$ and $h \in \Uq$ we have
  \[
    \gamma_n(h\rhd x) = \sum \Delta^{(n-1)}\left(h_{(1)}\right)\gamma_n(x)\Delta^{(n-1)}\left(S\left(h_{(2)}\right)\right).
  \]
\end{prop}

\begin{proof}
  Once again, we proceed by induction on $n$ with case $n=1$ being true by the definition of the left adjoint action. We start by computing $\varphi_{n+1}(h\rhd(a\otimes b))$ for any $h\in \Uq$, $a\in \Uq^{\lf}$ and $b\in \left(\Uq^{\lf}\right)^{\tilde{\otimes}n}$:
  \begin{align*}
    \varphi_{n+1}\big(h\rhd(a\otimes b)\big)
    &= \sum \varphi_{n+1}\left(h_{(1)}\rhd a\otimes h_{(2)}\rhd b\right)\\
    &= \sum_{i\geq 0} \left(h_{(1)}\rhd a\right)K^{i+\lvert h_{(2)} \rvert + \lvert b \rvert} v_i \otimes \left(u_i h_{(2)}\right)\rhd b\\
    &= \sum_{i\geq 0} h_{(1)}aS\left(h_{(2)}\right)K^{i+\lvert h_{(3)} \rvert + \lvert b \rvert} v_i h_{(4)}S\left(h_{(5)}\right) \otimes \left(u_i h_{(3)}\right)\rhd b\\
    &= \sum_{i\geq 0} h_{(1)}aS\left(h_{(2)}\right)K^{i} v_i K_{\lvert h_{(3)}\rvert} h_{(4)}K_{\lvert b \rvert}S\left(h_{(5)}\right) \otimes \left( u_i K^{-i} h_{(3)} K_{\lvert h_{(4)}\rvert} \right)\rhd b,
  \end{align*}
  the second to last equality following from the counit and antipode axioms. Now, we use the fact that the quasi-$R$-matrix and $\Psi$ intertwine the coproduct $\Delta$ and its opposite $\Delta^{\op}$ (see \cref{eq:coproduct-op-sweedler}) to obtain
  \begin{align*}
    \varphi_{n+1}\big(h\rhd(a\otimes b)\big)
    &= \sum_{i\geq 0} h_{(1)}aS\left(h_{(2)}\right)h_{(3)}K^i v_i K_{\lvert b \rvert}S\left(h_{(5)}\right) \otimes \left(h_{(4)}u_i K^{-i}\right)\rhd b\\
    &= \sum_{i\geq 0} h_{(1)}aK^i v_i K_{\lvert b \rvert}S\left(h_{(3)}\right) \otimes \left(h_{(2)}u_i K^{-i}\right)\rhd b\\
    &= \sum_{i\geq 0} h_{(1)}aK^{i+\lvert b \rvert} v_i S\left(h_{(3)}\right) \otimes \left(h_{(2)}u_i\right)\rhd b,
  \end{align*}
  where the second to last equality follows once again from the counit and antipode axioms. We now use the inductive definition of the unbraiding map and the induction hypothesis and we obtain
  \begin{align*}
    \gamma_{n+1}\big(h\rhd(a\otimes b)\big)
    &= \sum_{i\geq 0} h_{(1)}aK^{i+\lvert b \rvert} v_i S\left(h_{(3)}\right) \otimes \gamma_n\left(\left(h_{(2)}u_i\right)\rhd b\right)\\
    &= \sum_{i\geq 0} h_{(1)}aK^{i+\lvert b \rvert} v_i S\left(h_{(4)}\right) \otimes \Delta^{(n-1)}\left(h_{(2)}\right)\gamma_n(u_i\rhd b)\Delta^{(n-1)}\left(S\left(h_{(3)}\right)\right)\\
    &= \sum \Delta^{(n)}\left(h_{(1)}\right)\left(\sum_{i\geq 0} aK^{i+\lvert b \rvert} v_i \otimes \gamma_n(u_i\rhd b)\right)\Delta^{(n)}\left(S\left(h_{(2)}\right)\right)\\
    &= \sum \Delta^{(n)}\left(h_{(1)}\right)\gamma_{n+1}(a\otimes b)\Delta^{(n)}\left(S\left(h_{(2)}\right)\right),
  \end{align*}
  as expected.
\end{proof}

This last proposition leads to the following interesting corollary concerning centralizers.

\begin{cor}
  \label{cor:invariant-centralizer}
  Invariants elements of $\left(\Uq^{\lf}\right)^{\tilde{\otimes}n}$ under the left adjoint action are sent by the unbraiding map $\gamma_n$ into the centralizer of $\Uq$ in $\Uq^{\otimes n}$. Explicitly, any $x\in \left(\Uq^{\lf}\right)^{\tilde{\otimes}n}$ stable under the left adjoint action commutes with $\Delta^{(n-1)}(h)$ for any $h\in \Uq$.
\end{cor}

\begin{proof}
  Suppose that $z\in \Uq^{\otimes n}$ satisfies the following: for any $h\in \Uq$ we have
  \[
    \sum \Delta^{(n-1)}\left(h_{(1)}\right)z\Delta^{(n-1)}\left(S\left(h_{(2)}\right)\right )=\varepsilon(h)z.
  \]
  Then for such an element $z$ we have
  \begin{align*}
    \Delta^{(n-1)}(h)z
    &= \sum \Delta^{(n-1)}\left(h_{(1)}\right)z\Delta^{(n-1)}\left(S\left(h_{(2)}\right)\right)\Delta^{(n-1)}\left(h_{(3)}\right)\\
    &= \sum \varepsilon\left(h_{(1)}\right)z\Delta^{(n-1)}\left(h_{(2)}\right)\\
    &= z\Delta^{(n-1)}(h),
  \end{align*}
  the first equality following from the counit and antipode axiom, the second one from the hypothesis on $z$ and the last one from the counit axiom. The statement follows then immediately from \cref{prop:unbraiding_action}
\end{proof}

\subsection{Inserting units and unbraiding}
\label{sec:inserting_1}




Our next aim is to understand the images of the elements $\underline{\Delta}^{(k-1)}(x)_{\underline{i}}$ under the unbraiding map $\gamma_n$. If $x=\Lambda$ is the Casimir element of $\Uq$, we will show that $\underline{\Delta}^{(k-1)}(x)_{\underline{i}}$ coincides with the generator $\Lambda_A$ of the Askey--Wilson algebra, for $A=\{i_1,\ldots,i_k\}$. As a corollary we obtain that the Askey--Wilson algebra $\AW{n}$ lies inside the centralizer of $\Uq$ in $\Uq^{\otimes n}$, generalizing a result of Huang \cite[Corollary 4.6]{Huang17}.  

\begin{lem}
  For any $x\in \Uq^{\lf}$ we have
  \begin{equation}
    \label{eq:tau_adj_action}
  \tau_L(x) = \sum_{i\geq 0} K^{i+\lvert x \rvert}v_i\otimes u_i\rhd x.
\end{equation}
\end{lem}

Only a finite number of terms of the sum are non-zero since $x$ lies in the locally finite part of $\Uq$ for the left adjoint action.

\begin{proof}
  Let us denote by $\tau'_L(x)$ the right-hand side of \cref{eq:tau_adj_action}. One easily checks that $\tau'_L(xy) = \tau'_L(x)\tau'_L(y)$ using the relation $\Delta\otimes \id(\Theta) = \Psi_{23}(\Theta_{13})\Theta_{23}$. It remains to check that $\tau_L$ and $\tau'_L$ coincide on the generators of $\Uq$, which is a straightforward calculation.
\end{proof}

As noted in \cite[Lemma 4.1]{CrampeGaboriaudVinetZaimi20}, the map $\tau_L$ is also given by conjugation by the $R$-matrix.

\begin{lem}
  \label{lem:tau-conj}
  For any $x\in \Uq^{\lf}$ we have $\tau_L(x) = \Psi^{-1}\left(\Theta_{21}(1\otimes x)\Theta_{21}^{-1}\right)$.
\end{lem}

\begin{proof}
  We start with the conjugation by $\Theta_{21}$:
  \[
    \Theta_{21}(1\otimes x)\Theta_{21}^{-1}
    = \sum_{i,j\geq 0} v_i v_j \otimes u_ix S(u_j) K^j\\
    = \sum_{i,j\geq 0} v_i v_j \otimes u_i x S\left(K^i u_j\right) K^{i+j}\\
    = \sum_{i\geq 0} v_i\otimes (u_i\rhd x)  K^{i},
  \]
  the last equality following from $\Delta\otimes (\Theta) = \Psi_{23}(\Theta_{13})\Theta_{23}$ and from the definition of the adjoint action. Applying $\Psi^{-1}$ concludes the proof.
\end{proof}

\begin{cor}
  The map $\tau_L$ is a morphism of algebras and defines a left coation of $\Uq$ on $\Uq^{\lf}$, that is $(\Delta\otimes \id)\circ \tau_L = (\id\otimes \tau_L)\circ \tau_L$ and $(\varepsilon\otimes \id)\circ \tau_L=\id$.
\end{cor}

\begin{proof}
  The map $\tau_L$ is clearly a morphism of algebra thanks to \cref{lem:tau-conj}. It is also not difficult to check that it is a left coaction using the relation $\id\otimes \Delta (\Psi) = \Psi_{12}(\Theta_{13})\Theta_{12}$.
\end{proof}

Finally, we obtain an explicit formula for the image of $\underline{\Delta}^{(k-1)}(x)_{\underline{i}}$ under unbraiding.

\begin{prop}
  \label{prop:inserting_units}
  Let $n\in \mathbb{Z}_{>0}$, $k\in \mathbb{Z}_{\geq0}$ with $k\leq n$, $\underline{i}=(i_1,\ldots,i_k)$ with $1\leq i_1< \cdots < i_k \leq n$ and $x\in \Uq^{\lf}$. Then
  \[
    \gamma_n\left(\underline{\Delta}^{(k-1)}(x)_{\underline{i}}\right) = \left(\id^{\otimes (i_k-2)}\otimes \alpha_{i_k-1}\right)\circ\cdots\circ (\id\otimes \alpha_2) \circ \alpha_1 (x)\otimes 1^{n-i_k},
  \]
  with $\alpha_i=\Delta$ if $i\in\{i_1,\ldots,i_k\}$ and $\alpha_i=\tau_L$ otherwise.
\end{prop}

\begin{proof}
  Once again, we proceed on induction on $n$, and there is nothing to prove if $n=1$. Let us suppose the result proven for some $n\in\mathbb{Z}_{>0}$ and any $k,\underline{i}$ and $x$ as above. Let $k\leq n+1$, $\underline{i}=(i_1,\ldots,i_k)$ with $1\leq i_1 < \cdots < i_k\leq n+1$ and $x\in \Uq^{\lf}$.

  We first suppose that $i_{1}\neq 1$ so that $\underline{\Delta}^{(k-1)}(x)_{\underline{i}}=1\otimes \underline{\Delta}^{(k-1)}(x)_{\underline{i}-1}$, where $\underline{i}-1 = (i_1-1,\ldots,i_k-1)$. Then
  \[
    \varphi_{n+1}\left(\underline{\Delta}^{(k-1)}(x)_{\underline{i}}\right)=\sum_{i\geq 0}K^{i+\lvert x \rvert}v_i\otimes u_i\rhd \underline{\Delta}^{(k-1)}(x)_{\underline{i}-1} = \sum_{i\geq 0}K^{i+\lvert x \rvert}v_i\otimes \underline{\Delta}^{(k-1)}(u_i\rhd x)_{\underline{i}-1}.
  \]
  By the induction hypothesis, we have
  \[
    \gamma_n\left(\underline{\Delta}^{(k-1)}(u_i\rhd x)_{\underline{i}-1}\right) = \left(\id^{\otimes (i_k-3)}\otimes \beta_{i_k-2}\right)\circ\cdots\circ (\id\otimes \beta_2) \circ \beta_1 (u_i\rhd x)\otimes 1^{n-i_k+1},
  \]
  with $\beta_i=\Delta$ if $i\in\{i_1-1,\ldots,i_k-1\}$ and $\beta_i=\tau_L$ otherwise. Therefore,
  \begin{align*}
    \gamma_{n+1}\left(\underline{\Delta}^{(k-1)}(x)_{\underline{i}}\right)
    &= (\id\otimes \gamma_{n})\circ \varphi_{n+1}\left(\underline{\Delta}^{(k-1)}(x)_{\underline{i}}\right)\\
    &= \sum_{i\geq 0} K^{i+\lvert x \rvert} u_i\otimes \gamma_n\left(\underline{\Delta}^{(k-1)}(v_i\rhd x)_{\underline{i}-1}\right)\\
    &= \left(\id^{\otimes (i_k-2)}\otimes \beta_{i_k-2}\right)\circ\cdots\circ (\id\otimes \beta_2) \circ (\id\otimes\beta_1) \left(\sum_{i\geq 0} K^{i+\lvert x \rvert} u_i\otimes (v_i\rhd x)\right)\otimes 1^{n-i_k+1}\\
    &=\left(\id^{\otimes (i_k-2)}\otimes \beta_{i_k-2}\right)\circ\cdots\circ (\id\otimes \beta_2) \circ (\id\otimes\beta_1) \circ \tau_L(x)\otimes 1^{n-i_k+1},
  \end{align*}
  which has the desired form if we set $\alpha_i=\beta_{i-1}$ for $2\leq i \leq i_{k}-1$ and $\alpha_1=\tau_L$.

  We now suppose that $i_1=1$ so that
  \[
    \underline{\Delta}^{(k-1)}(x)_{\underline{i}}=\sum_{i\geq 0} x_{(1)}S\left(K^{i+\lvert x_{(2)}\rvert}v_i\right)\otimes u_i\rhd \underline{\Delta}^{(k-2)}\left(x_{(2)}\right)_{\underline{i}^{-}-1},
  \]
where $\underline{i}^{-}-1 = (i_2-1,\ldots,i_k-1)$. Then we have
  \begin{align*}
    \varphi_{n+1}\left(\underline{\Delta}^{(k-1)}(x)_{\underline{i}}\right)
    &=\sum_{i,j\geq 0}x_{(1)}S\left(K_{\mu+\lvert x_{(2)}\rvert}v_i\right)K^{i+j+\lvert x_{(2)}\rvert}v_j\otimes \left((u_ju_i)\rhd \underline{\Delta}^{(k-2)}\left(x_{(2)}\right)_{\underline{i}^{-}-1}\right)\\
    &=\sum_{i,j\geq 0}x_{(1)}S(v_i)K^{j}v_j\otimes \left((u_ju_i)\rhd \underline{\Delta}^{(k-2)}\left(x_{(2)}\right)_{\underline{i}^{-}-1}\right)\\
    &= \sum x_{(1)}\otimes \underline{\Delta}^{(k-2)}(x)_{\underline{i}^{-}-1},
  \end{align*}
  the last equality following once again from \cref{itm:inv_theta}. As in the case $i_1\neq 1$, we use the inductive definition of $\gamma_{n+1}$, apply the induction hypothesis and obtain the expected result, the only difference being that $\alpha_1=\Delta$ instead $\alpha_1=\tau_L$.
\end{proof}

\begin{rmk}
  In the specific case of $k=n$, the previous proposition states that $\gamma_n\circ\underline{\Delta}^{(n-1)} = \Delta^{(n-1)}\circ \gamma_1$.
\end{rmk}

With $x=\Lambda$ we recover the generator $\Lambda_A$ of $\AW{n}$ for a suitable $A\subset\{1,\ldots,n\}$.

\begin{cor}
  \label{cor:AW_generators}
  The element $\underline{\Delta}^{(k-1)}(\Lambda)_{\underline{i}}$ is sent to the element $\Lambda_A$ of the Askey--Wilson algebra under the unbraiding map $\gamma_n$, where $A=\{i_1,\ldots,i_k\}$.
\end{cor}

Since the Casimir element $\Lambda$ is central, combining \cref{prop:inserting_units} with \cref{cor:invariant-centralizer}, we obtain the following corollary:

\begin{cor}
  \label{cor:AW_centralizer}
  The Askey--Wilson algebra $\AW{n}$ is a subalgebra of the centralizer of $\Uq$ in $\Uq^{\otimes n}$.
\end{cor}


\section{Moduli algebras}
\label{sec:moduli}

In order to state the results of Costantino--Lê for (stated) skein algebras of punctured spheres, we introduce the quantum coordinate algebra and the reflection equation algebra.

\subsection{Quantum coordinate algebra} 
\label{sec:qcoor}

We begin by defining the quantum coordinate algebra, otherwise known as the Faddeev-Reshetikhin-Takhtajan algebra \cite{FRT89}, for the following $4\times 4$ $R$-matrix:
\[
  R =
  \begin{pmatrix}
    q & 0 & 0 & 0 \\
    0 & 1 & q-q^{-1} & 0\\
    0 & 0 & 1 & 0\\
    0 & 0 & 0 & q
  \end{pmatrix}.
\]

\begin{defn}
  The quantum coordinate algebra $\SL_q(2)$ is the algebra with generators the entries of the matrix $U = \begin{pmatrix}
    u^+_+ & u^+_-\\
    u^-_+ & u^-_-
  \end{pmatrix}$ with the relations given by
  \begin{align*}
   u^+_-u^+_+ &= q u^+_+u^+_-, & u^-_-u^+_- &= qu^+_-u^-_-, & u^-_+u^+_+ &= q u^+_+u^-_+, & u^-_-u^-_+ &= qu^-_+u^-_-, & u^+_-u^-_+ &= u^-_+u^+_-,
  \end{align*}
  and
  \begin{align*}
    u^+_+u^-_--q^{-1}u^+_-u^-_+ &= 1, & u^-_-u^+_+-qu^-_+u^+_- &= 1.
  \end{align*}
\end{defn} 

We endow this algebra with a structure of a Hopf algebra with coproduct $\Delta$, counit $\varepsilon$ and antipode $S$ given by
\begin{multline*}
  \Delta 
  \begin{pmatrix} 
    u_+^+ & u_-^+\\ 
    u_+^- & u_-^-
  \end{pmatrix} 
  =
  \begin{pmatrix} 
    u_+^+\otimes u_+^+ + u^+_-\otimes u^-_+ & u^+_-\otimes u^-_- + u^+_+\otimes u^+_-\\
    u^-_+\otimes u^+_+ + u^-_-\otimes u^-_+ & u^-_-\otimes u^-_- + u^-_+\otimes u^+_-
  \end{pmatrix},
  \quad
  \varepsilon
  \begin{pmatrix} 
    u_+^+ & u_-^+\\ 
    u_+^- & u_-^-
  \end{pmatrix}
  = 
  \begin{pmatrix} 
    1 & 0\\
    0 & 1
  \end{pmatrix},
  \\ \quad\text{and}\quad
  S 
  \begin{pmatrix} 
    u_+^+ & u_-^+\\ 
    u_+^- & u_-^-
  \end{pmatrix} 
  = 
  \begin{pmatrix} 
    u^-_- & -qu^+_-\\
    -q^{-1}u^-_+ & u^+_+
  \end{pmatrix}.
\end{multline*}

There exists a non-degenerate Hopf pairing $\langle\cdot,\cdot\rangle$ between $\SL_q(2)$ and $\Uq$ given by
\[
\left\langle K, \begin{pmatrix}u^+_+ & u^+_-\\u^-_+ & u^-_-\end{pmatrix}\right\rangle = \begin{pmatrix}q & 0\\ 0 & q^{-1}\end{pmatrix},\quad
\left\langle E, \begin{pmatrix}u^+_+ & u^+_-\\u^-_+ & u^-_-\end{pmatrix}\right\rangle = \begin{pmatrix}0 & 1\\ 0 & 0\end{pmatrix}\quad\text{and}\quad
\left\langle F, \begin{pmatrix}u^+_+ & u^+_-\\u^-_+ & u^-_-\end{pmatrix}\right\rangle = \begin{pmatrix}0 & 0\\ 1 & 0\end{pmatrix}.
\]
This pairing implies that any right comodule over $\SL_q(2)$ can be turned into a left module over $\Uq$: if $M$ is a right $\SL_q(2)$-comodule with coaction $\Delta_M$ then the action of $\Uq$ is
\[
 x\cdot m = \sum \langle x,\;m_{(0)}\rangle\, m_{(1)},
\]
with the right coaction on $M$ being $\Delta_M(m) = \sum m_{(1)}\otimes m_{(0)}$. 

\begin{ex}
  It can be checked that the left $\Uq$-module structure arising from the right $\SL_q(2)$-comodule structure on $\SL_q(2)$ is given on the generators by
\begin{align*}
  K\cdot u^+_+ &= q u^+_+, & K \cdot u^+_- &= q^{-1}u^+_-,& K\cdot u^-_+ &= qu^-_+,& K\cdot u^-_- &=q^{-1}u^-_-,\\
  E\cdot u^+_+ &= 0, & E \cdot u^+_- &= u^+_+,& E\cdot u^-_+ &= 0,& E\cdot u^-_- &=u^-_+,\\
  F\cdot u^+_+ &= u^+_-, & F \cdot u^+_- &= 0,& F\cdot u^-_+ &= u^-_-,& F\cdot u^-_- &=0.
\end{align*}
\end{ex}

The Hopf algebra $\SL_q(2)$ is cobraided with the co-$R$-matrix $\rho$ being given by $R$:
\[
  \rho\left(\begin{pmatrix} 
    u_+^+ & u_-^+\\ 
    u_+^- & u_-^-
  \end{pmatrix}\otimes\begin{pmatrix} 
    u_+^+ & u_-^+\\ 
    u_+^- & u_-^-
  \end{pmatrix}\right) = R.
\]
Therefore, the category of right $\SL_q(2)$-modules is braided: given two $\SL_q(2)$-modules $V$ and $W$ with respective coactions $\Delta_V$ and $\Delta_W$, the co-$R$-matrix $\rho$ defines an isomorphism $c_{V,W}^{\rho} \colon V\otimes W \rightarrow W\otimes V$ given by
\[
  c^{\rho}_{V,W}(v\otimes w) = \sum \rho\left(v_{(0)}\otimes w_{(0)}\right) w_{(1)}\otimes v_{(1)}
\]
where $\Delta_V(v) = \sum v_{(1)}\otimes v_{(0)}\in V\otimes\SL_q(2)$ and $\Delta_W(v) = \sum w_{(1)}\otimes w_{(0)}\in W\otimes \SL_q(2)$.

The following proposition is classical and compares the two braidings $c_{V,W}^\rho$ and $c_{V,W}$, where we see $V$ and $W$ as left $\Uq$-modules through the pairing.

\begin{prop}
  \label{prop:cobr=br}
  Let $V$ and $W$ be two right comodules over $\SL_q(2)$ that we also equip with the structure of left modules over $U_q(\slt)$. If the obtained modules are weight modules with a locally nilpotent action of $E$ and $F$. Then we have $c_{V,W}=c^\rho_{V,W}$ as maps from $V\otimes W$ to $W\otimes V$.
\end{prop}

\begin{proof}
  By definition, we have
  \[
    c^\rho_{V,W}(v\otimes w) = \sum \rho\left(v_{(0)}\otimes w_{(0)}\right) w_{(1)}\otimes v_{(1)},
  \]
  and
  \[
    c_{V,W}(v\otimes w) = \sum_{n\geq 0} q^{(\lvert v \rvert + 2n)(\lvert w \rvert -2n)/2}q^{n(n-1)/2}\frac{(q-q^{-1})^n}{[n]!}\left\langle E^n,\;v_{(0)}\right\rangle \left\langle F^n ,\; w_{(0)}\right\rangle w_{(1)}\otimes v_{(1)}
  \]
  for all weight vectors $v \in V$ and $w \in W$. Note that since $\Delta_V(K\cdot v) = (1\otimes K) \cdot \Delta_V(v)$, and similarly for $W$, we have
  \[
    c^R_{V,W}(v\otimes w) = \sum_{n\geq 0} q^{(\lvert v_{(0)} \rvert + 2n)(\lvert w_{(0)} \rvert -2n)/2}q^{n(n-1)/2}\frac{(q-q^{-1})^n}{[n]!}\left\langle E^n,\;v_{(0)}\right\rangle \left\langle F^n ,\; w_{(0)}\right\rangle w_{(1)}\otimes v_{(1)}.
  \]

  Therefore, the claim is proved once we have checked that
  \begin{equation}
    \rho(x\otimes y) = \sum_{n\geq 0} q^{(\lvert x \rvert + 2n)(\lvert y \rvert -2n)/2}q^{n(n-1)/2}\frac{(q-q^{-1})^n}{[n]!}\langle E^n,\;x\rangle \langle F^n ,\; y\rangle.\label{eq:rho'}
  \end{equation}

  Denote by $\rho'(x\otimes y)$ the right-hand side of \cref{eq:rho'}. Using the various properties of the $R$-matrix and of the co-$R$-matrix, it is a routine calculation to check that
  \begin{align*}
  \rho'(xy\otimes z) &= \sum \rho'\left(x \otimes z_{(1)}\right)\rho'\left(y \otimes z_{(2)}\right), & \rho'(1\otimes z) &= \varepsilon(z),\\
  \rho'(x\otimes yz) &= \sum \rho'\left(x_{(1)} \otimes z\right)\rho'\left(x_{(2)} \otimes y\right), & \rho'(x\otimes 1) &= \varepsilon(x),
\end{align*}
  for all $x,y$ and $z\in \SL_q(2)$. Therefore, it remains to check that $\rho(x\otimes y)=\rho'(x\otimes y)$ for $x$ and $y$ in the set $\{u^\pm_{\pm}\}$, which is an easy and omitted calculation. 
\end{proof}

\subsection{Reflection equation algebras}
\label{sec:REA}
We now introduce another algebra which is not isomorphic to the quantum coordinate algebra but is twist-equivalent to it \cite{Donin}.
\begin{defn}
  The reflection equation algebra $\mathcal{O}_q(\SL_2)$ is the algebra with generators the entries of the matrix $\mathcal{K}=
  \begin{pmatrix}
    k_+^+ & k_-^+\\
    k_+^- & k_-^-
  \end{pmatrix}$ which satisfy the following:
  \begin{enumerate}
  \item the quantum determinant relation: $k_+^+k_-^--qk_-^+k_+^- = 1$,
  \item the reflection equation:  $R_{21}(\mathcal{K} \otimes I) R (\mathcal{K} = I_2\otimes \mathcal{K}) = ( I\otimes \mathcal{K}) R_{21} (\mathcal{K} \otimes I) R$.
  \end{enumerate}
\end{defn}

From the theory of $L$-operators, see for example \cite[Proposition 3.116]{voigt-yuncken}, we may deduce an algebra isomorphism between $\mathcal{O}_q(\SL_2)$ and $\Uq^{\lf}$. It is explicitly given by
\[
\begin{pmatrix} 
   k_+^+ & k_-^+\\
   k_+^- & k_-^- 
\end{pmatrix}
\mapsto 
\begin{pmatrix}
   q^{-1}\left(\Lambda-q^{-1}K\right) & q^{-1}\left(q-q^{-1}\right)E \\
   \left(q-q^{-1}\right)KF & K 
\end{pmatrix}.
\]

The bar coproduct $\underline{\Delta}$ on $\Uq^{\lf}$ defined in \cref{sec:braided_tensor} has a much nicer expression through this isomorphism: it becomes the usual matrix coproduct given by
\[
\mathcal{K} \mapsto
\begin{pmatrix}
  k_+^+\otimes k_+^+ + k^+_-\otimes k^-_+ & k^+_-\otimes k^-_- + k^+_+\otimes k^+_-\\
  k^-_+\otimes k^+_+ + k^-_-\otimes k^-_+ & k^-_-\otimes k^-_- + k^-_+\otimes k^+_-
\end{pmatrix}.
\]

Finally, we note also that the element $\qtr:=qk^+_++q^{-1}k^-_-$ is sent to the Casimir element $\Lambda$.

\subsection{Alekseev moduli algebras}
\label{sec:alekseev}

We end this section by introducing Alekseev moduli algebras, also known as quantum loop algebras, in the case of $\Uq$. These algebras are attached to the punctured surfaces $\Sigma_{g,r}$. In the particular case of the punctured sphere $\Sigma_{0,n+1}$, we recover the braided tensor power $\mathcal{O}_q(\SL_2)^{\tilde{\otimes} n}\simeq \left(\Uq^{\lf}\right)^{\tilde{\otimes}n}$ already encountered in \cref{sec:unbraiding}. In general, we also need the notion of the elliptic double $\mathcal{D}_q(\SL_2)$. As a vector space, the elliptic double is isomorphic to $\mathcal{O}_q(\SL_2)^{\otimes 2}$, see \cite{Brochier-Jordan} for further details on the multiplication.

\begin{defn}
  The Alekseev moduli algebra associated with the punctured surface $\Sigma_{g,r}$ is
  \[
  \mathcal{L}_{\Sigma_{g,r}} = \mathcal{D}_q(\SL_2)^{\tilde{\otimes} g}\tilde{\otimes}\mathcal{O}_q(\SL_2)^{\tilde{\otimes} r-1},
  \]
  where $\tilde{\otimes}$ still denotes the braided tensor product.
\end{defn}

It should be noted that, as a vector space, the algebra $\mathcal{L}_{\Sigma_{g,r}}(\SL_2)$ is isomorphic to $\mathcal{O}_q(\SL_2)^{\otimes 2g+r-1}$. Indeed, $\mathcal{D}_q(\SL_2)$ is defined as a vector space by $\mathcal{O}_q(\SL_2)^{\otimes 2}$ and the multiplication is twisted using the $R$-matrix. We will also write $\mathcal{L}_{\Sigma_{g,r}} = \mathcal{O}_q(\SL_2)^{\hat{\otimes} 2g}\tilde{\otimes}\mathcal{O}_q(\SL_2)^{\tilde{\otimes} r-1}$, where $\hat{\otimes}$ emphasises that the multiplication on $\mathcal{O}_q(\SL_2)^{\otimes 2g}$ is not the trivial one, but is such that $\mathcal{O}_q(\SL_2)^{\hat{\otimes} 2g}=\mathcal{D}_q(\SL_2)^{\tilde{\otimes} g}$ as an algebra.


\section{Skein Algebras}
\label{sec:skein}

In this section, we shall prove that the Askey--Wilson algebra $\AW{n}$ is isomorphic to the Kauffman bracket skein algebra $\SkAlg{q}{\Sigma_{0,n+1}}$ of the $(n+1)$-punctured sphere and explicitly match up the generators on both sides of this correspondence. To do this we shall consider the Kauffman bracket skein algebra as a subalgebra of the stated skein algebra. Costantino and L\^{e} proved that the stated skein algebra of a $(n+1)$-punctured sphere is isomorphic to the braided tensor product of $n$ copies of the reflection equation algebra $\mathcal{O}_q(\SL_2)$ \cite{CL20}. The reflection equation algebra is in turn isomorphic to $\Uq^{\lf}$ (\cref{sec:REA}). Finally, we shall use the description of the Askey--Wilson algebra $\AW{n}$ as a subalgebra of the braided tensor product of $n$ copies of $\Uq^{\lf}$ which we developed in \cref{sec:unbraiding} to obtain our result. 

\subsection{Kauffman Bracket Skein Algebras and Stated Skein Algebra}
\label{sec:skeindefn}
The Kauffman bracket skein algebra is based on the Kauffman bracket:
\begin{defn}
Let \(L\) be a link without contractible components (but including the empty link). The \emph{Kauffman bracket polynomial} \(\langle L \rangle\) in the variable \(q\) is defined by the following local \emph{skein relations}:
\begin{align}
  \skeindiagram{leftnoarrow} &= q^{\frac{1}{2}}\; \skeindiagram{horizontal} + q^{-\frac{1}{2}} \; \skeindiagram{vertical}, \label{skeinrel_cross} \\ 
  \skeindiagram{circle} &= -q - q^{-1}. \label{skeinrel_loop}
\end{align}
\end{defn}
\noindent These diagrams represent links with blackboard framing which are identical outside a 3-dimensional disc and are as depicted inside the disc. It is an invariant of framed links and it can be `renormalised' to give the Jones polynomial. The Kauffman bracket can also be used to define an invariant of \(3\)-manifolds:
\begin{defn}
Let \(M\) be a smooth 3-manifold, \(R\) be a commutative ring with identity and \(q\) be an invertible element of \(R\). The \emph{Kauffman bracket skein module} \(\SkAlg{q}{M}\) is the \(R\)-module of all formal linear combinations of links, modulo the Kauffman bracket skein relations pictured above.
\end{defn}

\begin{rmk} For the remainder of the paper we will use the coefficient ring \(R := \mathbb{C}\left(q^{1/4}\right)\).
\end{rmk}

For a surface $\Sigma$, we define its \emph{skein algebra} $\SkAlg{q}{\Sigma}$ to be the skein module $\SkAlg{q}{\Sigma\times [0,1]}$ and define multiplication by first stacking the links on top of each other to obtain a link in \(\Sigma \times [0,2]\) and then rescaling the second coordinate to obtain \(\Sigma \times [0,1]\) again. Usually links are drawn by projecting onto $\Sigma$. In this case the multiplication $XY$ is obtained by drawing $Y$ above $X$. 

Typically, this algebra is noncommutative; however, its $q = \pm 1$ specialisation is commutative since in this case the right-hand side of the first skein relation is symmetric with respect to switching the crossing. \citeauthor{Bullock97}, \citeauthor{PrzytyckiSikora00}~\cite{Bullock97, PrzytyckiSikora00} showed that at $q=-1$ the skein algebra $\SkAlg{q}{\Sigma}$ is isomorphic to the ring of functions on the $\SL_2$ character variety of $\Sigma$. \cite{BullockFrohmanKania99} strengthened this statement by showing that the skein algebra is a quantization of the $\SL_2$ character variety of $\Sigma$ with respect to Atiyah--Bott--Goldman Poisson bracket. 

Recall that \(\Sigma_{g,r}\) denotes the compact oriented surface with genus $g$ and $r$ punctures. We will restrict ourselves to punctured surfaces so we assume that $r>1$.
Every punctured surface $\Sigma_{g,r}$ has a handlebody decomposition which is given by attaching $n=2g+r-1$ handles to a disc. For example, the handlebody decomposition of $\Sigma_{0,5}$ is shown in \cref{fig:5punctsphere}. 

\begin{defn}
  Let $\Sigma_{g,n}^{\bullet}$ denote the surface with a choice of marking on its boundary. The marking must be on the disc part of the handlebody decomposition of the surface.
\end{defn}

\begin{figure}[ht]
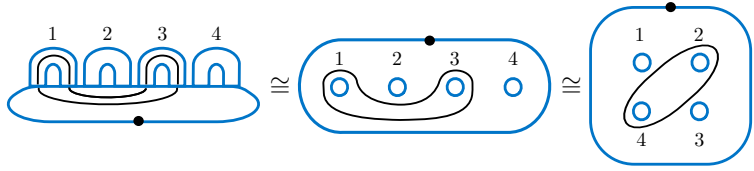

  \centering
  $\diagramhh{loopcalc}{5handle}{12pt}{1pt}{0.25} \cong \diagramhh{loopcalc}{5line}{15pt}{1pt}{0.25} \cong \diagramhh{loopcalc}{5circle}{27pt}{1pt}{0.25}$
  \caption{This figure shows the handlebody decomposition of $\Sigma^{\bullet}_{0,5}$. One of the five punctures is around the outside of the figures and the other four punctures are labeled 1-4. The marking on the boundary is depicted as a black circle and the generating loop $s_{13}$ is also shown.}
  \label{fig:5punctsphere}
\end{figure}

For every subset $A \subseteq \{1, \dots, n\}$ there is a simple closed curve $s_A$ which intersects the handles $A$. These simple curves $s_A$ form a generating set for the skein algebra:

\begin{thm}[\cite{Bullock99}]
  The curves $s_A$ for all non-empty subsets $A \subseteq \{1, \dots, n\}$ generate the skein algebra $\SkAlg{q}{\Sigma_{0,n+1}}$. 
\end{thm}

In order to relate these Kauffman bracket skein algebras to quantum loop algebras, one must consider the skein algebra as a subalgebra of an algebra of skeins which are not all closed loops.

\begin{defn}[\cite{Le18}]
  Let $\Sigma$ be an oriented surface with boundary $\partial \Sigma$. Let $T$ be a tangle in $\Sigma \times [0, 1]$ together with a colouring $\pm$ on each point where $T$ meets the boundary $\partial \Sigma$. The \emph{stated skein algebra} $\SkAlgSt{\Sigma}$ is the \(R\)-module of all formal linear combinations of isotopy classes of such tangles $T$, modulo the Kauffman bracket skein relations (\cref{skeinrel_cross,skeinrel_loop}) and the boundary conditions:
  \begin{gather}
    \label{eq:skeinboundary}
    \diagramhh{skeinrelations}{pp}{12pt}{1pt}{0.25} = \diagramhh{skeinrelations}{mm}{12pt}{1pt}{0.25} = 0 \quad \quad \diagramhh{skeinrelations}{mp}{12pt}{1pt}{0.25} = q^{-\frac{1}{4}} \quad \quad \diagramhh{skeinrelations}{cup}{12pt}{1pt}{0.25} = q^{\frac{1}{4}}\diagramhh{skeinrelations}{pmstraight}{12pt}{1pt}{0.25} - q^{\frac{5}{4}}\diagramhh{skeinrelations}{mpstraight}{12pt}{1pt}{0.25}
  \end{gather}
\end{defn} 

The reason we wish to consider stated skein algebras is that stated skein algebras satisfy excision \cite[Theorem 4.12]{CL20} which in particular means that the stated skein algebra $\SkAlgSt{\Sigma^{\bullet}_{0,n+1}}$ can be constructed out of copies of the simpler skein algebra $\SkAlgSt{\Sigma^{\bullet}_{0,2}}$.

\begin{rmk}
  Stated skein algebras are a special case for $\SL_2$ of \emph{internal skein algebras} which were defined by \citeauthor{GunninghamJordanSafronov19}~\cite{GunninghamJordanSafronov19} based on \emph{skein categories} \cite{Cooke19,JohnsonFreyd15}. Skein categories and thus internal skein algebras are defined for any linear ribbon category $\V$ over any unital commutative ring.
\end{rmk}

\subsection{Isomorphism of Skein Algebra and Askey--Wilson Algebra}

In this subsection we shall combine the results of this paper so far together with the results of Costantino and L\^{e} to obtain an explicit isomorphism between the Kauffman bracket skein algebra of the $(n+1)$-punctured sphere $\SkAlg{q}{\Sigma_{0,n+1}^{\bullet}}$ and the rank $n$ Askey--Wilson algebra $\AW{n}$.

Costantino and Lê \cite[Theorem 3.4]{CL20} show that the quantum coordinate algebra\footnote{Costantino and Lê follow Majid in referring to $\SL_q(2)$ as the quantum coordinate algebra and denote it as $\mathcal{O}_q(G)$. This $\mathcal{O}_q(G)$ does not correspond to our $\mathcal{O}_q(G)$ which denotes the reflection equation algebra.} $\SL_q(2)$ has a straightforward topological interpretation as the stated skein algebra of the bigon $\mathcal{B}$ with isomorphism
\[\SkAlgSt{\mathcal{B}} \xrightarrow{\sim} \SL_q(2): \alpha(a, b) \mapsto u^{a}_{b} \]
By embedding  the marked annulus $\Sigma^{\bullet}_{0,2}$ into the bigon as shown in the figure they conclude:

\begin{prop}[{\cite[Proposition 4.25]{CL20}}]
  \label{prop:isoskeinrea}
    The stated skein algebra $\SkAlgSt{\Sigma^{\bullet}_{0,2}}$ is isomorphic as a Hopf algebra to the reflection equation algebra $\mathcal{O}_q(\SL_2)$.
\end{prop}

The isomorphism is described explicitly in the proof of Proposition 4.25\footnote{Note that the reflection equation algebra is denoted $BSL_q(2)$ by Costantino and Lê and referred to as the transmuted or braided version of the quantum coordinate algebra.} and on generators this isomorphism is given by 
\begin{equation}
  \label{eq:skeintorea}
  \beta(\varepsilon_1, \varepsilon_2) \mapsto \bar{C}(-\varepsilon_1) k^{-\varepsilon_1}_{\varepsilon_2}
\end{equation}
where $C(+) = \bar{C}(-) = -q^{-5/4}$ and $C(-) = \bar{C}(+) = q^{-1/4}$.

\begin{figure}
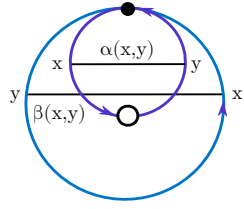

  \centering
  \diagramhh{loopcalc}{emb}{12pt}{1pt}{0.25}
  \caption{The bigon $\mathcal{B}$ shown in purple with its generators $\alpha(x,y)$ is embedded into $\Sigma^{\bullet}_{0,2}$ shown in blue with its generators $\beta(x,y)$.}
\end{figure}

The skein algebra of the (marked) annulus $\SkAlg{q}{\Sigma^{\bullet}_{0,2}}$ is isomorphic to $R[s_1]$ where $s_1$ is the loop around the puncture. If we consider $\SkAlg{q}{\Sigma^{\bullet}_{0,2}}$ as a subalgebra of the associated stated skein algebra we have 
\[s_1 = \diagramhh{loopcalc}{s1}{12pt}{1pt}{0.25} = q^{\frac{1}{4}}\diagramhh{loopcalc}{s1pm}{18pt}{1pt}{0.25} - q^{\frac{5}{4}} \diagramhh{loopcalc}{s1mp}{18pt}{1pt}{0.25} \mapsto q^{\frac{1}{4}} \bar{C}(-) k^-_- - q^{\frac{5}{4}} \bar{C}(+) k^+_+ = - q^{-1} k^-_- - q k^+_+, \]
so using \cref{prop:isoskeinrea} together with the results of \cref{sec:REA} we conclude:
\[f_1: \SkAlgSt{\Sigma^{\bullet}_{0,2}} \xrightarrow{\sim} \mathcal{O}_q(\SL_2),\; s_1 \mapsto - \qtr \]

Using the excision of stated skein algebras, this result can be extended to multiple punctures.

\begin{prop}[\cite{CL20}]
  \label{prop:nisoskeinrea}
    The stated skein algebra $\SkAlgSt{\Sigma^{\bullet}_{0,n+1}}$ is isomorphic as an algebra to the quantum loop algebra $\mathcal{L}_{\Sigma_{0,n+1}}=\mathcal{O}_q(\SL_2)^{\tilde{\otimes} n}$. Moreover, this isomorphism sends the closed loop $s_A$, for $A=\{i_1<\cdots<i_r\}$, to the element $-\Delta^{(k-1)}(\qtr)_{\underline{i}}$ of $\mathcal{O}_q(\SL_2)^{\tilde{\otimes} n}$.
\end{prop}

\begin{proof}
  The isomorphism is given in \cite[Proposition 4.25]{CL20}. We untangle the definition of this isomorphism and compute it on the closed loops $s_A$. Let $A=\{i_1 < \cdots < i_k\}$ be a subset of $\{1,\cdots,n\}$.
 
  Using one puncture to flatten the sphere the loop $s_A$ has form
  \[\diagramhh{loopcalc}{sA}{12pt}{1pt}{0.25}\]
  As before we apply \cref{eq:skeinboundary} to obtain
  \[\bar{C}(+)^{-1}\diagramhh{loopcalc}{sApm}{18pt}{1pt}{0.25} + \bar{C}(-)^{-1} \diagramhh{loopcalc}{sAmp}{18pt}{1pt}{0.25}\]
  Applying the relation again at each puncture leaves us with 
  \[\sum_{\varepsilon,\varepsilon_i}\bar{C}(-\varepsilon)^{-1} \prod_{i=1}^{k-1} \bar{C}(\varepsilon_i)^{-1} \diagramhh{loopcalc}{sAdecomp}{20pt}{1pt}{0.25}\]
  where we sum over all possible values of $\varepsilon,\varepsilon_i \in \{\pm\}$.
  The map $g_n$ is now simply given by \cref{eq:skeintorea} on each puncture so we have 
  \begin{align*} 
  \sum_{\varepsilon,\varepsilon_i}\left(\bar{C}(-\varepsilon)^{-1}\bar{C}(\varepsilon) \prod_{i=1}^{k-1} \bar{C}(\varepsilon_i)^{-1}\bar{C}(\varepsilon_i)\right) \left(\bigotimes_{i=1}^{k} k^{\varepsilon_i}_{\varepsilon_{i+1}}\right)_{\!\!\underline{i}} &= \sum_{\varepsilon \in \{\pm\}}\bar{C}(-\varepsilon)^{-1} \bar{C}(\varepsilon)
  \Delta^{(k-1)}\left(k^{\varepsilon}_{\varepsilon}\right)_{\underline{i}} \\
  &= -q\Delta^{(k-1)}\left(k^{+}_{+}\right)_{\underline{i}}-q^{-1}\Delta^{(k-1)}\left(k^{-}_{-}\right)_{\underline{i}}\\
  &= -\Delta^{(k-1)}(\qtr)_{\underline{i}}
  \end{align*}
  as required.
\end{proof}

\begin{rmk}
  In \cite{CL20}, the braided structure on the tensor product $\mathcal{O}_q(\SL_2)^{\tilde{\otimes} n}$ is defined using the co-$R$-matrix $\rho$. But, as an easy consequence of \cref{prop:cobr=br}, this braided structure is the same as the braided structure obtained with the $R$-matrix of $\Uq$ and the left action of $\Uq$ on $\mathcal{O}_q(\SL_2)$ arising from the adjoint action.
\end{rmk}

Combining \cref{prop:nisoskeinrea} with our previous results we conclude:
\begin{thm}
  \label{thm:iso}
  If $q$ is generic then there is an algebra isomorphism $\SkAlg{q}{\Sigma^{\bullet}_{0,n+1}} \to \AW{n}$ sending the closed loop $s_A\in \SkAlg{q}{\Sigma^{\bullet}_{0,n+1}}$ to $-\Lambda_A \in \AW{n}$.
\end{thm}

\begin{proof}
  By \cref{prop:nisoskeinrea} there is an isomorphism $\SkAlgSt{\Sigma^{\bullet}_{0,n+1}} \to \mathcal{O}_q(\SL_2)^{\tilde{\otimes} n}$. By \cref{sec:REA} there is an isomorphism $\mathcal{O}_q(\SL_2) \to \Uq^{\lf}$ which gives us an isomorphism $g_n: \SkAlgSt{\Sigma^{\bullet}_{0,n+1}} \to \left(\Uq^{\lf}\right)^{\tilde{\otimes} n}$. Combining this with the injective unbraiding map $\gamma_n: \left(\Uq^{\lf}\right)^{\tilde{\otimes} n} \to \Uq^{\otimes n}$ defined in \cref{sec:unbraiding} gives us an injective algebra morphism 
  \[\SkAlgSt{\Sigma^{\bullet}_{0,n+1}} \xrightarrow{g_n} \left(\Uq^{\lf}\right)^{\tilde{\otimes} n} \xrightarrow{\gamma_n} \Uq^{\otimes n}. \]
  Therefore, to prove the result it is sufficient to show that the generators $s_A \in \SkAlg{q}{\Sigma^{\bullet}_{0,n+1}} \subset \SkAlgSt{\Sigma^{\bullet}_{0,n+1}}$ are sent to the generators $-\Lambda_A \in \AW{n} \subset \Uq^{\otimes n}$, which follows from the explicit computations in \cref{prop:nisoskeinrea} and \cref{cor:AW_generators}.
\end{proof}



\section{Commutator Relations}
\label{sec:commutator}

As an illustration of the usefulness of being able to use diagrams for calculations involving Askey--Wilson algebras, we are now going to reprove the main results of \cite{DeClercq19}. The first result is proving that loops in $\AW{n}$ satisfy a generalisation of the commutator relations used to define $\AW{3}$.

\begin{thm}
\label{thm:commutators}
Let $A_1, A_2, A_3, A_4 \subseteq{\{1, \dots, n\}}$ such that $\max{A_i} < \min{A_{j}}$ whenever $i<j$ and both $A_i$ and $A_j$ are non empty.
Let $A$ and $B$ be one of the following
\begin{enumerate}
    \item $A = A_1 \cup A_2 \cup A_4$ and $B = A_2 \cup A_3$,
    \item $A = A_2 \cup A_3$ and $B = A_1 \cup A_3 \cup A_4$,
    \item $A = A_1 \cup A_3 \cup A_4$ and $B = A_1 \cup A_2 \cup A_4$
\end{enumerate}
We have the commutator
\[[\Lambda_{A},\; \Lambda_{B}]_q = \left(q^{-2}- q^{2}\right) \Lambda_{(A \cup B) \backslash (A \cap B)} + \left(q -q^{-1}\right)\left(\Lambda_{A \cap B} \Lambda_{A \cup B} + \Lambda_{A \backslash (A \cap B)} \Lambda_{B \backslash (A \cap B)}\right).\]
\end{thm}


\begin{proof}
	We use the isomorphism $\SkAlg{q}{\Sigma_{0, n+1}} \to \AW{n}\colon s_A \mapsto -\Lambda_A$ and instead prove the result for loops in $\SkAlg{q}{\Sigma_{0, n+1}}$.
	As usual, we represent $\Sigma_{0, n+1}$ as $n$ points in a line with the final point used to flatten the sphere onto the page. We omit any point not in $A_1 \cup A_2 \cup A_3 \cup A_4$ as these points make no difference to the calculation: either they are to the left or right of the loops, or the loops pass below them. We note that the condition on the $A_i$ means that all the points are partitioned into $A_1, A_2, A_3, A_4$ in order.

	If $A = A_1 \cup A_2 \cup A_4$ and $B = A_2 \cup A_3$ then we have
  \begin{align*}
    s_A s_B &= \diagramhh{skeincommutatorproofv2}{AB2}{12pt}{0pt}{0.3} = \diagramhh{skeincommutatorproofv2}{A1A42}{12pt}{0pt}{0.3} + q^{-1} \; \diagramhh{skeincommutatorproofv2}{NotStandard2}{12pt}{0pt}{0.3} \\
    &+ q \; \diagramhh{skeincommutatorproofv2}{A1A3A42}{12pt}{0pt}{0.3} + \diagramhh{skeincommutatorproofv2}{AllA2}{12pt}{0pt}{0.3} \\
    s_B s_A &= \diagramhh{skeincommutatorproofv2}{BA2}{12pt}{0pt}{0.3} = \diagramhh{skeincommutatorproofv2}{A1A42}{12pt}{0pt}{0.3} + q \; \diagramhh{skeincommutatorproofv2}{NotStandard2}{12pt}{0pt}{0.3} \\
     &+ q^{-1} \; \diagramhh{skeincommutatorproofv2}{A1A3A42}{12pt}{0pt}{0.3} + \diagramhh{skeincommutatorproofv2}{AllA2}{12pt}{0pt}{0.3}
  \end{align*}
Hence, we have
\begin{align*}
    q s_A s_B - q^{-1} s_B s_A &=  \left(q^{2}- q^{-2}\right) s_{(A \cup B) \backslash (A \cap B)} + \left(q -q^{-1}\right) (s_{A \backslash (A \cap B)} s_{B \backslash (A \cap B)}+ s_{A \cap B} s_{A \cup B}).
\end{align*}

The other cases are similar.
\end{proof}

The second result of \cite{DeClercq19} is even easier to prove:

\begin{thm}
	Let $ B \subseteq A \subset{\{1, \dots, n\}}$ then $\Lambda_A$ and $\Lambda_B$ commute.
\end{thm}

\begin{proof}
	If $ B \subseteq A \subset{\{1, \dots, n\}}$ then the loops $s_A$ and $s_B$ do not intersect so they commute.
\end{proof}

As the loops $s_A$ and $s_B$ also do not intersect if $A \cap \{\min{B}, \dots, \max{B}\} = \emptyset$, we also have:

\begin{prop}
  Let $ B,\, A \subset{\{1, \dots, n\}}$ such that $A \cap \{\min{B}, \dots, \max{B}\} = \emptyset$ then $\Lambda_A$ and $\Lambda_B$ commute.
\end{prop}


\begin{rmk}
  Note that the loops $s_A$ and $s_B$ also do not intersect when $A \cap \{\min{B}, \dots, \max{B}\} = \emptyset$. Hence, we can immediately conclude in this case that $s_A$ and $s_B$ commute. Alternatively, this is the case of \cref{thm:commutators}  with $A_1=\{\,a\in A \;\vert\; a<b,\; \forall b \in B\,\}$, $A_2=\emptyset$, $A_3=B$ and $A_4 = \{\,a\in A \;\vert \;a>b,\; \forall b \in B\,\}$.
\end{rmk}

\section{Action of the braid group}
\label{sec:braid_grp}

As noted in \cite[Section 8]{CFG21}, both the Askey--Wilson algebra $\AW{3}$ and the skein algebra of the $4$-punctured sphere admit an action of the braid group on $3$ strands, and these actions are compatible with the isomorphism between the Askey--Wilson algebra and the skein algebra. We give in this section a higher rank version of this result.

Recall that the braid group on $n$ strands $B_n$ is the group with the following presentation:
\[
  \left\langle \; \beta_1,\ldots,\beta_{n-1}\ \middle\vert\ \beta_i\beta_{i+1}\beta_i = \beta_{i+1}\beta_i\beta_{i+1} \text{ for } 1 \leq i < n-1, \beta_i\beta_j= \beta_j\beta_i \text{ for } \lvert i-j \rvert > 1 \; \right\rangle.
\]

\subsection{Action of $B_n$ on the skein algebra of the $n+1$-punctured sphere}
\label{sec:braid_skein}

The braid group $B_n$ acts by half Dehn twists on the skein algebra of the $(n+1)$-punctured sphere. It permutes the punctures by anti-clockwise rotations and any framed link on the sphere is continuously deformed during the rotation process. 

\begin{ex}
  For example, the generator $\beta_2$ permutes anti-clockwise the second and third punctures and the framed link is deformed during the process:
  \[\beta_2 \cdot \left( \diagramhh{braidgroup}{s134}{12pt}{1pt}{0.25}
  \right) = \diagramhh{braidgroup}{braided}{12pt}{1pt}{0.25}\]
\end{ex}

\begin{prop}
  \label{prop:action_skein}
  Let $A=\{i_1<\cdots<i_k\}$ be a non-empty subset of $\{1,\ldots,n\}$ and $1 \leq i \leq n-1$. We have
\[
  \beta_i\cdot s_A =
  \begin{cases}
    s_A & \text{if } i,i+1 \in A \text{ or } i,i+1\not\in A,\\
    s_{\left(A\backslash \{i\}\right)\cup \{i+1\}} & \text{if } i\in A, i+1\not\in A,\\
    \displaystyle \frac{\left[s_{A},s_{\{i,i+1\}}\right]_q-\left(q-q^{-1}\right)\left(s_{i+1}s_{A\cup\{i\}}+s_{i}s_{A\backslash\{i+1\}}\right)}{q^2-q^{-2}} & \text{if } i\not\in A, i+1\in A.
  \end{cases}
\]
\end{prop}

\begin{proof}
  If $i,i+1 \in A$ or $i,i+1\not\in A$, there is nothing to prove. If $i \in A$ and $i+1\not\in A$, it is clear that $\beta_i\cdot s_A =s_{(A\backslash \{i\})\cup \{i+1\}}$. The case $i\not\in A$ and $i+1\in A$ is a pleasant computation along the lines of the graphical proof of \cref{thm:commutators}.
\end{proof}

\subsection{Action of $B_n$ on the higher rank Askey--Wilson algebra}
\label{sec:braid_AW}

The action of the braid group $B_n$ on $\AW{n}$ by algebra automorphisms is given by conjugation by the $R$-matrix. Given $x\in \Uq\otimes \Uq$, we define $\beta(x)$ in (a completion of) $\Uq\otimes \Uq$ by:
\[
  \beta(x) = \sigma\left(\Psi^{-1}\left(\Theta x \Theta^{-1}\right)\right)
\]
where $\sigma(a\otimes b) = b\otimes a$. For $1 \leq i \leq n-1$ the action of the generator $\beta_i\in B_n$ is given by the endomorphism $\id^{\otimes i-1}\otimes \beta \otimes \id^{\otimes n-i-1}$. Once again, we should act on a completion of $\Uq^{\otimes n}$ because the quasi-$R$-matrix $\Theta$ is an infinite sum.

The properties of the quasi-$R$-matrix $\Theta$ ensure that we obtain an action of $B_n$. 

Thanks to \cref{lem:tau-conj}, the generators $\Lambda_A$ of the Askey--Wilson algebra can be rewritten using the action of the braid group. Given $i<j$, let $\beta_{i,j} = \beta_{j-1}\cdots\beta_{i+1}\beta_{i}$. Then if $A=\{i_1<\cdots<i_k\}$ is a non-empty subset of $\{1,\ldots,n\}$, we have
\[
\Lambda_A = \beta_{1,i_1}\beta_{2,i_2}\cdots\beta_{k,i_k}\cdot\left(\Delta^{(k-1)}(\Lambda)\otimes 1^{\otimes n-k}\right) \in \AW{n}.
\]

\begin{prop}
\label{prop:action_AW}
  The action of the braid group on (a completion of) $\Uq^{\otimes n}$ restricts to $\AW{n}$. Moreover, the images of the generators $\Lambda_A$ are given as follows:
\[
\beta_i\cdot \Lambda_A =
\begin{cases}
  \Lambda_A & \text{if } i,i+1 \in A \text{ or } i,i+1\not\in A,\\
  \Lambda_{(A\backslash \{i\})\cup \{i+1\}} & \text{if } i\in A,\; i+1\not\in A ,\\
  \displaystyle -\frac{\left[\Lambda_{A},\;\Lambda_{\{i,i+1\}}\right]_q-\left(q-q^{-1}\right)\left(\Lambda_{i+1}\Lambda_{A\cup\{i\}}+\Lambda_{i}\Lambda_{A\backslash\{i+1\}}\right)}{q^2-q^{-2}} & \text{if } i\not\in A,\; i+1\in A\text{.}
\end{cases}
\]
\end{prop}

\begin{proof}
  The first assertion follows from the explicit formulas for the action since this shows that $\psi_i(\Lambda_A) \in \AW{n}$ for all $1 \leq i \leq n-1$ and $A\subseteq \{1,\ldots,n\}$.

  Let $A=\{i_1<\cdots <i_k\}$ be a non-empty subset of $\{1,\ldots,n\}$. We first suppose that $i\in A$. Let $j$ be such that $i_j = i$. Since $\beta_i\beta_{r,s} = \beta_{r,s}\beta_i$ if $i+1<r$ or $s<i$, we have
  \begin{align*}
    \beta_i\cdot \Lambda_A &= \beta_i\cdot \left(\beta_{1,i_1}\beta_{2,i_2}\cdots\beta_{k,i_k}\cdot\left(\Delta^{(k-1)}(\Lambda)\otimes 1^{\otimes n-k}\right)\right)\\
                      &= \beta_{1,i_1}\cdots\beta_{j-1,i_{j-1}}\beta_i\beta_{j,i_j}\cdots\beta_{k,i_k}\cdot\left(\Delta^{(k-1)}(\Lambda)\otimes 1^{\otimes n-k}\right)\\
                      &= \beta_{1,i_1}\cdots\beta_{j-1,i_{j-1}}\beta_{j,i+1}\beta_{j+1,i_{j+1}}\cdots\beta_{k,i_k}\cdot\left(\Delta^{(k-1)}(\Lambda)\otimes 1^{\otimes n-k}\right),
  \end{align*}
which is obviously equal to $\Lambda_{(A\backslash \{i\})\cup \{i+1\}}$ if $i+1\not\in A$. 
 
If $i+1\in A$, then $i_{j+1} = i+1$ and since $\beta_{j,i+1}\beta_{j+1,i+1} = \beta_{j,i} \beta_{j,i+1} = \beta_{j,i} \beta_{j+1,i+1}\beta_j$, we find that
\[
  \beta_i\cdot\Lambda_A = \beta_{1,i_1}\beta_{2,i_2}\cdots\beta_{k,i_k}\beta_j\cdot\left(\Delta^{(k-1)}(\Lambda)\otimes 1^{\otimes n-k}\right).
\]
As $\Psi(\Delta^{\op}(x))\Theta = \Theta\Delta(x)$ and $j<k$, we have $\beta_j\cdot\left(\Delta^{(k-1)}(\Lambda)\otimes 1^{\otimes n-k}\right) = \Delta^{(k-1)}(\Lambda)\otimes 1^{\otimes n-k}$. Therefore, $\beta_i\cdot\Lambda_A = \Lambda_A$.

We now suppose that $i\not\in A$. If we also have that $i+1\not\in A$ then the arguments are similar to the previous case. If $i+1\in A$, we set $A'=\big(A\backslash\{i+1\}\big)\cup\{i\}$. Thanks to \cref{thm:commutators}, we have
\[
  \left[\Lambda_{A'},\;\Lambda_{\{i,i+1\}}\right]_q = -\left(q^2-q^{-2}\right)\Lambda_A + \left(q-q^{-1}\right)\left(\Lambda_{i+1}\Lambda_{A\cup\{i\}}+\Lambda_{i}\Lambda_{A\backslash\{i+1\}}\right).
\]
We now act with $\beta_i$ to obtain
\begin{multline*}
  \left[\beta_i\cdot \Lambda_{A'},\;\beta_i\cdot \Lambda_{\{i,i+1\}}\right]_q \\= -\left(q^2-q^{-2}\right)\beta_i\cdot \Lambda_A + \left(q-q^{-1}\right)\left(\left(\beta_i\cdot \Lambda_{i}\right)\left(\beta_i\cdot\Lambda_{A\cup\{i\}}\right)+\left(\beta_i\cdot \Lambda_{i+1}\right)\left(\beta_i\cdot\Lambda_{A\backslash\{i+1\}}\right)\right).
\end{multline*}
But $\beta_i\cdot\Lambda_{A'} = \Lambda_A$, $\beta_i\cdot\Lambda_{\{i,i+1\}}=\Lambda_{\{i,i+1\}}$, $\beta_i\cdot\Lambda_{A\cup\{i\}}=\Lambda_{A\cup\{i\}}$ and $\beta_i\cdot\Lambda_{A\backslash\{i+1\}}=\Lambda_{A\backslash\{i+1\}}$ by the previous cases and $\beta_i\cdot\Lambda_{i} = \Lambda_{i+1}$ and $\beta_i\cdot\Lambda_{i+1} = \Lambda_{i}$ since the Casimir $\Lambda$ is central. We then obtain the formula for $\beta_i\cdot\Lambda_A$.
\end{proof}

Form \cref{prop:action_skein} and \cref{prop:action_AW}, we immediately deduce the following

\begin{prop}
  The algebra isomorphism $\SkAlg{q}{\Sigma_{0, n+1}}\rightarrow \AW{n}$ given by $s_A\mapsto -\Lambda_A$ commutes with the action of the braid group $B_n$.
\end{prop}

\section{Graded Dimensions}
\label{sec:dimensions}

In this section we will compute the Hilbert series of the Askey--Wilson algebra $\AW{n}$ and the skein algebra $\SkAlg{q}{\Sigma_{g,r}}$ of the surface $\Sigma_{g,r}$ of genus $g$ with $r>0$ punctures. These algebras are filtered and their Hilbert series encodes vector space dimension of each graded part of the associated graded algebra. In the next section, we will use these Hilbert series to find presentations for $\AW{4}$ and $\SkAlg{q}{\Sigma_{0,5}}$.

\begin{defn}
The \emph{Hilbert series} of the \(\mathbb{Z}_{\geq 0}\) graded vector space \(A = \bigoplus_{n \in \mathbb{Z}_{\geq 0}} A[n-1]\) is the formal power series
\[
  h_A(t) = \sum_{n \in \mathbb{Z}_{\geq 0}} \dim(A[n]) t^n.
\]
The  Hilbert series of a \(\mathbb{Z}_{\geq 0}\) graded algebra \(A\) is the Hilbert series of its underlying \(\mathbb{Z}_{\geq 0}\) graded vector space.
\end{defn}

We will use the isomorphism between the subalgebra $\mathcal{L}_{\Sigma_{0,n+1}}^{\Uq}$ of the Aleeksev moduli algebra $\mathcal{L}_{\Sigma_{0,n+1}}$
which is invariant under the action of $\Uq$ and the Askey--Wilson algebra $\AW{n}$ from \cref{sec:moduli}, and also the isomorphism between $\mathcal{L}_{\Sigma_{g,r}}^{\Uq}$ and the skein algebra $\SkAlg{q}{\Sigma_{g,r}}$ from \cref{sec:skein}, so that we can instead compute the Hilbert series of $\mathcal{L}_{\Sigma_{g,r}}^{\Uq}$ whose compution is a generalisations of the calculations for $\Sigma_{0,4}$ by the first author in \cite{Cooke18}.

Recall that the Alekseev moduli algebra 
\[\mathcal{L}_{\Sigma_{g,r}} = \mathcal{D}_q(\SL_2)^{\tilde{\otimes} g}\tilde{\otimes}\mathcal{O}_q(\SL_2)^{\tilde{\otimes} r-1} = \mathcal{O}_q(\SL_2)^{\hat{\otimes} 2g}\tilde{\otimes}\mathcal{O}_q(\SL_2)^{\tilde{\otimes} r-1}\]
is generated as an algebra by elements $(k_i)_{\delta}^{\epsilon} \in \mathcal{K}_i \hookrightarrow{} \mathcal{L}_{\Sigma_{g,r}}$ for $i \in {1, \dots, (2g+r-1)}$. If we define the degree $ \left| (k_i)_{\delta}^{\epsilon} \right| = 1$ then all the relations of $\mathcal{L}_{\Sigma_{g,r}}$ are homogeneous except the determinant relations $(k_i)_+^+(k_i)_-^--q(k_i)_-^+(k_i)_+^- = 1$ for which the non-homogeneous element is in the ground ring. Hence, $\mathcal{L}_{\Sigma_{g,r}}$ is filtered with the $j^{th}$ filtered part $\left(\mathcal{L}_{\Sigma_{g,r}}\right)(j)$ 
being the vector space spanned by all monomials in the generators $(k_i)_{\delta}^{\epsilon}$ with degree at most $j$\footnote{Let $A$ be a filtered algebra with generators $\{x_i\}$ to which we assign degrees $|x_i| \in \mathbb{Z}_{>0}$. Then the degree of an element $f \in A$ is the smallest degree of any polynomial in the $\{x_i\}$ which represents $f$. Note that we have $\operatorname{degree}(fg) \leq \operatorname{degree}(f) \operatorname{degree}(g)$ rather than equality.
}. As $\mathcal{L}_{\Sigma_{g,r}}$ is filtered rather than graded we need to consider its associated graded algebra.

\begin{defn}
The \emph{associated graded algebra} of the \(\mathbb{Z}_{\geq 0}\) filtered algebra \(A = \bigcup_{n \in \mathbb{Z}_{\geq 0}} A(n)\) is 
\vspace{-1em}
\[
\mathscr{G}(A) = \bigoplus_{n \in \mathbb{Z}_{\geq 0}} A[n] \text{ where } A[n] = \begin{cases}
A(0) & \text{ for }n = 0 \\
\faktor{A(n)}{A(n-1)} &\text{ for } n > 0.
\end{cases}
\]
The \emph{Hilbert series} of the \(\mathbb{Z}_{\geq 0}\) filtered algebra \(A = \bigcup_{n \in \mathbb{Z}_{\geq 0}} A(n)\) is the Hilbert series of the associated graded algebra \(\mathscr{G}(A)\).
\end{defn}

As $\mathcal{L}_{\Sigma_{g,r}}$ is acted on by $\Uq$ we can also decompose it into its weight spaces to obtain its character.

\begin{defn}
  Let \(V\) be a vector space acted on by \(\Uq\) and let \(V^k\) denote the \(q^k\)-weight space of \(V\) where \(k \in \mathbb{Z}\). The \emph{character} of \(V\) is the formal power series
  \[
  \ch_V(u) = \sum_{k \in \mathbb{Z}} \dim \left(V^{k} \right) u^{k}.
  \]
\end{defn}

Using both the decomposition into graded parts and into weight spaces simultaneously gives the graded character.

\begin{defn}
Let \(V = \bigoplus_{n} V[n]\) be a graded vector space acted on by \(\Uq\). The \emph{graded character} of \(V\) is  
\[
h_V(u,t) := \sum_n \ch_{V[n]}(u) t^n = \sum_{n,k} \dim \left(V[n]^k\right) u^k t^n,
\]
where \(V[n]^k\) is the \(q^k\)-weight space of \(V[n]\). If \(V\) is filtered rather than graded the graded character of \(V\) \(h_V(u,t)\) is \(h_{\mathscr{G}(V)}(u,t)\), the graded character of associated graded vector space \(\mathscr{G}(V)\).
\end{defn}

As the Alekseev moduli algebra \(\mathcal{L}_{\Sigma_{g,r}}\) is simply multiple copies of the reflection equation algebra $\mathcal{O}_q(\SL_2)$ tensored together its graded character is easy to determine:

\begin{prop}
  \label{prop:charofL}
The graded character of \(\mathcal{L}_{\Sigma_{g,r}}\) is 
\begin{align*}
    h_{A}(u,t) = \left(\frac{(1 + t)}{(1 - t) (1-u^2t) (1-u^{-2}t)}\right)^{2g + r - 1}.
\end{align*}
\end{prop}
\begin{proof}
	From \cite[Proposition A.6.]{Cooke18} we have that the graded character of \(\mathcal{O}_q(\SL_2)\) is 
\begin{align*}
    h_{\mathcal{O}_q(\SL_2)}(u,t) = \frac{(1 + t)}{(1 - t) (1-u^2t) (1-u^{-2}t)}.
\end{align*}
and as $\mathcal{L}_{\Sigma_{g,r}} \cong (\mathcal{O}_q(\SL_2) \hat{\otimes} \mathcal{O}_q(\SL_2))^{\otimes g} \tilde{\otimes} \mathcal{O}_q(\SL_2)^{\otimes r - 1}$ this gives the result.
\end{proof}
\noindent
This can now be used to compute the Hilbert series of its $\Uq$ invariant subalgebra:

\begin{thm}
\label{thm:dimensions}
The Hilbert series of $\mathcal{L}_{\Sigma_{g,r}}^{\Uq}$ is
\[h(t) = \frac{(1+t)^{n-2}}{(1-t)^{n}(1-t^2)^{2n-3}}\left(\sum_{k=0}^{n-2}{\binom{n-2}{k}}^2t^{2k} - \sum_{k=0}^{n-3}\binom{n-2}{k}\binom{n-2}{k+1}t^{2k+1}\right) \]
where $n = 2g + r -1$.
\end{thm}
\begin{proof}
As we have filtered isomorphisms is sufficient to prove this result for $\mathcal{L}_{\Sigma_{g,r}}^{\Uq}$.
We have that
\begin{align*}
    (1-ax)^{-n} &= \sum_{k=0}^{\infty} \rchoose{n}{k} a^k x^k, 
\end{align*}
where $\rchoose{n}{k} := \displaystyle\binom{n+k-1}{k}$. Appling this to the graded character of $A = \mathcal{L}_{\Sigma_{g,r}}$ from \cref{prop:charofL} gives:
\begin{align*}
    h_{A}(u,t) &= \left(\frac{1 + t}{1 - t}\right)^n \frac{1}{(1-u^2t)^n(1-u^{-2}t)^n} \\
    &= \left(\frac{1 + t}{1 - t}\right)^n \left( \sum_{k=0}^{\infty} \rchoose{n}{k} u^{2k} t^k\right) \left( \sum_{l=0}^{\infty} \rchoose{n}{l} u^{-2l} t^l \right) \\
    &= \left(\frac{1 + t}{1 - t}\right)^n \left( \sum_{k+l = m} \rchoose{n}{k} \rchoose{n}{l} u^{2(k-l)} t^m\right)
\end{align*}
Similarly to in \cite{Cooke18} the Hilbert series \(h_{\mathscr{A}}(t)\) of $\mathscr{A} = \mathcal{L}_{\Sigma_{g,r}}^{\Uq}$ is determined by the \(u\) coefficient of \((u - u^{-1}) h_{A}(u,t)\). The \(u\) coefficient is 
\begin{align*}
h_{\mathscr{A}}(t)=&\left(\frac{1 + t}{1 - t}\right)^n \left( \sum_{k+l = m, k - l = 0} \rchoose{n}{k} \rchoose{n}{l} t^m  - \sum_{k+l = m, k - l = 1} \rchoose{n}{k} \rchoose{n}{l} t^m \right) \\
=&\left(\frac{1 + t}{1 - t}\right)^n \left( \sum_{m = 0}^{\infty} \rchoose{n}{m}^2 t^{2m}  - \sum_{m = 0}^{\infty} \rchoose{n}{m} \rchoose{n}{m+1} t^{2m+1} \right).
\end{align*}

Using \cite[(6.10)]{CGPdAV}, we have
\begin{multline*}
\frac{1}{(1 - t)^n} \left( \sum_{m = 0}^{\infty} \rchoose{n}{m}^2 t^{2m}  - \sum_{m = 0}^{\infty} \rchoose{n}{m} \rchoose{n}{m+1} t^{2m+1} \right) \\ = \frac{(1+t)^{n-2}}{(1-t^2)^{3(n-1)}}\left(\sum_{k=0}^{n-2}{\binom{n-2}{k}}^2t^{2k} - \sum_{k=0}^{n-3}\binom{n-2}{k}\binom{n-2}{k+1}t^{2k+1}\right),
\end{multline*}
which gives the expected formula after multiplying by $(1+t)^n$.
\end{proof}

Using the isomorphisms 
\[
  \mathcal{L}_{\Sigma_{g,r}}^{\Uq} \to \SkAlg{q}{\Sigma_{g,r}}: -\Delta^{(k-1)}(\qtr)_{\underline{i}} \mapsto s_A \text{ and } \mathcal{L}_{\Sigma_{0, n+1}}^{\Uq} \to \AW{n}: \Delta^{(k-1)}(\qtr)_{\underline{i}} \mapsto \Lambda_A
\]
from \cref{sec:moduli} and \cref{sec:skein} we can induce a filtered structure on $\SkAlg{q}{\Sigma_{g,r}}$ such that $s_A$ has degree $|A|$ and a filtered structure on $\AW{n}$ such that  $\Lambda_A$ has degree $|A|$. 

\begin{cor}
  \label{cor:hilbskein}
  The Hilbert series of the skein algebra $\SkAlg{q}{\Sigma_{g, r}}$ and the higher rank Askey--Wilson algebra $\AW{n}$ is 
\[h(t) = \frac{(1+t)^{n-2}}{(1-t)^{n}(1-t^2)^{2n-3}}\left(\sum_{k=0}^{n-2}{\binom{n-2}{k}}^2t^{2k} - \sum_{k=0}^{n-3}\binom{n-2}{k}\binom{n-2}{k+1}t^{2k+1}\right) \]
where $n = 2g + r -1$.
\end{cor}




\begin{rmk}
  In \cite[Section 6.2]{CGPdAV}, it is shown that the polynomial $(1-t)^n h(t)$ is the Hilbert series of the centraliser of the diagonal action $U(\mathfrak{sl}_2)$ in $U\left(\mathfrak{sl}_2^{\otimes 2}\right)$ and that the numerator has positive coefficients. 
  The term $\frac{1}{(1-t)^n}$ is from counting the $n$ simple loops $s_1,\ldots,s_n$ which are central and have no relations with any other loops so that $\SkAlg{q}{\Sigma_{0, n+1}}$ is free over the subalgebra generated by them.
\end{rmk}

\begin{rmk}
 This Hilbert series can also be written in terms of the hypergeometric function as follows:
 \[
    h(t) = \frac{(1+t)^n}{(1-t)^n} \left( \prescript{}{2}{F}_1(n,n;1;t^2) - nt \prescript{}{2}{F}_1(n,n+1;2;t^2) \right)
 \]
 \end{rmk}




\section{Presentation of the Skein Algebra of the Five-Punctured Sphere}
\label{sec:presentation}

In this section we shall use the isomorphisms between the Askey--Wilson algebra $\AW{n}$, the \(\Uq\)-invariants of the Alekseev moduli algebra and the skein algebra $\SkAlg{q}{\Sigma_{0, n+1}}$ together with the Hilbert series computed in the previous section to obtain a presentation for $\SkAlg{q}{\Sigma_{0, 5}}$ and therefore also for $\AW{4}$. This case represents the lowest of the higher-rank Askey--Wilson algebras and was the case considered by Post and Walker. 

Presentations of the Kauffman bracket skein algebra in the punctured surface case are only known for a handful of the simplest cases: punctured spheres with up to four punctures and punctured tori with either one or two punctures. 
The Hilbert series of $\SkAlg{q}{\Sigma_{g,r}}$ only depends on $n = 2g + r -1$ and so the cases for which a presentation for $\SkAlg{q}{\Sigma_{g,r}}$ in known correspond to $n=1,2,3$ whereas in this section we shall consider the five-punctured sphere which corresponds to $n=4$. Applying \cref{cor:hilbskein} for $n=4$ we get:

\begin{cor}
    \label{cor:hilbn4}
The Hilbert series of $\SkAlg{q}{\Sigma_{0,5}}$ and $\AW{4}$ is
\[
   h(t) = \frac{(1+t)^2}{(1-t)^4(1-t^2)^5}\left(1-2t+2t^2-2t^3+t^4\right)=\frac{1+t^2+4t^3+t^4+t^6}{(1-t)^4(1-t^2)^5}
\]
\end{cor}

The difficulty for finding presentations for more complex punctured surfaces is that the number of generators and relations required increases and also, whilst it is easy to find relations using diagrams and resolving crossings, it is difficult to prove that you have found all the relations. We are able to overcome this second difficulty using the above Hilbert series.

In order to find the relations we will first generalise the relations in the presentation for the four-punctured sphere found by Bullock and Przytycki\footnote{We have corrected a sign error in the first relation which appears in the published version of the paper \cite{BullockPrzytycki00}.} before considering the additional relations which are of a genuinely different nature. 

Let $x_1 = s_{12}, x_2 = s_{23}$ and $x_3 = s_{13}$ and let $s_4$ denote the loop around the outside puncture (see \cref{figure:FPuncSphereEdges}). If curve $x_i$ separates $s_i,s_j$ from $s_k, s_\ell$, let
$ p_i = s_i s_j + s_k s_\ell$. Explicitly,
\[
p_1 = s_1 s_2 + s_3 s_4,\quad \quad p_2 = s_2 s_3 + s_1 s_4,\quad \quad p_3 = s_1 s_3 + s_2 s_4
\]

\begin{thm}[\cite{BullockPrzytycki00}]
\label{thm:presentationofsphereskein}
As an algebra over the polynomial ring $R[s_1,s_2,s_3,s_4]$,
the Kauffman bracket skein algebra \(\SkAlg{q}{\Sigma_{0,4}}\) has a presentation with generators \(x_1,\,x_2,\,x_3\) and  relations
\begin{align}
    \left[x_i,\; x_{i+1}\right]_{q} &= \left(q^2 - q^{-2}\right)x_{i+2} + \left(q - q^{-1}\right)p_{i+2} \text{ (indices taken modulo 3)}; \label{eq:cubic_commutator}\\
    \Omega_K &= \left(q +q^{-1}\right)^2 - \left(p_1 p_2 p_3 p_4 + p_1^2 + p_2^2 + p_3^2 + p_4^2\right); \label{eq:cubic}
\end{align}
where we have used the following Casimir element:
\[\Omega_K := -q x_1 x_2 x_3 + q^2 x_1^2 + q^{-2} x_2^2 + q^2 x_3^2 + q p_1 x_1 + q^{-1}p_2 x_2 + q p_3 x_3.\]
\end{thm}

\begin{figure}[ht]
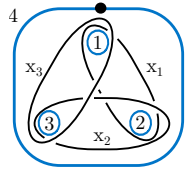

    \centering
    \diagramhh{present}{cubic}{0pt}{0pt}{0.25}
    \caption{This figure shows the product $x_1 x_2 x_3$, which is the leading term of the cubic relation, on the four-punctured sphere $\Sigma^{\bullet}_{0,4}.$}
    \label{figure:FPuncSphereEdges}
\end{figure}

\begin{figure}[ht]
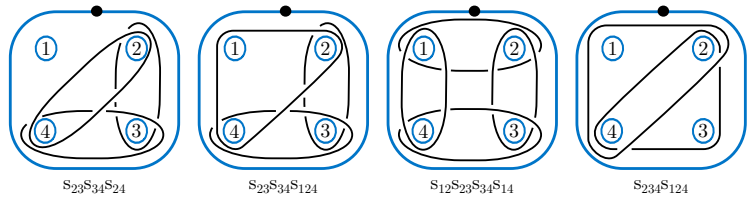

    \centering
    \diagramhh{present}{4cubic}{0pt}{0pt}{0.25} \diagramhh{present}{cubictriple}{0pt}{0pt}{0.25} \diagramhh{present}{quadratic}{0pt}{0pt}{0.25} \diagramhh{present}{2tripleloop}{0pt}{0pt}{0.25}
    \caption{The leading terms for the four types of relations in $\SkAlg{q}{\Sigma_{0,5}}$ which are generalisations of the cubic relation in $\SkAlg{q}{\Sigma_{0,4}}$.}
    \label{figure:4looprels}
\end{figure}
The first set of relations (\cref{eq:cubic_commutator}) are the commutators relations from \cref{thm:commutators} and we will also want the commutator relations from this theorem for our case. The second relation (\cref{eq:cubic}) can be derived by taking $x_1 x_2 x_3$ and resolving all crossings. As there are four ways to embed three points into four points we will end up with four such cubic relations; that is we will have relations whose left-hand sides are $s_{12} s_{23} s_{13}$, $s_{23} s_{34} s_{24}$, $s_{34} s_{14} s_{13}$ and $s_{12} s_{14} s_{24}$. We shall also have four cubic relations where one of the \emph{diagonal loops}, $s_{13}$ or $s_{24}$, has been replaced by a triple loop; these have left-hand sides $s_{12} s_{23} s_{134}$, $s_{23} s_{34} s_{124}$, $s_{34} s_{14} s_{123}$ and $s_{12} s_{14} s_{234}$.
Furthermore, instead of using three loops to create a closed loop, we can use four loops giving the \emph{quartic relation}:
\begin{align*}
s_{12}&s_{23}s_{34}s_{14} \\
    &= s_{1}s_{3}s_{12}s_{23} + s_{2}s_{4}s_{12}s_{14} + s_{2}s_{4}s_{23}s_{34} + s_{1}s_{3}s_{34}s_{14} \\
    &+ \left(qs_{3}s_{12} + qs_{2}s_{13} + q^{-1}s_{1}s_{23} + s_{123} + q^{-1}s_{4}s_{1234}+s_{1}s_{2}s_{3}\right)s_{123} \\
    &+ \left(qs_{4}s_{13} + qs_{3}s_{14} + q^{-1}s_{1}s_{34} + s_{134} + q^{-1}s_{2}s_{1234}+s_{1}s_{3}s_{4}\right)s_{134} \\
    &+ \left( qs_{4}s_{12} + qs_{2}s_{14} + q^{-1}s_{1}s_{24}  + s_{124} + q^{-1}s_{3}s_{1234}+s_{1}s_{2}s_{4}\right)s_{124} \\
    &+ q^{-2}\left(q^{-1}s_{4}s_{23} + q^{-1}s_{3}s_{24} + q^{-1}s_{2}s_{34}  + q^{-2}s_{234} + q^{-1}s_{1}s_{1234} + \left(2-q^2\right)s_{2}s_{3}s_{4}\right)s_{234} \\
    &+ \left(q^2s_{12} + \left(q+q^{-1}\right)s_{1}s_{2} +s_{3}s_{4}s_{1234}\right)s_{12} + \left(q^{-2}s_{23}+q^{-1}s_{2}s_{3}+q^{-2}s_{1}s_{4}s_{1234}\right)s_{23} \\
    &+ \left(q^2s_{14} + \left(q+q^{-1}\right)s_{1}s_{4} +s_{2}s_{3}s_{1234}\right)s_{14} + \left(q^{-2}s_{34} + q^{-1}s_{3}s_{4}+q^{-2}s_{1}s_{2}s_{1234}\right)s_{34} \\
    &+ \left(q^{-2}s_{24} +(q^{-1}-q)s_{2}s_{4}+q^{-2}s_{1}s_{3}s_{1234}\right)s_{24} +\left(q^2s_{13} + s_{2}s_{4}s_{1234}\right)s_{13} \\
    &+q^{-2}s_{1234}^2 -s_{2}^2s_{4}^2 -s_{1}^2s_{3}^2   +q^{-1}s_{1}s_{2}s_{3}s_{4}s_{1234} +\left(2+q^{-2}\right)\left(s_{1}^2 + s_{2}^2 + s_{3}^2 + s_{4}^2\right) \\
    &-2q^2-5-4q^{-2}-q^{-4}
\end{align*}
or we can create a closed loop from two triple loops:
\[s_{123}s_{134} = s_{12}s_{14} + s_{23}s_{34} -s_{3}s_{234} -s_{1}s_{124} + s_{1234}s_{13} -\left(q+q^{-1}\right)s_{24}  -s_{2}s_{4}\]

\begin{figure}[ht]
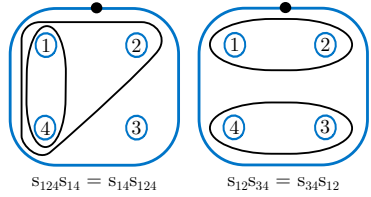

    \centering
    \diagramhh{present}{commutesubset}{0pt}{0pt}{0.25} \diagramhh{present}{commute}{0pt}{0pt}{0.25}
    \caption{Loops commute when the points in one are a subset of the points in the other as in the left image or when the loops are in different parts of the surface as in the right image.}
    \label{figure:commuting}
\end{figure}

Whenever two loops do not intersect they commute. In the case $n=3$ this only happens when one of the loops either contains a single point or all the points. As these loops which contain a single point or all the points are all central this is encoded by adding these loops to the polynomial ring. In the case $n=4$ we still have these central loops but we also have pairs of loops neither of which are central which commute. This happens when the points of one of the loops is a subset of the points of the other loop:
\[s_A s_B = s_B s_A \text{ for }  A \subseteq B; \]
or when the loops are simply in different parts of the surface:
\[s_{12} s_{34} = s_{34} s_{12} \text{ and } s_{23} s_{14} = s_{14} s_{23}.\]

Note that determining whether two loops do not intersect is not a simple as noting that $A \cap B = \emptyset$ as we have
\begin{align*}
    s_{13}s_{24} &= q^{-2}s_{12} s_{34} + q^2 s_{23} s_{14} + q^{-1} s_{1} s_{2} s_{34} + q^{-1} s_{3} s_{4} s_{12} + q s_{1} s_{4} s_{23} + q s_{2} s_{3} s_{14} \\
                 &+ s_{1} s_{234} + s_{4} s_{123} + s_{3} s_{124} + s_{2} s_{134}
                 + s_{1} s_{2} s_{3} s_{4} + \left(q+q^{-1}\right)s_{1234}
\end{align*}
We shall call this relation the \emph{crossing relation}. Whilst in the case $n=4$ we only have a single crossing relation, relations of this type are a general feature of $\SkAlg{q}{\Sigma_{0, n+1}}$ for higher $n$. 

\begin{figure}[ht]
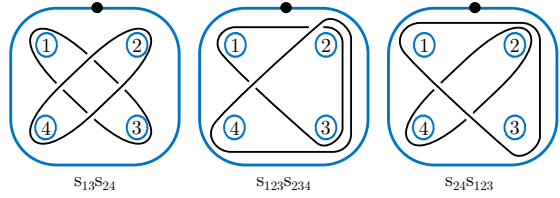

    \centering
    \diagramhh{present}{cross}{0pt}{0pt}{0.25} 
    \diagramhh{present}{2triplelink}{0pt}{0pt}{0.25} 
    \diagramhh{present}{doubletriplecross}{0pt}{0pt}{0.25} 
    \caption{There are three types of relation in $\SkAlg{q}{\Sigma_{0,5}}$ which do not correspond to any relation in $\SkAlg{q}{\Sigma_{0,4}}$. On the left is the leading term of the cross relation, in the middle of a triple link relation and on the right a double and triple relation.}
    \label{figure:crossandtriple}
\end{figure}

For the remaining relations we need to consider afresh the proof of \cref{thm:commutators}. This considers two loops $s_A$ and $s_B$ which are simply linked; that is to say there intersection looks like the intersection of $s_{12}$ and $s_{23}$. When the crossings of $s_A s_B$ are resolved one of the resultant terms is not a simple loop as the loop goes around the outside of the punctures $C = A \cap B$ (see \cref{figure:extra} for examples of the non-simple loops you obtain). In the proof this term is eliminated using $s_B s_A$ to yield the commutator relation. However, note that $s_{12} s_{23} s_{34}$ and $s_{123} s_{234}$ both yield the same non-simple term $s_{1 \overline{23} 4}$. Hence, we get a relation
\begin{align*}
    s_{123}s_{234} &= qs_{12}s_{23}s_{34} -qs_{3}s_{12}s_{234}  -q^{-1}s_{2}s_{34}s_{123} -q^2s_{12}s_{24} - s_{34}s_{13} \\
    &+s_{3}s_{134} + q^{-2}s_{2}s_{124}  +\left(q^{-1} - q\right)s_{2}s_{4}s_{12} + s_{1234}s_{23} + \left(q+q^{-1}\right)s_{14} \\
    &+ q^{-1}s_{2}s_{3}s_{1234}+s_{1}s_{4} 
\end{align*}
and by symmetry similar relations with left hand sides $s_{234} s_{134}$, $s_{134} s_{124}$ and $s_{124} s_{123}$.
Finally, resolving the crossing of $s_{24} s_{123}$ yields the terms $s_{4 \overline{1} 23}$ and $s_{12 \overline{3} 4}$ which can be obtained from $s_{14} s_{123}$ and $s_{34} s_{123}$ respectively. This gives the relation
\begin{align*}
    s_{24}s_{123} &= q^2s_{12}s_{234} + q^{-2}s_{23}s_{124} + qs_{4}s_{12}s_{23} \\
&- qs_{2}s_{4}s_{123} - \left(q^3+q^{-3}\right)s_{134} -q^2s_{1}s_{34} -q^{-2}s_{3}s_{14} -q^2s_{4}s_{13} \\
&+ \left(1-q^2-q^{-2}\right)s_{2}s_{1234} -qs_{1}s_{3}s_{4}
\end{align*}
and by symmetry similar relations with left-hand sides $s_{13}s_{234}$, $s_{24} s_{134}$ and $s_{13}s_{124}$.
The remainder of this section will be dedicated to proving that this set of relations is complete.

\begin{thm}
\label{thm:presentation_fivepunctures}
The Kauffman bracket skein algebra $\SkAlg{q}{\Sigma_{0,5}}$ of the five-punctured sphere with generic $q$ has a presentation as an algebra over $R$ given by the simple loops $s_A$ for all $A \subseteq \{1,2,3,4\}$ with commuting relations
\begin{gather}
s_A \text{ for } |A| = 1 \text{ or } 4 \text{ is central }, \\
s_A s_{123} = s_{123} s_A \text{ for } A \in \{12, 23, 13\}, \\
s_{12} s_{34} = s_{34} s_{12} \text{ and } s_{23} s_{14} = s_{14} s_{23}, 
\end{gather}
commutator relations
\[
    [s_{1}, s_{2}]_q = \left(q^{-2}- q^{2}\right) s_{(A \cup B) \backslash (A \cap B)} + \left(q -q^{-1}\right) \left( s_{A \cap B} s_{A \cup B} + s_{A \backslash (A \cup B)} s_{B \backslash (A \cup B)} \right)
\]
where $A$ and $B$ are set of points with conditions as stated in \cref{thm:commutators}, and relations for the following terms (see Appendix for full list of relations)
\[\begin{array}{ccccr}
    s_{12}s_{23}s_{13} & s_{23}s_{34}s_{24} & s_{34}s_{14}s_{13} & s_{12}s_{14} s_{24}  &\text{(cubic relations)} \\
    s_{12}s_{23}s_{134} & s_{23}s_{34}s_{124} & s_{34}s_{14}s_{123} & s_{12}s_{14} s_{234}  & \text{(cubic relations with triples)} \\
    &s_{12} s_{23} s_{34} s_{14} & s_{13} s_{24} && \text{(quartic and cross relations)} \\
    & s_{123}s_{134} & s_{234} s_{124} && \text{(triple loop relations)} \\
    s_{134} s_{124} & s_{123} s_{124} & s_{123} s_{234} & s_{234} s_{134} & \text{(triple link relations)} \\
    s_{13} s_{234}  & s_{13} s_{124}  & s_{24} s_{123}  & s_{24} s_{134} & \text{(double and triple relations)}
\end{array}
\]
\end{thm}


\begin{rmk}
Looking at the skein relations it is easy to see that given a relation for $s_A s_B$ switching all the over crossing for under crossings gives a relation for $s_B s_A$ which has the same terms with modified coefficients. Furthermore, given a relation reflecting all the terms in the vertical or horizontal plane and again modifying the coefficients will give another relation.
\end{rmk}

\subsection{Term Rewriting Systems and the Diamond Lemma}

In order to prove \cref{thm:presentation_fivepunctures} we shall use a Term Rewriting System (TRWS).

\begin{defn}
An \emph{abstract rewriting system} is a set A together with a binary relation \(\to\) on A called the \emph{reduction relation} or \emph{rewrite relation}. 
\begin{enumerate}
    \item It is \emph{terminating} if there are no infinite chains \(a_0 \to a_1 \to a_2 \to \dots\). 
    \item It is \emph{locally confluent} if for all \(y \xleftarrow{} x \xrightarrow{} z\) there exists an element \(y \downarrow z \in A\) such that there are paths \(y \to \dots \to (y \downarrow z) \) and \(z \to \dots \to (y \downarrow z) \). 
    \item It is \emph{confluent} if for all \(y \xleftarrow{} \dots \xleftarrow{} x \xrightarrow{} \dots \xrightarrow{} z\) there exists an element \(y \downarrow z \in A\) such that there are paths \(y \to \dots \to (y \downarrow z) \) and \(z \to \dots \to (y \downarrow z) \).
\end{enumerate}
In a terminating confluent abstract rewriting system an element \(a \in A\) will always reduce to a unique reduced expression regardless of the order of the reductions used.
\end{defn}
    
The \emph{diamond lemma} (or Newman's lemma) for abstract rewriting systems states that a terminating abstract rewriting system is confluent if and only if it is locally confluent. Bergman's diamond lemma is an application to ring theory of the diamond lemma for abstract rewriting systems. The definitions given in this section can be found \cite[Section~1]{diamond}. 


Let \(\mathcal{R}\) be a commutative ring with multiplicative identity and \(X\) be an alphabet (a set of symbols from which we form words).  

\begin{defn}
A \emph{reduction system} \(S\) consists of term rewriting rules \(\sigma : W_{\sigma} \mapsto f_{\sigma}\) where \(W_{\sigma} \in \langle X \rangle\) is a word in the alphabet \(X\) and \(f_{\sigma} \in \mathcal{R} \langle X \rangle\) is a linear combination of words. A \emph{\(\sigma\)-reduction} \(r_{\sigma}(T)\) of an expression \(T \in \mathcal{R}\langle X \rangle\) is formed by replacing an instance of \(W_{\sigma}\) in \(T\) with \(f_{\sigma}\). A \emph{reduction} is a \(\sigma\)-reduction for some \(\sigma \in S\). If there are no possible reductions for an expression we say it is \emph{irreducible}.
\end{defn}

\begin{defn}
The five-tuple \((\sigma, \tau, A, B, C)\) with \(\sigma, \tau \in S\) and \(A,B, C \in \langle X \rangle\) is an \emph{overlap ambiguity} if \(W_{\sigma} = AB\) and \(W_{\tau} = BC\) and an \emph{inclusion ambiguity} if \(W_{\sigma} = B\) and \(W_{\tau} = ABC\).
These ambiguities are \emph{resolvable} if reducing \(ABC\) by starting with a \(\sigma\)-reduction gives the same result as starting with a \(\tau\)-reduction. 
\end{defn} 

\begin{ex}
    Suppose we have an alphabet \(X = \langle a,b \rangle\) and reduction system \(S=\{\,\sigma : ab \mapsto ba, \tau : ba \mapsto a\,\}\). Then \(r_{\sigma}(T) = aba + a\) is a \(\sigma\)-reduction of \(T= aab + a\). We also have an overlap ambiguity \((\sigma, \tau, a,b,a)\) which is resolvable as \( aba \xmapsto{r_{\sigma}} ba^2 \xmapsto{r_{\tau}} a^2 \) gives the same expression as \(aba \xmapsto{r_{\tau}} a^2\).
\end{ex}

\begin{defn}
A \emph{semigroup partial ordering} \(\leq\) on \(\langle X \rangle\) is a partial order such that \(B \leq B' \) implies that \(ABC \leq AB'C\) for all words \(A,B, B', C\). It is \emph{compatible with the reduction system} \(S\) if  for all \(\sigma \in S\) the monomials in \(f_{\sigma}\) are less than \(W_{\sigma}\).
\end{defn}

\begin{defn}
A reduction system \(S\) satisfies the \emph{descending chain condition} or is \emph{terminating} if for any expression \(T \in \mathcal{R} \langle X \rangle\) any sequence of reductions terminates in a finite number of reductions with an irreducible expression. 
\end{defn}

\begin{lem}[The Diamond Lemma {\cite[Theorem~1.2]{diamond}}]
Let \(S\) be a reduction system for \(\mathcal{R} \langle X \rangle\) and let \(\leq\) be a semigroup partial ordering on \(\langle X \rangle\) compatible with the reduction system \(S\) with the descending chain condition. The following are equivalent:
\begin{enumerate}
    \item All ambiguities in \(S\) are resolvable (\(S\) is \emph{locally confluent});
    \item Every element \(a \in \mathcal{R} \langle X \rangle\) can be reduced in a finite number of reductions to a unique expression \(r_S(a)\) (\(S\) is \emph{confluent});
    \item The algebra \(K = \mathcal{R} \langle X \rangle / I\), where \(I\) is the two-sided ideal of \(\mathcal{R} \langle X \rangle\) generated by the elements \((W_{\sigma} - f_{\sigma})\), can be identified with the \(\mathcal{R}\)-algebra \(k \langle X \rangle_{\mathrm{irr}}\) spanned by the \(S\)-irreducible monomials of \(\langle X \rangle\) with multiplication given by \(a \cdot b = r_S(ab)\). These \(S\)-irreducible monomials are called a Poincare--Birkhoff--Witt basis of \(K\).
\end{enumerate}
\end{lem}

\subsection{Linear Basis for $\SkAlg{q}{\Sigma_{0,5}}$}

In this subsection we will construct a locally consistant, terminating term rewriting system from the relations stated in \cref{thm:presentation_fivepunctures} for $\SkAlg{q}{\Sigma_{0,5}}$. This will give a linear basis for $\SkAlg{q}{\Sigma_{0,5}}$ over the commutative ring $\mathcal{R}$ $=$ $R[s_{1}, s_{2}, s_{3}, s_{4}, s_{1234}]$ which we will then use to prove \cref{thm:presentation_fivepunctures}.

The obvious approach would be to take each relation from \cref{thm:presentation_fivepunctures} (except the first as this is implicit in our choice of base ring) and turn it into a rewriting rule. Whilst this approach works well for the relations whose left hand side is pairwise (is the product of two loops) adding the non-pairwise relations leads to an infinite TRWS which would be difficult to show that it was consistent\footnote{In \cite{Cooke18} the consistency of the resulting infinite system was proven by induction for the case fo $\SkAlg{q}{\Sigma_{0,4}}$ however this method does not scale well.}. 

\begin{ex}
Assume we construct a TRWS with all the pairwise relations for \cref{thm:presentation_fivepunctures} and the non-pairwise cubic relation for $s_{12} s_{23} s_{13}$. One of the ambiguities for this TRW is the word $s_{23} s_{12} s_{23} s_{13}$ which on the one hand can be reduced using the cubic relation and on the other hand using the commutator relation for $s_{23} s_{12}$. Using the commutator relation gives as its leading term $s_{12} s_{23}^2 s_{13}$ which cannot be reduced further. However, the term $s_{12} s_{23}^2 s_{13}$ does not arise if you start reducing with the cubic relation, and thus the system is not consistent. In order to make a consistent system you would need to add a rewriting rule for  $s_{12} s_{23}^2 s_{13}$. Considering $s_{23} s_{12} s_{23}^2 s_{13}$ using the same argument as above leads to the conclusion you also need $s_{12} s_{23}^3 s_{13}$ and indeed $s_{12} s_{23}^n s_{13}$ for all $n \in \mathbb{Z}_{>0}$ in the TRWS: that is you end up with an infinite TRWS.
\end{ex}

In order to avoid an infinite TRWS we add some extra generators so that all the relations are pairwise. 
In order to make the cubic relations pairwise we add the generators $s_{1 \overline{2} 3}$, $s_{2 \overline{3} 4}$, $s_{3 \overline{4} 1}$ and $s_{4 \overline{1} 2}$ and for the quartic relation we add the generators  $s_{1 \overline{23} 4}$, $s_{2 \overline{34} 1}$, $s_{3 \overline{14} 2}$ and $s_{4 \overline{12} 3}$. Finally, we also need the generators $s_{12,34}$ and $s_{23,14}$ which are just the products $s_{12} s_{34}$ and $s_{12} s_{34}$ respectively considered as a generator.
We shall order these extended generators and place them into five groups as follows:
\begin{description}[leftmargin=!,labelwidth=\widthof{\bfseries{\textrm{III}.}}]
\item[\textrm{I}.] $s_{12}$ $s_{23}$ $s_{34}$ $s_{14}$
\item[\textrm{II}.] $s_{1 \overline{2} 3}$ $s_{2 \overline{3} 4}$ $s_{3 \overline{4} 1}$ $s_{4 \overline{1} 2}$ $s_{12, 34}$ $s_{23,14}$
\item[\textrm{III}.] $s_{1 \overline{23} 4}$, $s_{2 \overline{34} 1}$, $s_{3 \overline{14} 2}$          
\item[\textrm{IV}.] $s_{13}$ $s_{24}$
\item[\textrm{V}.] $s_{123}$ $s_{234}$ $s_{134}$ $s_{123}$ 
\end{description}

\begin{figure}[ht]
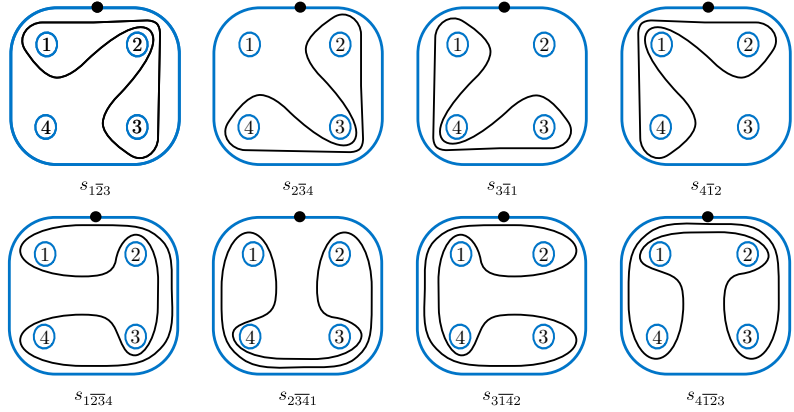

    \centering
    \diagramhh{present}{extra}{0pt}{0pt}{0.25} 
    \caption{Extra generators we add to make all the relations in $\SkAlg{q}{\Sigma_{0,5}}$ pairwise.}
    \label{figure:extra}
\end{figure}

\begin{defn}
    By convention if a point $x \not\in A$ then the loop $s_A$ passes on the inside of the point $x$ if the points are in a circle or below the point $x$ if the points are in a line. If instead a loop passes on the outside (or above) of the point $x$ we refer to $x$ as a \emph{double point}\footnote{In the handlebody decomposition of the surface, the handle associated to a double point will intersect the loop in two arcs.}. The extra generators have double points denoted by $\overline{x}$.
\end{defn}

Using these generators the non-pairwise relations the cubic relations become
\begin{align*}
s_{1 \overline{2} 3}s_{13}&= \left(q^{-1}s_{1}s_{23} + qs_{3}s_{12} + s_{1}s_{2}s_{3} + s_{123}\right)s_{123}  + q^{-1}\big(s_{2}s_{3}\\ &+ q^{-1} s_{23}\big)s_{23}  + q(s_{1}s_{2} + q s_{12})s_{12} + s_{1}^2 + s_{2}^2 + s_{3}^2 - \left(q + q^{-1}\right)^2,\\
s_{1 \overline{2} 3}s_{134}&=  \left(q^{-1}s_{1}s_{23} + qs_{3}s_{12}  + s_{1}s_{2}s_{3} + s_{123}\right)s_{1234} +q(s_{1}s_{2} + qs_{12})s_{124} \\
&+  q^{-1}\left(s_{2}s_{3} + q^{-1}s_{23}\right)s_{234} + s_{1}s_{14} + s_{2}s_{24}  + s_{3}s_{34} + \left(q + q^{-1}\right)s_{4},
\end{align*}
and their symmetries, and the quartic relation becomes the two relations
\begin{align*}
    s_{12}s_{2 \overline{34} 1}&= \left(q^{-1}s_{1}s_{234} + qs_{2}s_{134} + s_{1}s_{2}s_{34} + s_{1234}\right) s_{1234} + q(s_{1}s_{34} + qs_{134})s_{134} \\
&+ q^{-1}\left(s_{2}s_{34}s_{234}  + q^{-1}s_{234}\right)s_{234} + s_{1}^2 + s_{2}^2 + s_{34}^2 - \left(q + q^{-1}\right)^2,\\
s_{23}s_{3 \overline{14} 2}&= \left(qs_{3}s_{124} + q^{-1}s_{2}s_{134} + s_{2}s_{3}s_{14} +s_{1234} \right) s_{1234} + q(s_{2}s_{14}+ qs_{124})s_{124} \\
&+ q^{-1}\left(s_{3}s_{14} + q^{-1}s_{134}\right)s_{134} + s_{2}^2 + s_{3}^2 + s_{14}^2 - \left(q + q^{-1}\right)^2.
\end{align*}

Unfortunately, adding extra generators massively increases the number of relations, but these extra relations can be generated from the original relations in a manner which will now be described.

Firstly we have the relations which relate the new generators to the simple generators which we shall call the \emph{generator generating relations}. For the generators which consist of two disjoint loops these are trivial:
\[s_{12}s_{34} \mapsto s_{12 , 34}  \quad s_{23}s_{14} \mapsto s_{23 , 14} \]
For the generators with a single double point the relations have the form:
\[s_{12}s_{23} \mapsto q^{-1}s_{1 \overline{2} 3} + qs_{13} + s_{1}s_{3} + s_{2}s_{123} \]
For the generators with two double points the relations have the form:
\begin{align*}
    s_{12}s_{2 \overline{3} 4}&\mapsto q^{-1}s_{1 \overline{23} 4} + s_{1}s_{4} + qp^{-1}\left(s_{34}s_{13} - p^{-1}s_{14} - s_{1}s_{4} - s_{3}s_{134}\right) \\
    &+ p^{-1}s_{2}\left(s_{34}s_{123} - p^{-1}s_{124} - s_{4}s_{12} - s_{3}s_{1234}\right) \\
    s_{12}s_{23 , 14}&\mapsto qs_{13}s_{14} + q^{-1}s_{1 \overline{2} 3}s_{14} + s_{2}s_{123}s_{14} + s_{1}s_{3}s_{14}
\end{align*}
where $p = q$ except when finding the coefficients for the symmetric relations it does not invert. For the full list of relations see the code\footnote{Code available at \href{https://github.com/jcooke848/Askey-Wilson-Algebras-as-Skein-Code.git}{https://github.com/jcooke848/Askey-Wilson-Algebras-as-Skein-Code.git} \label{code}}.
 
Now let $\{bc \mapsto r_{bc}\}$ be a relation in \cref{thm:presentation_fivepunctures} excluding the commuting and the commutator relations, and let $\{ab \mapsto r_{ab}\}$ be a relation in the generator generating relations. Hence, we have an equality
\[ 
    r_{ab}c = abc = a r_{bc}
\]
and furthermore $r_{ab}c$ contains a new generator multiplied by $c$ on the right. Rearrange with respect to $mc$ where $m$ is the largest new generator in $r_{ab}$ and turn this equation into a reduction rule for $mc$. 

After generating these new relations for all the original $\{bc \mapsto r_{bc}\}$ relations, iterate by the generating the relations where $\{bc \mapsto r_{bc}\}$ is one of these newly generated relations. This generates all the new relations including the extra generators apart from the commuting and commutator relations involving the extra generators. To generate these consider monomials $ba$ where $b>a$, one of them is a new generator and $ba$ is not reducible using the generator generating relations or any other relations we have generated so far. Consider the monomial $\prod b_i a_j$ where $\prod b_i$ generates $b$ and $a_j$ generates $a$ as the leading term using the generator generating relations. We can generate a relation if:
\begin{itemize}
    \item All the $b_i$ terms are larger than all the $a_j$ terms by considering $x =\prod b_i a_j$
    \item All the $b_i$ terms are smaller than all the $a_j$ terms by considering $x =\prod a_j b_i$
\end{itemize}
Take $x$ and apply the generator generating relations to $\prod b_i$ and $\prod a_j$ separatly, then order the result using the commutator and commuting relations and rearrange the result to obtain a relation for $ba$. This generates a term rewriting system with 241 relations which we shall call $S_{\mathcal{B}}$ we shall now show is confluent. 

 \begin{prop}
 \label{prop:ambiguities}
All ambiguities in the term rewriting system $S_{\mathcal{B}}$ are resolvable.
 \end{prop}

\begin{proof}
    As all reductions in $S_{\mathcal{B}}$ are pairwise, it is sufficient to check $abc$ where $a,b,c$ are basis elements such that $\{\;ab\mapsto r_{ab},\; bc \mapsto r_{bc}\;\}$ are reductions in $S_{\mathcal{B}}$. As there are 241 relations, there are a very large number of such ambiguities so we have used a computer to check these\cref{code}.
\end{proof}

To use the diamond lemma for ring theory we now need to prove that the system $S_{\mathcal{B}}$ terminates. We shall do this by constructing a partial order which is compatible with the term rewriting system. This partial order will be constructed by chaining together three different partial orders. The first ordering is ordering by \emph{reduced degree} \cite[Section~15]{ReducedOrder}:

\begin{defn}
Give the letters of the finite alphabet \(X\) an ordering \(x_1 \leq \dots \leq x_N\). Any word \(W\) of length \(n\) can be written as \(W = x_{i_1} \dots x_{i_n}\) where \(x_{i_j} \in X\). An \emph{inversion} of \(W\) is a pair \(k \leq l\) with \(x_{i_{k}} \geq x_{i_{l}}\) i.e.\ a pair with letters in the incorrect order. The number of inversions of \(W\) is denoted \(|W|\).
\end{defn}
\begin{defn}
Any expression \(T\) can be written as a linear combination of words \(T= \sum c_l W_l\). Define \(\rho_n(T):= \sum_{\operatorname{length}(W_l)=n, c_l \neq 0} |W_l|\). The \emph{reduced degree of \(T\)} is the largest \(n\) such that \(\rho_n(T) \neq 0\).
\end{defn}
\begin{defn}
Under the \emph{reduced degree ordering}, \(T \leq S\) if 
\begin{enumerate}
\item The reduced degree of  \(T\) is less than the reduced degree of \(S\), or
\item The reduced degree of \(T\) and \(S\) are equal, but  \(\rho_n(T) \leq \rho_n(S)\) for maximal nonzero \(n\). 
\end{enumerate}
\end{defn}

The second ordering is by total degree.

\begin{defn}
    The \emph{total degree} of $T \in k<X>$ is the maximal degree of its monomials. Under the \emph{total degree ordering} $T \leq S$ if the total degree of $T$ is less than or equal to the total degree of $S$. 
    \end{defn}

\begin{defn}
Let $s$ be one of the extended generators. The \emph{degree} of $s$ is 
\[\operatorname{degree}(s) = (\text{number of points inside loop}) + 2(\text{number of double points}),\]
so for example $s_{1 \overline{2} 3}$ has degree $4$. 
The degree of a monomial is the sum of the degree of its terms. 
\end{defn}

The final partial order is a partial order based on the notion on how near are the loops that make up a monomial in the ordered list of loops. 

\begin{defn}
Let $\operatorname{group}(s)$ denote the group (1-5) which the loop $s$ is in. For a monomial $m = \prod_{i \in I} s_i$ we define
\[\operatorname{nearness}(m) = \sum_{i,j \in I} | \operatorname{group}(s_i) - \operatorname{group}(s_j) |.\]
Let $\{m_i\}$ and $\{n_j\}$ be the maximal total degree monomials of expressions $T$ and $S$ respectively. Under the \emph{group distance ordering} $T \leq S$ if 
\[\max_i(\text{number of distinct loops in }m_i) < \max_j(\text{number of distinct loops in }m_j)\]
or these maxima are equal and 
\[\sum_i \operatorname{nearness}(m_i) > \sum_j \operatorname{nearness}(n_j).\]
\end{defn}

We now combine these three partial orders to obtain a single partial order of $\mathcal{R}\langle X\rangle$. 

\begin{defn}
\label{defn:ordering}
Let $m, n \in \mathcal{R}\langle X\rangle$. We define $m \leq n$ if one of the following conditions is satisfied:
\begin{enumerate}
    \item $m<n$ with respect to the reduced degree ordering
    \item $m = n$ under the reduced degree ordering and $m<n$ with respect to the total degree ordering
    \item $m = n$ under the reduced degree ordering, they have the same total degree and $m<n$ with respect to the group distance ordering.
\end{enumerate}
\end{defn}

\begin{lem}
\label{lem:partialorder}
The term rewriting system $S_{\mathcal{B}}$ is compatible with the ordering defined in \cref{defn:ordering}.  
\end{lem}
\begin{proof}
    This requires that for every rewriting rule $\{\sigma\mapsto r_{\sigma}\}$, $\sigma < m$ for all monomials $m$ in $r_{\sigma}$. This can easily be checked using the code \footref{code}. 
\end{proof}

We can now apply the diamond lemma for ring theory.

\begin{thm}
The term-rewriting system $S_{\mathcal{B}}$ is confluent and hence the reduced monomials form a linear basis for the associated algebra $\mathcal{B}$. 
\end{thm}

\begin{proof}
    The ambiguities are resolvable by \cref{prop:ambiguities} and by \cref{lem:partialorder} there is a compatible term rewriting system so the term rewriting system terminates and hence we can apply the diamond lemma for ring theory. 
\end{proof}

If we filter $\mathcal{B}$ by degree, we have a surjective filtered algebra homomorphism
\[
    \phi: \mathcal{B} \to \SkAlg{q}{\Sigma_{0,5}}.
\]
We now need to prove that $\phi$ is an isomorphism and thus that $\mathcal{B}$ is a presentation for $\SkAlg{q}{\Sigma_{0,5}}$. To do this we shall compute the Hilbert series of $\mathcal{B}$ and show it is the same as the Hilbert series for $\SkAlg{q}{\Sigma_{0,5}}$ which we have already computed in \cref{thm:dimensions}. 
\begin{prop}
The Hilbert series of $\mathcal{B}$ is
\[\frac{t^4 - 2t^3 + 4t^2 -2t +1}{(1-t^2)^3(1-t)^6}\]
\end{prop}

\begin{proof}
In order to compute the Hilbert series, we first consider what are the conditions on a monomial $m = \prod_{i \in I} s_i$ if it is reduced and therefore is in the vector space basis. Firstly, note that there is a relation $yx$ for any $y>x$ so the $s_i$ must be ordered. Furthermore, there is a relation $xy$ between any two loops in the same group; hence, 
\[ m = s_{\textrm{I}}^{\alpha} s_{\textrm{II}}^{\beta} s_{\textrm{III}}^{\gamma} s_{\textrm{IV}}^{\delta} s_{\textrm{V}}^{\epsilon} \]
where for example $s_{\textrm{I}}$ is a loop in group 1 and $\alpha, \beta, \gamma, \delta, \epsilon \in \mathbb{Z}_{\geq 0}$. 
Also note that as all the relations are pairwise we only need to concern ourselves with neighbouring terms in the monomial. 

The Hilbert series of $\{\;x^m \mathrel{|} m \in \mathbb{Z}_{>0},\; \deg(x) = n\;\}$ is $\frac{t^n}{1-t^n}$, 
so the Hilbert series when there is only a $s_{\textrm{I}}$ loop is 
\[
1 + \frac{4t^2}{1-t^2}.
\]
Given a $s_{\textrm{II}}$ loop there are two possible choices for $s_{\textrm{I}}$; hence the Hilbert series for $s_\textrm{I}^{\alpha} s_{\textrm{II}}^{\beta}$ is
\[
    1 + \frac{4t^2}{1-t^2} + \frac{6t^4}{1-t^4}\left(1 + \frac{2t^2}{1-t^2}\right)
    = 1 + \frac{4t^2}{1-t^2} + \frac{6t^4}{(1-t^2)^2}
\]
Given a $s_{\textrm{III}}$ loop there are three choices of $s_{\textrm{II}}$ or if there is no $s_{\textrm{II}}$ loop there are three choices for $s_{\textrm{I}}$; hence the Hilbert series for $s_{\textrm{I}}^{\alpha} s_{\textrm{II}}^{\beta} s_{\textrm{III}}^{\gamma}$ is
\[
  1 + \frac{4t^2}{1-t^2} + \frac{6t^4}{(1-t^2)^2} + \frac{4t^6}{1-t^6}\left( \frac{3t^4}{(1-t^2)^2} + \frac{3t^2}{1-t^2} + 1 \right) = \frac{1-t^8}{(1-t^2)^4}
\]
Given a $s_{\textrm{IV}}$ loop there are no choices for $s_{\textrm{III}}$ (the relations are derived from the cubic relation) and two choices for $s_{\textrm{V}}$. There are four choices for $s_{\textrm{II}}$ and if there is no $s_{\textrm{II}}$ loop there is a free choice of $s_{\textrm{I}}$; hence the Hilbert series for $s_{\textrm{I}}^{\alpha} s_{\textrm{II}}^{\beta} s_{\textrm{III}}^{\gamma} s_{\textrm{IV}}^{\epsilon} s_{\textrm{V}}^{\delta}$  assuming $\delta \neq 0$ is 
\[
\frac{2t^2}{1-t^2}\left(\frac{2t^3}{1-t^3}+1\right)\left( \frac{4t^4}{(1-t^2)^2} + \frac{4t^2}{1-t^2} + 1\right)
\]

Finally, we assume there is no $s_{\textrm{IV}}$ loop but fix a $s_{\textrm{V}}$ loop. There are two choices for $s_{\textrm{III}}$, if there is no $s_{\textrm{III}}$ loop there are five choices for $s_{\textrm{II}}$ and if there is only a $s_{\textrm{I}}$ loop then there is a free choice. Hence, the Hilbert series for $s_{\textrm{I}}^{\alpha} s_{\textrm{II}}^{\beta} s_{\textrm{III}}^{\gamma} s_{\textrm{V}}^{\epsilon}$  assuming $\epsilon \neq 0$ is 
\[
\frac{4t^3}{1-t^3} \left(\frac{2t^6}{(1-t^2)^3} + \frac{5t^4}{(1-t^2)^2} + \frac{4t^2}{1-t^2} + 1 \right) = \frac{4t^3}{1-t^3} \frac{(1-t^4)}{(1-t^2)^4}.
\]
Combining these cases gives a Hilbert series for $\mathcal{B}$ of 
\begin{gather*}
\frac{1-t^8}{(1-t^2)^4} + \frac{2t^2}{1-t^2}\left(\frac{2t^3}{1-t^3}+1\right)\left( \frac{4t^2}{(1-t^2)^2} + 1\right) +  \frac{4t^3}{1-t^3} \frac{(1-t^4)}{(1-t^2)^4} \\
= \frac{t^4 - 2t^3 + 4t^2 -2t +1}{(1-t^2)^3(1-t)^6}
\end{gather*}
as required. 
\end{proof}

This means we have an isomorphism
\[\phi: \SkAlg{q}{\Sigma_{0,5}} \to \mathcal{B}\]
and therefore we have a presentation for $\SkAlg{q}{\Sigma_{0,5}}$. Finally, we will remove the extra generators to reduce the presentation to that in \cref{thm:presentation_fivepunctures} thus proving this theorem.
\begin{proof}[Proof of \cref{thm:presentation_fivepunctures}]
    We have that $\mathcal{B}$ is a presentation for $\SkAlg{q}{\Sigma_{0,5}}$ as the algebras are isomorphic. Using generator generating relations we can eliminate the non-simple loop generators of $\mathcal{B}$. The relations in $\mathcal{B}$ between simple loops are the same as in $\mathcal{A}$ and it is straightforward to check that the new cubic and quartic relations reduce (see code). The extra relations in $\mathcal{B}$ which were generated from these relations reduce by how they were defined. Hence, we can reduce the presentation $\mathcal{B}$ thus concluding the proof. 
\end{proof}

\begin{appendices}
    \crefalias{section}{appsec}
    \section{Appendix}
\label{app:1}
In this appendix we list explicitly the full list of relations for the presentation $\SkAlg{q}{\Sigma_{0,5}}$ which is given in Theorem \ref{thm:presentation_fivepunctures}. 

\subsection{Commuting}
\begin{align*}
    s_{124}s_{24} &=  s_{24}s_{124} &
s_{124}s_{14} &=  s_{14}s_{124} &
s_{134}s_{14} &=  s_{14}s_{134}\\
s_{134}s_{13} &=  s_{13}s_{134} &
s_{124}s_{12} &=  s_{12}s_{124} &
s_{234}s_{34} &=  s_{34}s_{234}\\
s_{234}s_{24} &=  s_{24}s_{234} &
s_{134}s_{34} &=  s_{34}s_{134} &
s_{234}s_{23} &=  s_{23}s_{234}\\
s_{123}s_{13} &=  s_{13}s_{123} &
s_{123}s_{12} &=  s_{12}s_{123} &
s_{123}s_{23} &=  s_{23}s_{123}\\
s_{14}s_{23} &=  s_{23}s_{14} &
s_{34}s_{12} &=  s_{12}s_{34} &
\end{align*}

\subsection{Commutators} 
\begin{align*}
s_{23}s_{12} &=  \left(1-q^2\right)s_{1}s_{3} + \left(q^{-1} - q^3\right)s_{13} + \left(1-q^2\right)s_{2}s_{123} + q^2s_{12}s_{23}\\
s_{34}s_{23} &=  \left(1-q^2\right)s_{2}s_{4} + \left(q^{-1} - q^3\right)s_{24} + \left(1-q^2\right)s_{3}s_{234} + q^2s_{23}s_{34}\\
s_{14}s_{12} &=  \left(1-q^{-2}\right)s_{2}s_{4} + \left(q - q^{-3}\right)s_{24} + \left(1-q^{-2}\right)s_{1}s_{124} + q^{-2}s_{12}s_{14}\\
s_{14}s_{34} &=  \left(1-q^2\right)s_{1}s_{3} + \left(q^{-1} - q^3\right)s_{13} + \left(1-q^2\right)s_{4}s_{134} + q^2s_{34}s_{14}\\
s_{13}s_{12} &=  \left(1-q^{-2}\right)s_{2}s_{3} + \left(q - q^{-3}\right)s_{23} + \left(1-q^{-2}\right)s_{1}s_{123} + q^{-2}s_{12}s_{13}\\
s_{13}s_{23} &=  \left(1-q^2\right)s_{1}s_{2} + \left(q^{-1} - q^3\right)s_{12} + \left(1-q^2\right)s_{3}s_{123} + q^2s_{23}s_{13}\\
s_{13}s_{14} &=  \left(1-q^2\right)s_{3}s_{4} + \left(q^{-1} - q^3\right)s_{34} + \left(1-q^2\right)s_{1}s_{134} + q^2s_{14}s_{13}\\
s_{24}s_{12} &=  \left(1-q^2\right)s_{1}s_{4} + \left(q^{-1} - q^3\right)s_{14} + \left(1-q^2\right)s_{2}s_{124} + q^2s_{12}s_{24}\\
s_{24}s_{23} &=  \left(1-q^{-2}\right)s_{3}s_{4} + \left(q - q^{-3}\right)s_{34} + \left(1-q^{-2}\right)s_{2}s_{234} + q^{-2}s_{23}s_{24}\\
s_{24}s_{34} &=  \left(1-q^2\right)s_{2}s_{3} + \left(q^{-1} - q^3\right)s_{23} + \left(1-q^2\right)s_{4}s_{234} + q^2s_{34}s_{24}\\
s_{24}s_{14} &=  \left(1-q^{-2}\right)s_{1}s_{2} + \left(q - q^{-3}\right)s_{12} + \left(1-q^{-2}\right)s_{4}s_{124} + q^{-2}s_{14}s_{24}\\
s_{13}s_{34} &=  \left(1-q^{-2}\right)s_{1}s_{4} + \left(q - q^{-3}\right)s_{14} + \left(1-q^{-2}\right)s_{3}s_{134} + q^{-2}s_{34}s_{13}\\
s_{123}s_{34} &=  \left(1-q^{-2}\right)s_{3}s_{1234} + \left(1-q^{-2}\right)s_{4}s_{12} + \left(q - q^{-3}\right)s_{124} + q^{-2}s_{34}s_{123}\\
s_{123}s_{14} &=  \left(1-q^2\right)s_{1}s_{1234} + \left(1-q^2\right)s_{4}s_{23} + \left(q^{-1} - q^3\right)s_{234} + q^2s_{14}s_{123}\\
s_{234}s_{12} &=  \left(1-q^2\right)s_{2}s_{1234} + \left(1-q^2\right)s_{1}s_{34} + \left(q^{-1} - q^3\right)s_{134} + q^2s_{12}s_{234}\\
s_{234}s_{14} &=  \left(1-q^{-2}\right)s_{4}s_{1234} + \left(1-q^{-2}\right)s_{1}s_{23} + \left(q - q^{-3}\right)s_{123} + q^{-2}s_{14}s_{234}\\
s_{134}s_{23} &=  \left(1-q^2\right)s_{3}s_{1234} + \left(1-q^2\right)s_{2}s_{14} + \left(q^{-1} - q^3\right)s_{124} + q^2s_{23}s_{134}\\
s_{134}s_{12} &=  \left(1-q^{-2}\right)s_{1}s_{1234} + \left(1-q^{-2}\right)s_{2}s_{34} + \left(q - q^{-3}\right)s_{234} + q^{-2}s_{12}s_{134}\\
s_{124}s_{23} &=  \left(1-q^{-2}\right)s_{2}s_{1234} + \left(1-q^{-2}\right)s_{3}s_{14} + \left(q - q^{-3}\right)s_{134} + q^{-2}s_{23}s_{124}\\
s_{124}s_{34} &=  \left(1-q^2\right)s_{4}s_{1234} + \left(1-q^2\right)s_{3}s_{12} + \left(q^{-1} - q^3\right)s_{123} + q^2s_{34}s_{124}\\
\end{align*}

\subsection{Cubic Relations}
\begin{align*}
    s_{34}s_{14}s_{13} &=  -\left(q + 2q^{-1} + q^{-3}\right)+q^{-1}s_{4}^2+q^{-1}s_{3}^2+q^{-1}s_{1}^2 \\
&+ s_{3}s_{4}s_{34} + q^{-2}s_{1}s_{4}s_{14} + s_{1}s_{3}s_{13} + q^{-1}s_{1}s_{3}s_{4}s_{134} + qs_{34}^2 + s_{1}s_{34}s_{134} \\
&+ q^{-3}s_{14}^2 + q^{-2}s_{3}s_{14}s_{134} + qs_{13}^2 + s_{4}s_{13}s_{134} + q^{-1}s_{134}^2\\
s_{12}s_{14}s_{24} &=  -\left(q^3 + 2q + q^{-1}\right)+qs_{4}^2+qs_{2}^2+qs_{1}^2 \\
&+ s_{1}s_{2}s_{12} + q^2s_{1}s_{4}s_{14} + s_{2}s_{4}s_{24} + qs_{1}s_{2}s_{4}s_{124} + q^{-1}s_{12}^2 + s_{4}s_{12}s_{124} \\
&+ q^3s_{14}^2 + q^2s_{2}s_{14}s_{124} + q^{-1}s_{24}^2 + s_{1}s_{24}s_{124} + qs_{124}^2\\
s_{12}s_{23}s_{13} &=  -\left(q + 2q^{-1} + q^{-3}\right)+q^{-1}s_{3}^2+q^{-1}s_{2}^2+q^{-1}s_{1}^2 \\
&+ s_{1}s_{2}s_{12} + q^{-2}s_{2}s_{3}s_{23} + s_{1}s_{3}s_{13} + q^{-1}s_{1}s_{2}s_{3}s_{123} + qs_{12}^2 + s_{3}s_{12}s_{123} \\
&+ q^{-3}s_{23}^2 + q^{-2}s_{1}s_{23}s_{123} + qs_{13}^2 + s_{2}s_{13}s_{123} + q^{-1}s_{123}^2\\
s_{23}s_{34}s_{24} &=  -\left(q + 2q^{-1} + q^{-3}\right)+q^{-1}s_{4}^2+q^{-1}s_{3}^2+q^{-1}s_{2}^2 \\
&+ s_{2}s_{3}s_{23} + q^{-2}s_{3}s_{4}s_{34} + s_{2}s_{4}s_{24} + q^{-1}s_{2}s_{3}s_{4}s_{234} + qs_{23}^2 + s_{4}s_{23}s_{234} \\
&+ q^{-3}s_{34}^2 + q^{-2}s_{2}s_{34}s_{234} + qs_{24}^2 + s_{3}s_{24}s_{234} + q^{-1}s_{234}^2\\
\end{align*}

\subsection{Cubic Relations with Triples} 
\begin{align*}
 s_{12}s_{14}s_{234} &= \left(q^2 + 1\right)s_{3}+qs_{1}s_{2}s_{4}s_{1234}-q^2s_{1}^2s_{3} +  s_{4}s_{1234}s_{12} + qs_{2}s_{23} \\
 &+ qs_{4}s_{34} + q^2s_{2}s_{1234}s_{14} -q^3s_{1}s_{13} +  s_{1}s_{1234}s_{24} +  s_{2}s_{4}s_{234} + qs_{1234}\\
 &+s_{124} +  s_{1}s_{12}s_{23} + q^{-1}s_{12}s_{123} + q^2s_{1}s_{34}s_{14} + q^3s_{14}s_{134} + q^{-1}s_{24}s_{234}\\
  s_{34}s_{14}s_{123} &= \left(1+q^{-2}\right)s_{2}+\left(1-q^2-q^{-2}\right)s_{2}s_{4}^2+q^{-1}s_{1}s_{3}s_{4}s_{1234} + q^{-1}s_{1}s_{12} \\
 &+ q^{-1}s_{3}s_{23} +  s_{1}s_{1234}s_{34} + q^{-2}s_{3}s_{1234}s_{14} +  s_{4}s_{1234}s_{13} \\
 &+ \left(-q^3+q^{-1}-q^{-3}\right)s_{4}s_{24} +  s_{1}s_{3}s_{123} + \left(-q^2+1\right)s_{3}s_{4}s_{234} + q^{-1}s_{1234}s_{134} \\
 &+ q^{-2}s_{4}s_{12}s_{14} + q^2s_{4}s_{23}s_{34} + qs_{34}s_{234} + q^{-3}s_{14}s_{124} + qs_{13}s_{123}\\
  s_{23}s_{34}s_{124} &= q^{-1}s_{2}s_{3}s_{4}s_{1234}+\left(1+q^{-2}\right)s_{1}+-q^2s_{1}s_{3}^2 + q^{-1}s_{2}s_{12} \\(q^2+1)
 &+  s_{4}s_{1234}s_{23} + q^{-2}s_{2}s_{1234}s_{34} + q^{-1}s_{4}s_{14} + \left(-q^3-q+q^{-1}\right)s_{3}s_{13} \\
 &+  s_{3}s_{1234}s_{24} + \left(-q^2+1\right)s_{2}s_{3}s_{123} + q^{-1}s_{1234}s_{234} + \left(q^{-2}-1\right)s_{3}s_{4}s_{134} \\
 &+  s_{2}s_{4}s_{124} + q^2s_{3}s_{12}s_{23} + qs_{23}s_{123} +  s_{3}s_{34}s_{14} + q^{-3}s_{34}s_{134} + qs_{24}s_{124}\\
 s_{12}s_{23}s_{134} &= \left(1+q^{-2}\right)s_{4}- s_{2}^2s_{4}+q^{-1}s_{1}s_{2}s_{3}s_{1234} +  s_{3}s_{1234}s_{12} \\
 &+ q^{-2}s_{1}s_{1234}s_{23} + q^{-1}s_{3}s_{34} + q^{-1}s_{1}s_{14} +  s_{2}s_{1234}s_{13} -qs_{2}s_{24} \\
 &+ q^{-1}s_{1234}s_{123} + \left(q^{-2}-1\right)s_{2}s_{3}s_{234} +  s_{1}s_{3}s_{134} +  s_{2}s_{12}s_{14} + qs_{12}s_{124} \\
 &+  s_{2}s_{23}s_{34} + q^{-3}s_{23}s_{234} + qs_{13}s_{134}
\end{align*}

\subsection{Quartic Relation}
\begin{align*}
    s_{12}&s_{23}s_{34}s_{14} \\
    &= s_{1}s_{3}s_{12}s_{23} + s_{2}s_{4}s_{12}s_{14} + s_{2}s_{4}s_{23}s_{34} + s_{1}s_{3}s_{34}s_{14} \\
    &+ \left(qs_{3}s_{12} + qs_{2}s_{13} + q^{-1}s_{1}s_{23} + s_{123} + q^{-1}s_{4}s_{1234}+s_{1}s_{2}s_{3}\right)s_{123} \\
    &+ \left(qs_{4}s_{13} + qs_{3}s_{14} + q^{-1}s_{1}s_{34} + s_{134} + q^{-1}s_{2}s_{1234}+s_{1}s_{3}s_{4}\right)s_{134} \\
    &+ \left( qs_{4}s_{12} + qs_{2}s_{14} + q^{-1}s_{1}s_{24}  + s_{124} + q^{-1}s_{3}s_{1234}+s_{1}s_{2}s_{4}\right)s_{124} \\
    &+ q^{-2}\left(q^{-1}s_{4}s_{23} + q^{-1}s_{3}s_{24} + q^{-1}s_{2}s_{34}  + q^{-2}s_{234} + q^{-1}s_{1}s_{1234} + (2-q^2)s_{2}s_{3}s_{4}\right)s_{234} \\
    &+ \left(q^2s_{12} + \left(q+q^{-1}\right)s_{1}s_{2} +s_{3}s_{4}s_{1234}\right)s_{12} + \left(q^{-2}s_{23}+q^{-1}s_{2}s_{3}+q^{-2}s_{1}s_{4}s_{1234}\right)s_{23} \\
    &+ \left(q^2s_{14} + \left(q+q^{-1}\right)s_{1}s_{4} +s_{2}s_{3}s_{1234}\right)s_{14} + \left(q^{-2}s_{34} + q^{-1}s_{3}s_{4}+q^{-2}s_{1}s_{2}s_{1234}\right)s_{34} \\
    &+ \left(q^{-2}s_{24} +\left(q^{-1}-q\right)s_{2}s_{4}+q^{-2}s_{1}s_{3}s_{1234}\right)s_{24} +\left(q^2s_{13} + s_{2}s_{4}s_{1234}\right)s_{13} \\
    &+q^{-2}s_{1234}^2 -s_{2}^2s_{4}^2 -s_{1}^2s_{3}^2   +q^{-1}s_{1}s_{2}s_{3}s_{4}s_{1234} +\left(2+q^{-2}\right)\left(s_{1}^2 + s_{2}^2 + s_{3}^2 + s_{4}^2\right)\\
    &-2q^2-5-4q^{-2}-q^{-4}\\
\end{align*}

\subsection{Loop Triple Relations}
\begin{align*}
    s_{123}s_{134} &=  -s_{2}s_{4} + s_{1234}s_{13} -\left(q + q^{-1}\right)s_{24} -s_{3}s_{234} -s_{1}s_{124} + s_{12}s_{14} + s_{23}s_{34}\\
s_{134}s_{123} &=  \left(1 - q^2 - q^{-2}\right)s_{2}s_{4} + s_{1234}s_{13} -\left(q^3 - q^{-3}\right)s_{24} -q^2s_{3}s_{234} -q^{-2}s_{1}s_{124} \\
&+ q^{-2}s_{12}s_{14} + q^2s_{23}s_{34}\\
s_{234}s_{124} &=  -q^2s_{1}s_{3} - \left(q^3+q\right)s_{13} + s_{1234}s_{24} -q^2s_{2}s_{123} -s_{4}s_{134} + q^2s_{12}s_{23} + s_{34}s_{14}\\
s_{124}s_{234} &=  -q^2s_{1}s_{3} - \left(q^3+q\right)s_{13} + s_{1234}s_{24} -s_{2}s_{123} -q^2s_{4}s_{134} + s_{12}s_{23} + q^2s_{34}s_{14}\\
\end{align*}

\subsection{Link Triple Relations}
\begin{align*}
s_{123}s_{234} &=  q^{-1}s_{2}s_{3}s_{1234}+s_{1}s_{4} + \left(q^{-1} - q\right)s_{2}s_{4}s_{12} + s_{1234}s_{23} + \left(q + q^{-1}\right)s_{14} \\
&+ s_{3}s_{134} + q^{-2}s_{2}s_{124} -q^2s_{12}s_{24} -qs_{3}s_{12}s_{234} -s_{34}s_{13} -q^{-1}s_{2}s_{34}s_{123}\\ 
&+ qs_{12}s_{23}s_{34}\\
s_{234}s_{123} &=  qs_{2}s_{3}s_{1234}+s_{1}s_{4} + \left(-q^3+q\right)s_{2}s_{4}s_{12} + s_{1234}s_{23} + \left(q + q^{-1}\right)s_{14} \\
&+ q^2s_{3}s_{134} + s_{2}s_{124} -q^4s_{12}s_{24} -q^3s_{3}s_{12}s_{234} -q^2s_{34}s_{13} -qs_{2}s_{34}s_{123}\\ 
&+ q^3s_{12}s_{23}s_{34}\\
s_{124}s_{134} &=  s_{2}s_{3}+qs_{1}s_{4}s_{1234} + \left(q + q^{-1}\right)s_{23} + s_{1234}s_{14} + q^2s_{1}s_{123} \\
&+ s_{4}s_{234} -q^2s_{12}s_{13} -qs_{4}s_{12}s_{134} -s_{34}s_{24} -qs_{1}s_{34}s_{124} \\
&+ qs_{12}s_{34}s_{14}\\
s_{134}s_{124} &=  s_{2}s_{3}+q^{-1}s_{1}s_{4}s_{1234} + \left(q + q^{-1}\right)s_{23} + s_{1234}s_{14} + s_{1}s_{123} \\
&+ q^{-2}s_{4}s_{234} -s_{12}s_{13} -q^{-1}s_{4}s_{12}s_{134} -q^{-2}s_{34}s_{24} -q^{-1}s_{1}s_{34}s_{124}\\ 
&+ q^{-1}s_{12}s_{34}s_{14}\\
s_{134}s_{234} &=  qs_{3}s_{4}s_{1234}+s_{1}s_{2} + \left(q + q^{-1}\right)s_{12} + \left(-q^3+q\right)s_{1}s_{3}s_{23} + s_{1234}s_{34} \\
&+ s_{3}s_{123} + q^2s_{4}s_{124} -q^4s_{23}s_{13} -q^3s_{4}s_{23}s_{134} -q^2s_{14}s_{24} -qs_{3}s_{14}s_{234} \\
&+ q^3s_{23}s_{34}s_{14}\\
s_{234}s_{134} &=  q^{-1}s_{3}s_{4}s_{1234}+s_{1}s_{2} + \left(q + q^{-1}\right)s_{12} + \left(q^{-1} - q\right)s_{1}s_{3}s_{23} + s_{1234}s_{34}\\
&+ q^{-2}s_{3}s_{123} + s_{4}s_{124} -q^2s_{23}s_{13} -qs_{4}s_{23}s_{134} -s_{14}s_{24} -q^{-1}s_{3}s_{14}s_{234}\\ 
&+ qs_{23}s_{34}s_{14}\\
s_{124}s_{123} &=  \left(q^2 - q^{-2} + q^{-4}\right)s_{3}s_{4}+\left(1 - q^{-1} + q^{-3}\right)s_{1}s_{2}s_{1234} + s_{1234}s_{12} \\
&+ \left(q-q^{-1}\right)s_{2}s_{4}s_{23} + \left(q^3+q-q^{-1}+q^{-5}\right)s_{34} + \left(q^{-3}-q^{-1}\right)s_{1}s_{3}s_{14} \\
&+ \left(q^2 - q^{-2} + q^{-4}\right)s_{2}s_{234} + \left(q^2 - 1 + q^{-4}\right)s_{1}s_{134} -q^{-4}s_{23}s_{24} \\
&-q^{-3}s_{1}s_{23}s_{124} -q^2s_{14}s_{13} -qs_{2}s_{14}s_{123} + q^{-1}s_{12}s_{23}s_{14}\\
s_{123}s_{124} &=  \left(q^4-q^2+q^{-2}\right)s_{3}s_{4}+\left(q^3-q+q^{-1}\right)s_{1}s_{2}s_{1234} + s_{1234}s_{12} \\
&+ \left(q^3-q\right)s_{2}s_{4}s_{23} + \left(q^5-q+q^{-1}+q^{-3}\right)s_{34} + \left(q^{-1} - q\right)s_{1}s_{3}s_{14} \\
&+ \left(q^4-1+q^{-2}\right)s_{2}s_{234} + \left(q^4-q^2+q^{-2}\right)s_{1}s_{134} -q^{-2}s_{23}s_{24} \\
&-q^{-1}s_{1}s_{23}s_{124} -q^4s_{14}s_{13} -q^3s_{2}s_{14}s_{123} + qs_{12}s_{23}s_{14}\\
\end{align*}

\subsection{Crossing Relations}
\begin{align*}
    s_{13}s_{24} &=  \left(q + q^{-1}\right)s_{1234}+s_{1}s_{2}s_{3}s_{4} + q^{-1}s_{3}s_{4}s_{12} + qs_{1}s_{4}s_{23} + q^{-1}s_{1}s_{2}s_{34} \\
&+ qs_{2}s_{3}s_{14} + s_{4}s_{123} + s_{1}s_{234} + s_{2}s_{134} + s_{3}s_{124} + q^{-2}s_{12}s_{34} + q^2s_{23}s_{14}\\
s_{24}s_{13} &=  \left(q + q^{-1}\right)s_{1234}+s_{1}s_{2}s_{3}s_{4} + qs_{3}s_{4}s_{12} + q^{-1}s_{1}s_{4}s_{23} + qs_{1}s_{2}s_{34} \\
&+ q^{-1}s_{2}s_{3}s_{14} + s_{4}s_{123} + s_{1}s_{234} + s_{2}s_{134} + s_{3}s_{124} + q^2s_{12}s_{34} + q^{-2}s_{23}s_{14}\\
\end{align*}

\subsection{Double and Triple Crossing Relations}
\begin{align*}
    s_{123}s_{24} &=  -s_{2}s_{1234}-qs_{1}s_{3}s_{4} -s_{1}s_{34} -s_{3}s_{14} -q^2s_{4}s_{13} -qs_{2}s_{4}s_{123} \\
&-\left(q + q^{-1}\right)s_{134} + qs_{4}s_{12}s_{23} + s_{12}s_{234} + s_{23}s_{124}\\
s_{234}s_{13} &=  -s_{3}s_{1234}-qs_{1}s_{2}s_{4} -s_{4}s_{12} -s_{2}s_{14} -q^2s_{1}s_{24} -qs_{1}s_{3}s_{234} \\
&-\left(q + q^{-1}\right)s_{124} + qs_{1}s_{23}s_{34} + s_{23}s_{134} + s_{34}s_{123}\\
s_{13}s_{124} &=  -q^{-1}s_{2}s_{3}s_{4}+\left(1 - q^2 - q^{-2}\right)s_{1}s_{1234} -q^2s_{4}s_{23} -q^{-2}s_{2}s_{34} -q^{-2}s_{3}s_{24} \\
&-\left(q^3 - q^{-3}\right)s_{234} -q^{-1}s_{1}s_{3}s_{124} + q^{-1}s_{3}s_{12}s_{14} + q^{-2}s_{12}s_{134} + q^2s_{14}s_{123}\\
s_{24}s_{134} &=  \left(1 - q^2 - q^{-2}\right)s_{4}s_{1234}-qs_{1}s_{2}s_{3} -q^2s_{3}s_{12} -q^{-2}s_{1}s_{23} -q^2s_{2}s_{13} \\
&-\left(q^3 - q^{-3}\right)s_{123} -qs_{2}s_{4}s_{134} + qs_{2}s_{34}s_{14} + q^2s_{34}s_{124} + q^{-2}s_{14}s_{234}\\
s_{24}s_{123} &=  \left(1 - q^2 - q^{-2}\right)s_{2}s_{1234}-qs_{1}s_{3}s_{4} -q^2s_{1}s_{34} -q^{-2}s_{3}s_{14} -q^2s_{4}s_{13} \\
&-qs_{2}s_{4}s_{123} -\left(q^3 - q^{-3}\right)s_{134} + qs_{4}s_{12}s_{23} + q^2s_{12}s_{234} + q^{-2}s_{23}s_{124}\\
s_{13}s_{234} &=  \left(1 - q^2 - q^{-2}\right)s_{3}s_{1234}-qs_{1}s_{2}s_{4} -q^{-2}s_{4}s_{12} -q^2s_{2}s_{14} -q^2s_{1}s_{24} \\
&-qs_{1}s_{3}s_{234} -\left(q^3 - q^{-3}\right)s_{124} + qs_{1}s_{23}s_{34} + q^2s_{23}s_{134} + q^{-2}s_{34}s_{123}\\
s_{134}s_{24} &=  -s_{4}s_{1234}-qs_{1}s_{2}s_{3} -s_{3}s_{12} -s_{1}s_{23} -q^2s_{2}s_{13} -\left(q + q^{-1}\right)s_{123} \\
&-qs_{2}s_{4}s_{134} + qs_{2}s_{34}s_{14} + s_{34}s_{124} + s_{14}s_{234}\\
s_{124}s_{13} &=  -q^{-1}s_{2}s_{3}s_{4}-s_{1}s_{1234} -s_{4}s_{23} -s_{2}s_{34} -q^{-2}s_{3}s_{24} -\left(q + q^{-1}\right)s_{234} \\
&-q^{-1}s_{1}s_{3}s_{124} + q^{-1}s_{3}s_{12}s_{14} + s_{12}s_{134} + s_{14}s_{123}\\
\end{align*}

\end{appendices}


\printbibliography[title={References}]

\end{document}